\documentclass[12pt,a4paper,leqno,oneside]{amsart}
\usepackage[utf8]{inputenc}
\usepackage[a4paper,hmargin={3cm,3cm},vmargin={3cm,4cm}]{geometry}
\usepackage{hyperref}
\usepackage{amsmath,amssymb,amsthm}
\usepackage[foot]{amsaddr}
\usepackage{graphicx}

\newtheorem{theorem}{Theorem}[section]

\newtheorem{definition}[theorem]{Definition}
\newtheorem{proposition}[theorem]{Proposition}
\newtheorem{lemma}[theorem]{Lemma}

\numberwithin{equation}{section}

\theoremstyle{remark}
\newtheorem{remark}[theorem]{Remark}

\newcommand{\inter}{\mathop{\rm int}\nolimits}

\newcommand{\txconst}{{\alpha_1}}
\newcommand{\good}{{\txconst}}
\newcommand{\gap}{{\alpha_2}}

\newcommand{\lap}{\Delta}
\newcommand{\Ber}{{\mathbb Q}^\star}
\newcommand{\EBer}{{\mathbb E}^\star}
\newcommand{\Plr}{\mathsf{P}}
\newcommand{\Elr}{\mathsf{E}}
\newcommand{\Olr}{\Omega}
\newcommand{\piece}{\gamma}
\newcommand{\vv}{v}

\def\comp{{\mathcal C}}
\def\Max{{\mathrm{max}}}
\def\Sec{{\mathrm{sec}}}

\def\supp{\mathop{\rm supp}\nolimits}
\def\dist{\mathop{\rm dist}\nolimits}

\def\diam{\mathop{\rm diam}\nolimits}
\def\ld{\mathop{\rm ld}\nolimits}
\def\tx{\mathop{\tt tx}\nolimits}
\def\Ran{\mathop{\rm Ran}\nolimits}
\def\d{\mathrm{d}}
\def\bbone{\boldsymbol 1 }

\def\bV{\bar{\mathcal V}}
\def\bC{\bar{\mathcal C}}
\def\EO{\mathsf{EO}}
\def\EV{\mathsf{EV}}
\def\EE{\mathsf{E}}

\def\IQ{\mathsf{IQ}}
\def\FF{\mathsf{F}}
\def\DD{\mathsf{T}}
\def\Q{\mathcal F}

\begin{document}
  \title[Giant component left by random walk]{Giant vacant
    component left by a random walk in a random $d$-regular graph}

  \author[J. Černý]{Jiří Černý$^1$}
  \author[A. Teixeira]{Augusto Teixeira$^1$}
  \address{\itshape $^1$ ETH Zurich, Department of Mathematics,
    R\"amistrasse 101, 8092 Zurich, Switzerland}
  \author[D. Windisch]{David Windisch$^2$}
  \address{\itshape $^2$ The Weizmann Institute of Science,
    Faculty of Mathematics and Computer Science, POB 26, Rehovot 76100,
    Israel}
  \thanks{The research of A. Teixeira and D. Windisch was partially supported by
    Swiss National Science Foundation Grant PDFM22-120708/1}
  \date{\today}

\begin{abstract}
  We study the trajectory of a simple random walk on a $d$-regular graph
  with $d \geq 3$ and locally tree-like structure as the number $n$ of
  vertices grows. Examples of such graphs include random $d$-regular
  graphs and large girth expanders. For these graphs, we investigate
  percolative properties of the set of vertices not visited by the walk
  until time $un$, where $u>0$ is a fixed positive parameter. We show that
  this so-called \textit{vacant} set exhibits a phase transition in $u$ in
  the following sense: there exists an explicitly computable  threshold
  $u_\star \in (0, \infty)$ such that, with high probability as $n$ grows, if
  $u<u_\star$, then the largest
  component of the vacant set has a volume of order $n$,  and if
  $u>u_\star$, then it has a volume of order $\log n$. The critical
  value $u_\star$ coincides with the critical intensity of a random
  interlacement process on a $d$-regular tree. We also show that the
  random interlacements model describes the structure of the vacant set in
  local neighbourhoods.
\end{abstract}

\maketitle

\section{Introduction and main results}

In this work we consider the simple random walk on a graph $G$ chosen among a
certain class of finite regular graphs including, for example, typical
realizations of random regular graphs, or expanders with large girth. The main
object of our study is the complement of the trajectory of the random walk
stopped at a time $u|G|$, for $u>0$, the so-called vacant set, and its
percolative properties.

We show that the vacant set  undergoes the following phase
transition in $u$:  as long as $u < u_\star$, the
vacant set has a unique component with volume of order $|G|$, whereas if
$u > u_\star$, the largest component of the vacant set only has a volume of
order $\log |G|$, with high probability as the size of $G$ diverges.
More importantly, we show that the above phase transition corresponds to
the phase transition in a random interlacement model on a regular tree.
In particular, the critical value $u_\star$ is the same for both models.

The random interlacement (on $\mathbb Z^d$, $d \geq 3$) was recently introduced
by Sznitman \cite{Szn09} to provide a model describing the microscopic
structure of the bulk when considering the large $N$ asymptotics of the
disconnection time of the discrete cylinder
$(\mathbb Z/N\mathbb Z)^{d-1}\times \mathbb Z$ \cite{DS06}, or percolative
properties of the vacant set left by the simple random walk on the torus
$(\mathbb Z/N\mathbb Z)^{d}$ \cite{BS08}. Later, in \cite{Szn09b,Win08}, it was
proved that the random interlacements indeed \emph{locally} describe this
microscopic structure. In \cite{Szn09,SS09} it was shown that the random
interlacement undergoes a phase transition for a non-trivial value
$u_\star(\mathbb Z^d)$ of the parameter $u$ driving its intensity.
The best bounds on the disconnection time known from
\cite{Szn09d} and \cite{Szn09e} involve parameters derived from random
interlacements and the upcoming work \cite{TW10} connects distinct
regimes for random interlacements with distinct regimes for the vacant
set left by random walk on the torus. It can currently not be proved
that the critical value $u_\star(\mathbb Z^d)$ for random
interlacements is itself connected with a critical value for the
vacant set on the torus or for the disconnection time.

We provide such a connection in our model. This is possible for the following
reasons: For the considered graphs $G$,  large neighbourhoods of  typical
vertices of $G$ are isomorphic to a ball in a regular tree, and, as we will
show, the corresponding local microscopic model is the random interlacement on
such a tree. Connected components of this interlacement model admit a
particularly simple description in terms of a branching process, and its
critical value $u_\star$ is explicitly computable (see \eqref{e:ustar}),
giving us a good local control of configurations of the vacant set on $G$.
Good expansion properties of $G$
then allow us to extend the local control to a global one.

\medskip

We now come to the precise statements of our results. We consider a sequence of
finite connected graphs $G_k=(V_k, \mathcal E_k)$ such that the number $n_k$ of
vertices in $V_k$ tends to infinity as $k \to \infty$. We are principally
interested in the case where $G_k$ is a sequence of $d$-regular random graphs,
$d\ge 3$, or of $d$-regular expanders with large girth (such as, for example,
  Lubotzky-Phillips-Sarnak graphs \cite{LPS88}). As we shall show below, these
two classes of graphs satisfy the following assumptions, which are the only
assumptions we need in order to prove our main theorems. We assume that for
some $d \geq 3$, $\txconst\in(0,1)$, and all~$k$,
\begin{equation*}
  \tag{A0}
  \label{a:reg}
  \text{$G_k$ is $d$-regular, that is all
    its vertices have degree $d$, and}\\
\end{equation*}
\begin{equation*}
  \tag{A1}
  \label{a:tx1}
  \parbox{.8\textwidth}{for any $x \in V_k$, there is at most one cycle
    contained in the ball with radius $\txconst \log_{d-1} n_k$ centered at $x$.}
\end{equation*}
We also assume that the
spectral gap $\lambda_{G_k}$ of $G_k$ (we recall the definition in
  \eqref{def:gap} below) is uniformly bounded from below by a constant
  $\gap>0$, that is
\begin{equation*}
  \tag{A2}
  \label{a:gap}
  \lambda_{G_k} > \gap > 0, \text{ for all $k \geq 1$}.
\end{equation*}
Under \eqref{a:reg}, this final assumption is equivalent to assuming that $G_k$
are expanders, see \eqref{eq:cheeger}. Note that in general \eqref{a:tx1} does
not imply \eqref{a:gap}, see Remark~\ref{r:nonequivalence}.

We consider a continuous-time random walk on $G_k$. More precisely, we write $P$
for the  canonical law on the space $D([0,\infty),V_k)$ of cadlag functions
from $[0, \infty)$ to $V_k$ of the continuous-time simple random walk on $G_k$
with i.i.d.~mean-one exponentially distributed waiting times and uniformly
distributed starting point. We use $(X_t)_{t \geq 0}$ to denote the canonical
coordinate process. For a fixed parameter $u\geq 0$ not depending on $k$, we
define the \textit{vacant set}  as the set of all vertices not visited by the
random walk until time $un_k$:
\begin{equation}
\label{e:Vuk}
  \mathcal V_k^u = \{x\in V_k: x \neq X_t,
  \textrm{ for all } 0 \leq t \leq un_k \}.
\end{equation}
We use $\comp ^u_\Max  \subset V$ to denote the largest connected component
of $\mathcal V_k^u$.

The following theorems are the main results of the
present paper.
The critical parameter $u_\star$ in the statements coincides with the
critical parameter for random interlacements on the infinite $d$-regular
tree $\mathbb T_d$, which, according to \cite{Tei09}, equals
\begin{equation}
  \label{e:ustar}
  u_\star= \frac{d(d-1) \ln(d-1)}{(d-2)^2}.
\end{equation}

\begin{theorem}[subcritical phase]
  \label{t:subcritical}
  Assume \eqref{a:reg}--\eqref{a:gap}, and fix $u>u_\star.$
  Then for every $\sigma >0 $  there exist constants
  $K(d,\sigma,u,\txconst, \gap), C(d,\sigma,u,\txconst, \gap)<\infty$, such that
  \begin{equation}
    P[|\comp^u_\Max|\ge K\ln n_k]\le C n_k^{-\sigma },
    \textrm{ for all } k\ge 1.
  \end{equation}
\end{theorem}

\begin{theorem}[supercritical phase]
  \label{th:usmall}
  Assume \eqref{a:reg}--\eqref{a:gap}, and fix $u<u_\star$.
  Then for every $\sigma >0$ there exist constants
  $\rho(d,\sigma,u,\txconst, \gap) \in (0,1)$ and
  $C(d,\sigma,u,\txconst, \gap)<\infty$, such that
  \begin{equation}
    P\bigl[ |\comp^u_\Max| \geq \rho n_k \bigr] \geq 1 - C n_k^{-\sigma},
    \textrm{ for all } k\ge 1.
  \end{equation}
\end{theorem}
For the statement on the uniqueness of the giant component we denote the second
largest component of $\mathcal V_k^u$ by $\comp^u_\Sec$.
\begin{theorem}[supercritical phase--uniqueness]
  \label{t:secondcomponent}
  Assume \eqref{a:reg}--\eqref{a:gap}, and fix $u<u_\star.$
  Then for every $\kappa >0$,
  \begin{equation}
    \lim_{k \to \infty} P\bigl[ |\comp^u_\Sec| \geq \kappa  n_k \bigr] =0.
  \end{equation}
\end{theorem}

From the last theorem it follows that there exists a function $f$ satisfying
$f(n)=o(n)$ and $P[|\comp^u_\Sec| \leq f(n_k)]\to 1$. More information on
the asymptotics of  $f(n)$ could be obtained from our techniques. However, they
are not sufficient to prove $f=O(\log n)$, which is the conjectured size of
$\comp^u_\Sec$, based on the behaviour of Bernoulli percolation.

Let us now comment on related results. The size of the vacant components left
by a random walk on a finite graph has so far only been studied by Benjamini
and Sznitman in \cite{BS08} for $G_k$ given by a $d$-dimensional integer torus
with large side length $k$ and sufficiently large dimension $d$. In this case,
the authors prove that the vacant set has a suitably defined unique giant
component occupying a non-degenerate fraction of the total volume with
overwhelming probability, provided $u>0$ is chosen sufficiently small. Their
work does not prove anything, however, for the large $u$ regime, let alone any
results on a phase transition in $u$. Our results are the first ones to
establish such a phase transition for a random walk on a finite graph.
Moreover, our results provide some indication that a phase transition occurs
for random walk on the torus as well, and that the critical parameter
$u_\star({\mathbb Z}^d)$ for random interlacements on ${\mathbb Z}^d$ should
play a key role.

A similar phase transition was proved for Bernoulli percolation on various
graphs: first by Erd\H os and R\'enyi \cite{ER60} on the complete graph, and
more recently  on large-girth expanders in \cite{ABS04}, as well as on many
other graphs satisfying a so-called triangle condition \cite{BCvdHSS05a}. For
our results the paper \cite{ABS04} is the most relevant, some of our proofs
build on techniques introduced there. A very precise description of the
Bernoulli percolation on random regular graphs  was recently  obtained  in
\cite{NP07,Pit08}.

\medskip

Let us now comment on the proofs of our results. For most of the arguments, we
do not work with the law $P$ of the random walk, but with a different
measure $Q$ on $D([0,\infty),V_k)$. The trajectory of the canonical
process $X$ under $Q$  is constructed from an  i.i.d.~sequence
$(Y^i)_{i\in \mathbb N}$, of uniformly-started random walk trajectories of
length $L=n_k^\gamma$, for $\gamma <1$, called \emph{segments}. To create a
nearest-neighbour path, the endpoint of segment $Y^i$ and the starting point of
segment $Y^{i+1}$ are connected using a \emph{bridge} $Z^i$, $i\in \mathbb N$,
which is a random walk bridge of length $\ell=\log^2 n_k$. Since $\ell$ is much
larger than the mixing time of the random walk on $G_k$,  $Q$ provides
a very good approximation of $P$, see Lemma~\ref{l:partsbridges}.

The set $\bV^u_k=V_k\setminus \cup_{i<\lfloor un_k/(L+\ell) \rfloor} \Ran Y^i$,
the so-called \emph{vacant set left by segments}, plays a particular role in
our proofs. It is a complement of `a cloud of independent random walk
trajectories', similar to the vacant set of a random interlacement. Observe
that $\bV^u_k$ is an enlargement of $\mathcal V^u_k$.

To prove Theorem~\ref{t:subcritical}, we analyse a breadth-first search
algorithm exploring one component of the set  $\bV^u_k$. We show that this
algorithm is likely to terminate in no more than $K \ln n_k$ steps. To prove
this, we need to control the probability that a (not yet explored) vertex $y$
is found to be vacant at a particular step of the algorithm. The main
difficulty is that, unlike in Bernoulli site percolation models, this event is
not independent of the past of the algorithm. We will derive an estimate of the
form (see Proposition~\ref{p:c})
\begin{equation}
  \label{eq:approximate}
  \Plr [y\notin \Ran Y^i| A\cap \Ran Y^i=\varnothing] \sim
  f(d,u)^{(L+\ell)/(un)},
\end{equation}
where $A$ will be the part of ${\bar {\mathcal V}}^u_k$ already explored by the
algorithm.  The explicitly computable quantity $f(d,u)$ appears in the study of
random interlacement on the infinite tree ${\mathbb T}_d$ at level $u$
\cite{Tei09}, and it equals the probability that a given vertex $z$ (different
  from the root of the tree) is vacant, given its parent in the tree is vacant.

The estimate \eqref{eq:approximate} will imply that the probability of $y$
being vacant given the past of the algorithm is well approximated by $f(d,u)$.
Since for $u>u_\star$ we have $f(d,u) < 1/(d-1)$,  the considered breadth-first
search algorithm can be controlled by a sub-critical branching process,
yielding Theorem~\ref{t:subcritical}.

There is an additional difficulty coming from the fact that the estimate
\eqref{eq:approximate} holds only under suitable restrictions on the set $A$
and the vertex $y$ (see \eqref{e:propercone}). These restrictions are however
always satisfied for a large majority of the steps of the algorithm, as we will
show in Proposition~\ref{p:mrproper}.

\medskip

We now comment on the proof of Theorem~\ref{th:usmall}. This proof consists of
the following two steps: first we show that for some slightly larger parameter
$u_k \in (u,u_\star)$, there are many components of $\mathcal{V}^{u_k}_k$
having volume at least $n_k^\delta$, for some $\delta>0$. Then, we use a
sprinkling technique, based on the following heuristic idea: we reduce $u_k$ to
$u$ and prove that with high probability, the mentioned components merge into a
cluster of size at least $\rho n_k$, cf.~\cite{ABS04}.

For the first step, we use the fact that $\bar{\mathcal{V}}^{u_k}_k$ can be
locally compared with the vacant set of  random interlacements on a $d$-regular
tree. This is proved in Proposition~\ref{pr:dominat}, which again uses an
approximation of type \eqref{eq:approximate}, see \eqref{e:perc}.  Since
$u_k < u_\star$,  the random interlacement at level $u_k$ is  super-critical,
yielding the existence of components  of volume $n_k^\delta$ in $G_k$.
Lemma~\ref{lm:lotstrings} then implies that going from $\bV^{u_k}_k$ to
$\mathcal V^{u_k}_k$ (by inserting the bridges $Z^i$) does not destroy these
components.

Regarding the second step,  it is by no means obvious how to  perform a
sprinkling as mentioned above. Indeed, a simple deletion of the last part
$X_{[un_k, u_k n_k]}$ of the trajectory would require us to deal with the
distribution of the set $X_{[un_k,u_k n_k]}$ given $\mathcal V^{u_k}_k$, which
seems difficult. Instead, we perform the sprinkling in the manner natural for
random interlacements (cf.~\cite{Szn09c}): we  remove some segments $Y^i$
independently at random.

The deletion of segments, however, disconnects the trajectory of the process.
We bypass this problem by adding extra bridges before the sprinkling (cf.
  \eqref{eq:ZgivenY} and Lemma~\ref{lm:lotstrings} again), so that even after
the deletion of some segments, we can extract a nearest-neighbour trajectory of
length at least $un$, with high probability.

We then use the expansion properties (cf.~\eqref{eq:cheeger}) of our graph to
show that the sprinkling construction merges some of the clusters of size
$n^\delta_k$ into a giant component of size at least $\rho n_k$.

\medskip

The proof of the uniqueness, that is of Theorem~\ref{t:secondcomponent}, again
combines sprinkling with the local comparison with random interlacements. Using
this comparison and the branching process approximation of the random
interlacement on the tree, we will show that at level $u_k$ there are, with a
high probability, only $o(n_k)$ vertices contained in vacant clusters of size
between $\ln^2 n_k$ and $n_k^c$, for some $c\in (0,1)$, see Lemma~\ref{l:bad}
for the exact formulation. This statement is a weaker version of the so-called
`absence of components of intermediate size' which is usually proved for
Bernoulli percolation.

This will allow us to show that any component of $\mathcal V_k^u$ of size at
least $\kappa n_k$, should contain at least $\kappa n_k/2$ vertices $x$ being
in vacant components of size at least $n_k^c$ at level $u_k$. The sprinkling
then shows that any two groups of size $\kappa n_k/2$ of such vertices are
connected in $\mathcal V_k^u$, excluding two giant components with a high
probability.

\medskip

We close this introduction with two remarks concerning our assumptions.
\begin{remark}
  The assumptions \eqref{a:reg}--\eqref{a:gap} are designed in order to include
  two classes of $d$-regular graphs:  expanders with girth larger than
  $c \log |V_k|$, and typical realizations of a random $d$-regular graph. In the
  random $d$-regular graph case these assumptions also help us separate the
  randomness of the graph from the randomness of the walk.

  The fact that the typical realization of the random $d$-regular graph
  satisfies assumption \eqref{a:tx1} follows from Lemma~2.1 of
  \cite{LS09}, where they show it for $\good=\tfrac 15$. To see that
  \eqref{a:gap} holds one can use the estimate on the second  eigenvalue
  of the adjacency matrix $A$ of the random $d$-regular graph of Friedman
  \cite{Fri08} (or older results, e.g. \cite{BS87,Fri91}, which however
    only provide estimates for $d$ even and not too small). Indeed, in
  \cite{Fri08} it is shown that this second eigenvalue is
  $2\sqrt{d-1}+o(1)$, with a high probability. The largest eigenvalue of
  this matrix is $d$. This implies \eqref{a:gap}, since the generator of
  the random walk is given by $\tfrac A d - \text{Id}$.
\end{remark}
\begin{remark}
  \label{r:nonequivalence}
  The assumption \eqref{a:tx1} does not imply \eqref{a:gap}. This can be seen
  easily by considering two copies $G$, $G'$ of a large girth expander with $n$
  vertices, choosing two edges, $e=\{x,y\}$ of $G$ and $e'=\{x',y'\}$ of $G'$,
  erasing $e,e'$ and joining $G$, $G'$ with two new edges $\{x,x'\}$, and
  $\{y,y'\}$. The new graph is $d$-regular. It satisfies \eqref{a:tx1},
  potentially with a slightly different constant than $G$. However, the new
  edges create a bottleneck for the random walk, implying that the spectral gap
  of the new graph decreases to zero with the number of vertices $n$.
\end{remark}

The paper is organised as follows. In Section~\ref{s:notation}, we set up the
notation. In Section~\ref{s:CP}, we prove an
estimate of the form \eqref{eq:approximate}. The piecewise independent measure
$Q$ is constructed in Section~\ref{s:partsbridges}.
Sections~\ref{s:sub}, \ref{s:super} and \ref{s:uniq} contain the proofs of
Theorems~\ref{t:subcritical}, \ref{th:usmall} and \ref{t:secondcomponent},
respectively.

\medskip

\textbf{Acknowledgements.} The authors wish to thank Itai Benjamini
for proposing the study of the vacant set on expanders and are
indebted to Alain-Sol Sznitman for helpful discussions. This work was
started when the third author was at ETH Zurich.

\section{Notation}
\label{s:notation}
In this section we introduce additional notation and recall some known
results about random interlacements.

\subsection{Basic notations.}
Throughout the text $c$ or $c'$ denote strictly positive constants only
depending on $d$, and the parameters $\good$ and $\gap$ in assumptions
\eqref{a:tx1} and \eqref{a:gap}, with value changing from place to place. The
numbered constants $c_0$, $c_1$, \ldots are fixed and refer to their first
appearance in the text. Dependence of constants on additional parameters
appears in the notation. For instance $c_\piece$ denotes a positive constant
depending on $\piece$ and possibly on $d$, $\good$, $\gap$.

We write $\mathbb N = \{0,1,\dots\}$ for the set of natural numbers and for
$a \in \mathbb{R}$ we write $\lfloor a \rfloor$ for the largest integer smaller
or equal to $a$ and define $\lceil a \rceil = \lfloor a \rfloor +1$. In this
paper we use $\ln x$ for the natural logarithm and use $\ld$ to denote the
logarithm with base $d-1$,
\begin{equation}
  \label{e:ld}
  \ld x = \log_{d-1} x = \ln x / \ln (d-1).
\end{equation}
For a set $A$ we denote by $|A|$ its cardinality.

Recall that we have introduced a sequence of finite connected graphs
$G_k =(V_k, {\mathcal E}_k)$ in the introduction.
We will always omit the subscript $k$ of the sequence of graphs $G_k$ and their
sizes $n_k$. In particular, we always assume that $n$ is the number of vertices
of $G$.
For $d$ as in \eqref{a:reg},
we will also consider the infinite $d$-regular tree, denoted
${\mathbb T}_d = ({\mathbb V}_d, {\mathbb E}_d)$.

We now introduce some notation valid for an arbitrary graph $G=(V, \mathcal E)$. We
use $\dist(\cdot,\cdot)$ to denote the usual graph distance and write $x\sim y$,
if $x,y$ are neighbours in $G$. We write $B(x,r)$ for the ball centred at
$x$ with radius $r$, $B(x,r)=\{y\in V:\dist(x,y)\le r\}$.
For $A\subset V$ we define its complement
$A^c=V\setminus A$, its $r$-neighbourhood $B(A,r)=\bigcup_{x\in A}B(x,r)$, and
its interior and exterior boundary
\begin{equation}
  \partial_i A = \{x\in A: \exists y\in A^c, x\sim y\},\qquad
  \partial_e A = \{x\in A^c: \exists y\in A, x\sim y\}.
\end{equation}
We write $\inter(A)$ for $A\setminus \partial_i A$. We
define the tree excess of a connected set $A\subset V$, denoted by $\tx(A)$, as
the number of edges which can be removed from the subgraph of $G$ induced by $A$
while keeping it connected. Equivalently,
\begin{equation}
  \label{e:txcharact}
  \tx(A)= |\mathcal E_A|-|A|+1
\end{equation}
where $\mathcal E_A$ stands for the edges of the subgraph induced by $A$. By a
\textit{cycle} we mean a sequence of vertices $x_1, \dots, x_k$ such that
$x_1 = x_k$ and $x_{i+1}\sim x_i$ for all $1\le i <k$. Note that $\tx(A) = 0$
if and only of there is no cycle in $A$.

\subsection{Random walk on graphs.}
\label{ss:rw}
We use $P_x$ to denote the law of the canonical continuous-time simple random walk
on $G$ started at $x\in V $, that is of the Markov process with generator given by
\begin{equation}
  \label{def:lap} \lap f (x) = \sum_{y \in V} (f(y)-f(x)) p_{xy},
  \qquad \text{for } f:V\to \mathbb R, x \in V,
\end{equation}
where $p_{xy}=1/{d_x}$ if $x\sim y$, and  $p_{xy}=0$ otherwise, and $d_x$
denotes the degree of the vertex $x$. We write $P^G_x$ for $P_x$ whenever
ambiguity would otherwise arise.

With exception of Lemma~\ref{pr:AF} and
Proposition~\ref{pr:EH}, we will always work with regular graphs $G$, in which
case $d_x$ is the same for every vertex $x \in V$. We use $X_t$ to denote the
canonical process and $(\mathcal F_t)_{t \geq 0}$ the canonical filtration. We
write $P^\ell_x$ for the restriction of $P_x$ to $D([0,\ell],V)$ and
$P^\ell_{xy}$ for the law of random walk bridge, that is for $P^\ell_x$
conditioned on $X_\ell=y$. We write $E_x, E^\ell_x, E^\ell_{xy}$ for the
corresponding expectations. The canonical shifts on $D([0,\infty),V)$ are
denoted by $\theta_t$. The time of the $n$-th jump is denoted by $\tau_n$, i.e.
$\tau_0=0$ and for $n \geq 1$,
$\tau_n = \inf\{ t \geq 0: X_t \neq X_0\} \circ \theta_{\tau_{n-1}} + \tau_{n-1}$.
The process counting the number of jumps before time $t$ is denoted by
$N_t=\sup\{k:\tau_k\le t\}$. Note that under $P_x$, $(N_t)_{t \geq 0}$ is a
Poisson process on ${\mathbb R}_+$ with intensity $1$, but this
is not true under $P^\ell_{xy}$.  We write $\hat X_n$ for the discrete skeleton of
the process $X_t$, that is $\hat X_n = X_{\tau_n}$. For $0 \leq s \leq t$, we
use $X_{[s,t]}$ to denote denote the set of vertices visited by the random walk
between times $s$ and $t$, $X_{[s,t]} = \{X_r: s \leq r \leq t\}$.

Given $A\subset V$, we denote with $H_A$ and $\tilde H_A$ the respective
entrance and hitting time of~$A$
\begin{equation}
  \label{e:hittingtimes}
  \begin{split}
    H_A &= \inf\{t \geq 0 : X_t \in A\},\qquad
    \text{and}\qquad
    {\tilde H}_A = H_A\circ \theta_{\tau_1} + \tau_1.
  \end{split}
\end{equation}
We write $\hat H_A$ for the discretised entrance time, $\hat H_A = N_{H_A}$.

For the remaining notation, we assume that $G$ is a finite connected graph. For
such $G$ we denote by $\pi$ the stationary distribution for the simple random
walk on $G$ and use $\pi_x$ for $\pi (x)$. $P$ stands for the law of the simple
random walk started at $\pi $ and $E$ for the corresponding expectation.
Under assumption \eqref{a:reg} the stationary
distribution is the uniform distribution. For all real valued functions $f,g$
on $V$ we define the Dirichlet form
\begin{equation}
  \label{def:dir}
  {\mathcal D}(f,g) =
  \frac{1}{2} \sum_{x,y \in V}
    (f(x)-f(y))(g(x)-g(y)) \pi_x p_{xy}
    = - \sum_{x \in G} \lap f(x) g(x) \pi_x.
\end{equation}
The spectral gap of $G$ is given by
\begin{align}
  \label{def:gap}
  \lambda_G = \min \bigl\{ {\mathcal D}(f,f): \pi(f^2)=1, \pi(f)=0 \bigr\}.
\end{align}
From \cite{SC97}, p.~328, it follows that under assumption \eqref{a:reg},
\begin{align}
  \label{eq:I}
  \sup_{x,y \in V} |P_x[X_t=y] - \pi_y| \leq e^{-\lambda_G t},
  \text{ for all } t \geq 0.
\end{align}
A function $h : V \to {\mathbb R}$ is called harmonic on $A$ if $\lap h(x) = 0$
for all $x \in A$. For two non-empty disjoint subsets $A,C$ of $V$ we define
the equilibrium potential $g^\star_{A,C}$ as the unique function harmonic on
$(A\cup C)^c$, satisfying $g^\star|_A=1$, $g^\star|_C=0$. It is well known that
\begin{gather}
  \label{e:potential}
  g^\star_{A,C}(x)=P_x[H_A \le H_C],\\
  \label{eq:Dgg}
  \mathcal{D}(g^\star_{A,C},g^\star_{A,C})
  = \sum_{z \in A} P_z[\tilde H_A > H_C] \pi_z.
\end{gather}

We define the isoperimetric constant of  $G$ as
$\iota_G = \min \{ |\partial_e A|/|A|: A \subset V,  |A| \leq |V|/2 \}$. If
assumption \eqref{a:reg} holds, then Cheeger's inequality (\cite[Lemma
    3.3.7]{SC97})  yields $ c \iota_G^2 \leq \lambda_G \leq c' \iota_G$. The
assumption \eqref{a:gap} then implies the existence of $\gap'>0$ such that
\begin{equation}
  \label{eq:cheeger}|
  \partial_e A| \ge \gap' |A|,\qquad \text{for all $k\ge 1$ and
    $A\subset V$ with $|A|\le |V|/2$}.
\end{equation}

\subsection{Random interlacement}
Let us give a brief introduction to random interlacements. Although we
will not directly use any results on random interlacements in this paper, random
interlacements give a natural interpretation to the key result in
Section~6. Consider an infinite locally finite graph
$\mathbb{G} = (\mathbb{V}, \mathbb{E})$ for which the simple random walk (with
  law denoted by $P_x^\mathbb{G}$) is transient. According to \cite{Szn09,
  Tei09}, the interlacement set on $\mathbb{G}$ is given by the trace left by a
Poisson point process of doubly infinite trajectories modulo time-shift in $\mathbb{G}$ which
visit every point only finitely many times. The complement of the interlacement
set is called vacant set. Although the precise construction of the random
interlacements on a graph is delicate, we give here a characterization of the
law ${\mathbb Q}_u$ that the vacant set induces on $\{0,1\}^\mathbb{V}$. For
this, consider a finite set $K \subset \mathbb{V}$ and define  the capacity
of $K$ as
\begin{equation}
  \text{cap}(K) = \sum_{x\in K}P^\mathbb{G}_x[\tilde H_K = \infty],
\end{equation}
with $\tilde H_K$ as in \eqref{e:hittingtimes}. The law $\mathbb Q_u$ of the indicator
function of the vacant set at level $u$ is the only measure on
$\{0,1\}^\mathbb{V}$ such that
\begin{equation}
  \label{e:Qchar}
  \mathbb Q_u[W_y = 1, \text{ for all $y \in K$}] = \exp\{-u \cdot \text{cap}(K)\},
\end{equation}
where $\{W_y\}_{y \in {\mathbb V}}$ are the canonical projections from
$\{0,1\}^\mathbb{V}$ to $\{0,1\}$, see (1.1) in \cite{Tei09}.


\section{Conditional probability estimate} \label{s:CP}

In this section, we derive in Proposition~\ref{p:c} an estimate on the
probability that in a finite time interval $[0,T]$, a random walk does not
visit a vertex $y$ in the boundary of a set $A$, given that it does not visit
the set $A$. This estimate will be crucial in the analysis of the breadth-first
search algorithm exploring the components of the vacant set, used in the proof
of Theorem~\ref{t:subcritical}.

We recall first  a  variational formula for the
expected entrance time.
\begin{lemma}
  \cite[Chapter 3, Proposition 41]{AF}
  \label{pr:AF}
  For a non-empty subset $A \subseteq V$,
  \begin{equation}
    \label{pr1} (E H_A)^{-1} =
    \inf \left\{ {\mathcal D}(f,f): \, f : V \to {\mathbb R},
      \, f=1 \text{ on }  A, \, \pi(f) = 0 \right\}.
  \end{equation}
  The minimizing function $f^\star$ in \eqref{pr1} is given by
  \begin{equation}
    \label{pr2} f^\star(x) = 1 - \frac{E_x H_A}{E H_A}.
  \end{equation}
\end{lemma}

Using this variational formula, we obtain the following estimate.
\begin{proposition}
  \label{pr:EH}
  Let $A$ and $C$ be disjoint non-empty subsets of $V$ and let
  $g^\star=g^\star_{A,C}$ be the equilibrium potential \eqref{e:potential} and
  $f^\star$ the minimising function \eqref{pr2}. Then
  \begin{equation}
    \label{eq:EH}
    {\mathcal D}(g^\star,g^\star)
    \Bigl( 1 - 2 \sup_{x \in C} |f^\star(x)| \Bigr) \leq
    \frac{1}{E[H_A]} \leq {\mathcal D}(g^\star,g^\star) \pi(C)^{-2}.
  \end{equation}
\end{proposition}

\begin{proof}
  We prove the right-hand inequality in \eqref{eq:EH} first. To this end,
  we modify the function $g^\star$ such that it becomes admissible for the
  variational problem \eqref{pr1} of Lemma~\ref{pr:AF}.  We define the function
  $\tilde g$ on $V$ by
  ${\tilde g} = ({g^\star - \pi(g^\star)})/({1- \pi(g^\star)})$. Then $\tilde g$
  equals $1$ on $A$ and $\pi({\tilde g})=0$, so we obtain from \eqref{pr1} that
  \begin{align}
    E[H_A]^{-1} \leq {\mathcal D}({\tilde g},{\tilde g})
    = {{\mathcal D}(g^\star,g^\star)}{(1-\pi(g^\star))^{-2}}.
  \end{align}
  Since $g^\star$ is non-negative, bounded by $1$ and non-zero only on $C^c$,
  we have $\pi(g^\star) \leq \pi(C^c)$ and the right-hand inequality of
  \eqref{eq:EH} follows.

  To prove the left-hand inequality in \eqref{eq:EH}, observe that the
  maximizer $f^\star$  of the variational problem \eqref{pr1} satisfies
  $f^\star = 1$ on $A$. Therefore
  \begin{equation}
    \label{EH1}
    {E[H_A]}^{-1} = {\mathcal D}(f^\star,f^\star)
    \geq \inf \bigl\{ {\mathcal D}(g,g):
      g: V \to {\mathbb R}, g = 1 \text{ on } A, \, g
      = f^\star \text{ on } C \bigr\}.
  \end{equation}
  Since $G$ is finite, the infimum is attained by a function $\hat g$ which
  satisfies the given boundary conditions on $A$ and $C$ and which is harmonic
  in $(A \cup C)^c$. In particular, the process
  $({\hat g}(X_{t \wedge H_{A \cup C}}))_{t \geq 0}$ is a $P_x$-martingale for
  any $x \in V$. From the optional stopping theorem, it follows that
  ${\hat g}(x) = g^\star(x) + \psi(x)$,  $x \in V$, where the function $\psi$
  is defined by
  \begin{equation}
    \label{EH1.2}
    \psi(x) = E_x [f^\star(X_{H_C}) \mathbf{1}_{\{H_C < H_A\}}],
    \textrm{ for } x \in V.
  \end{equation}
  Therefore,
  \begin{align}
    \label{EH2} E[H_A]^{-1} &\geq {\mathcal D}(g^\star+\psi, g^\star+\psi)
    \geq {\mathcal D}(g^\star, g^\star) + 2 {\mathcal D}(g^\star, \psi).
  \end{align}
  Since $\psi$ equals $0$ on $A$, $\lap g^\star$ equals $0$ on $(A \cup C)^c$
  and on $\inter C$, and $\lap g^\star (x) \geq 0$ for all $x \in \partial_i C$
  (indeed, $g^\star$ is non-negative on $V$ and equal to $0$ on $\partial_i C$),
  we have
  \begin{equation}
    {\mathcal D}(g^\star, \psi)
    = - \sum_{x \in \partial_i C} \lap g^\star (x) \psi(x) \pi_x
    \geq - |\psi|_\infty \sum_{x \in \partial_i C} \lap g^\star (x) \pi_x.
  \end{equation}
  Using again $g^\star$ equals $0$ on $\partial_i C$ and observing that
  $\lap f =  -\lap (1-f)$ for any real-valued function $f$ on $V$, as can be
  directly seen from the definition of $\lap$ in \eqref{def:lap}, we obtain
  \begin{align}
    \label{EH4} {\mathcal D}(g^\star, \psi) &\geq  |\psi|_\infty \sum_{x
      \in \partial_i C} \lap (1-g^\star) (x) (1-g^\star(x)) \pi_x.
  \end{align}
  Since $1-g^\star$ vanishes on $A$, while $\lap (1-g^\star) = \lap g^\star$
  vanishes on $(A \cup C)^c$ and on $\inter C$, the right-hand side equals
  $- |\psi|_\infty {\mathcal D}(1-g^\star,1-g^\star) = - |\psi|_\infty {\mathcal D}(g^\star,g^\star)$.
  Putting together \eqref{EH2} and \eqref{EH4} and using \eqref{EH1.2}, we
  therefore obtain
  \begin{align*}
    {E[H_A]}^{-1} \geq {\mathcal D}(g^\star, g^\star)  (1- 2|\psi|_\infty)
    \geq {\mathcal D}(g^\star, g^\star)  (1- 2 \sup_{x \in C} |f^\star(x)|).
  \end{align*}
  This yields the left-hand estimate in \eqref{eq:EH} and completes the proof
  of Proposition~\ref{pr:EH}.
\end{proof}

In order to apply the left-hand estimate of \eqref{eq:EH}, a bound on
$\sup_{x \in C}|f^\star(x)|$ is required. We will derive such a bound in
Proposition~\ref{pr:Eratio} below. In its proof we will need the following
technical lemma.

\begin{lemma} \label{lm:hit}
  Assume \eqref{a:reg} and consider $r,s \in {\mathbb N}$ and $x \in V$, such
  that $\tx(B(x,r+s)) \leq 1$. Then for any $y \in \partial_i B(x, r+s)$,
  \begin{align}
    P_y [H_{B(x,r)} < H_{B(x,r+s)^c} ] \leq c (d-1)^{-  s  }. \label{eq:hit}
  \end{align}
\end{lemma}

\begin{proof}
  We write $B = B(x,r)$ and $B' = B(x, r+s)$ and for every vertex  $z \in B'$
  we define $r_z = \dist(x,z)$. If $\tx(B') = 0$, then $B'$ is a tree and
  $r_{X_t}$ behaves like a random walk on $\mathbb{N}$ with drift, which steps
  left with probability $p=1/d$ and right otherwise.

  It is a known fact that the probability that a random
  walk on $\mathbb Z$ jumping with probability $p$ to the right and $1-p$ to
  the left started at $x>0$ hits $R\ge x$ before hitting zero equals (see
    e.g. \cite{D05}, Chapter 4, Example 7.1)
  \begin{equation}
    \label{eq:gamble}
    \frac{q^x - 1}{q^R - 1},\qquad\text{where }q=(1-p)/p.
  \end{equation}
  The inequality \eqref{eq:hit} then follows directly from
  \eqref{eq:gamble}.

  We thus assume that $\tx(B')=1$.  Let us call a vertex $z$ in
  $B'\setminus \{x\}$
  exceptional, if $z$ does not have $d-1$ neighbours $z'$ with $r_{z'}>r_z$. We
  claim that
  \begin{equation}
    \label{eq:driftout}
    \parbox[c]{0.8\textwidth}
      {There are at most two exceptional vertices. All of them
        are at the same distance (say
          $\rho $) of $x$
      and have at most two neighbours $z'$ with $r_{z'} \leq r_z$}.
  \end{equation}
  To see this, consider an exceptional vertex $z \in B'$. By definition,
  there is a pair $z_1, z_2$ of neighbours of $z$ with
  $r_{z_1}, r_{z_2} \leq r_z$. By considering geodesic paths from $z_1$
  and $z_2$ to $x$, one can extract a cycle in $B(x,r_z)$ containing $z$
  and exactly two of its neighbours $z_1$, $z_2$. By construction, this
  cycle has at most two vertices which maximize the distance to $x$. One of
  them is $z$. Second might be $z_1$ or $z_2$, in which case this vertex
  has the same distance to $x$ as $z$ and is also exceptional. To show
  that there cannot be another exceptional vertex other than $z$ (and
    potentially one of $z_1$, $z_2$), we suppose that there is one, we
  call it $z'$. By the same reasoning we can extract a cycle in $B'$
  containing $z'$ with $z'$ maximizing the distance to $x$. This cycle
  thus must be different from the one containing $z$. This is impossible
  since $\tx(B')=1$. Similarly, if $z$ has three or more neighbours $z_i$
  with $r_{z_i}\le r_z$, then every pair of them can be used to extract
  a cycle, all of them being different. This is again in contradiction
  with $\tx(B')=1$. With this we conclude \eqref{eq:driftout}.

  Let  $Y_t = \dist (B, X_{t \wedge H_{B \cup  (B')^c}})$. We compare $Y$ with
  a continuous-time birth-death process $U_t$ on $\{0,\dots,s+1\}$ given by the
  following transition rates
  \begin{equation}
    p_{i,i+1}=1-p_{i,i-1}=
    \begin{cases}
      (d-2)/d, &\text{if $i = \rho -r$},\\
      (d-1)/d, &\text{if $i\in \{1,\dots,s\}\setminus\{\rho -r\}$},
    \end{cases}
  \end{equation}
  and such that the states $0$ and $s+1$ are absorbing. More precisely, using
  \eqref{eq:driftout}, we can couple $Y$ (under law $P_y$) with $U$ (started
    from $s$) in such way that $U_t \le Y_k$ for every $t\ge 0$. This implies
  that
  \begin{equation}
    \label{e:domin}
    \begin{array}{c}
      P_y[H_{B}<H_{(B')^c}] \text{ is smaller or equal to the probability} \\
      \text{that $U$ hits $0$ before $s+1$, given that $U_0=s$}.
    \end{array}
  \end{equation}
  The last probability will now be estimated using a standard birth-death
  process computation. Let $f(i)$ be the probability that $U$ started at $i$
  hits $0$ and set $h(i)=f(i-1)-f(i)$, $i\in\{1,\dots s+1\}$. Clearly $f(0)=1$,
  $f(s+1)=0$ and the strong Markov property on the time of the first jump
  implies that $h(i)p_{i,i-1}=h(i+1)p_{i,i+1}$, $i\in \{1,\dots,s\}$. Fixing
  $h(s+1)=\gamma $ we use the above facts to get
  \begin{equation}
    \label{e:bd}
    1=f(0)\ge h(1)=
    \gamma\cdot \frac{p_{1,2}\dots p_{s,s+1}}{p_{1,0} \dots
      p_{s,s-1}}
    = \frac \gamma2 (d-1)^{s-1}(d-2)2.
  \end{equation}
  Moreover, conditioned on $U_0 = s$, the probability that
  $U$ hits zero before $s+1$ is $f(s) = \gamma$,
  i.e. $\text{Prob}[\text{$U$ hits $0$ before $s+1 | U_0=s$}]=f(s)=\gamma $.
  Putting this together with \eqref{e:domin} and \eqref{e:bd} the proof of
  Lemma~\ref{lm:hit} is finished.
\end{proof}

We now apply the last lemma to estimate the probability that the random walk,
started outside of the larger one of two concentric balls visits the small ball
before time $T > 0$.

\begin{lemma}
  \label{lm:mix}
  Assume that $G$ satisfies \eqref{a:reg} and consider $T>0$, $r,s\in \mathbb N$
  and $x \in V$ such that $\tx(B(x,r+ s)) \leq 1$. Then, for some $c,c'>0$,
  \begin{align}
    \label{eq:hmix}
    P_y [H_{B(x,r)} < T]
    \leq c T(d-1)^{-s} + e^{-c'T}\quad \text{for all $y \in B(x,r+s)^c$.}
  \end{align}
\end{lemma}

\begin{proof}
  As in the previous proof we write $B = B(x, r)$, $B'=B(x,r+s)$. From an
  exponential upper bound on the probability that a Poisson random variable
  with expectation $T$ is larger than $2T$, we have
  \begin{equation}
    P_y [H_B < T] \le P_y[\hat H_B\le 2T]+e^{-c' T},
  \end{equation}
  where $\hat H_B$ is the entrance time for the discrete-time walk defined
  below~\eqref{e:hittingtimes}.

  On the way from $y$ to $x$ (as in the lemma), the simple random walk must
  visit some vertex $z \in \partial_i B'$. After reaching such vertex, it
  either hits $B$ without exiting $B'$ or it exits $B'$. The probability of the
  first event is bounded from above by $c(d-1)^{-s}$, see Lemma~\ref{lm:hit}.
  When the second event occurs, the simple random walk must again pass through
  $\partial_i B'$ in order to visit $x$. At this point we can repeat the
  previous reasoning. However, before time $2T$ we can repeat this procedure at
  most $2T$ times, since we are considering a discrete-time walk. A union bound
  then implies
  \begin{equation}
    P_y[\hat H_B\le 2T]\le 2 T c (d-1)^{-s}
  \end{equation}
  and Lemma~\ref{lm:mix} follows by renaming constants.
\end{proof}

Finally, we prove the proposition that will allow us to use the left-hand side
of the estimate \eqref{eq:EH} on $E[H_A]^{-1}$, derived in the beginning of
this section.

\begin{proposition} \label{pr:Eratio}
  Let $G=(V,\mathcal E)$ be a graph on $n$ vertices satisfying
  \eqref{a:reg}, \eqref{a:gap} and let $A \subseteq V$,
  $s \in (0,\log_{d-1} n] \cap {\mathbb N}$ such that
  $|A| \leq n/2$ and  $\tx(B(x, s)) \leq 1$ for every $x\in A$.
  Then
  \begin{equation}
    \label{eq:Eratio}
    \sup_{y \in V: \dist(y,A) > s} \Bigl| \frac{E_y[H_A]}{E[H_A]} - 1 \Bigr|
    \leq c |A| (d-1)^{-s} \log^4 n.
  \end{equation}
\end{proposition}

\begin{proof}[Proof of Proposition~\ref{pr:Eratio}.]
  In essence, the proof is an application of the estimate \eqref{eq:I}, which
  shows that the distribution of the random walk on $G$ at time
  $T=\lambda_G^{-1} \log^2 n$ is close to uniform. From Lemma~\ref{lm:mix}, we
  know that it is unlikely that the random walk started at $y$ reaches a point
  $x$ in $A$ before time $T$ and this will yield \eqref{eq:Eratio}.

  We shall require the following rough bounds:
  \begin{equation}
    \label{Er1} \frac{n}{4|A|} \leq E[H_A] \leq \sup_{z \in G} E_z[H_A]
    \leq c n\log n, \textrm{ for some constant } c>0.
  \end{equation}
  The first inequality in \eqref{Er1} follows from the right-hand estimate of
  \eqref{eq:EH} with $C$ chosen as $A^c$, \eqref{eq:Dgg}, and our assumption
  that $|A|\leq n/2$. To prove the last inequality in \eqref{Er1}, observe that
  for  $t= 2 \log n/\gap$, assumption \eqref{a:gap} and \eqref{eq:I} imply
  \begin{equation}
    \inf_{z \in V} P_z[H_A \leq 2 \log n / \gap]
    \geq \inf_{z \in V}P_z [X_{ 2 \log n / \gap} \in A] \geq (2n)^{-1}
  \end{equation}
  By the simple Markov property applied at integer multiples of  $t$, it
  follows that $H_A$ is stochastically dominated by $t$ times a geometrically
  distributed random variable with success probability $1/2n$ and \eqref{Er1}
  readily follows.

  Let $y$ be chosen as in the statement and let us first consider the
  expectation of $H_A$ starting from $X_{T}$. From \eqref{eq:I} and our crude
  estimate \eqref{Er1}, we obtain, for any $z \in V$,
  \begin{equation}
    \begin{split}
      \label{Er2}
      \bigl| E_z[ E_{X_T}[H_A] ] - E[H_A] \bigr|
      &\leq \sum_{z' \in V} \big| P_z[X_T = z'] - \pi_{z'}\big| E_{z'}[H_A] \\
      &\leq \sum_{z' \in V} e^{-\log^2 n} n \log n
      \le n^3 e^{-\log^2 n}.
    \end{split}
  \end{equation}
  We now apply this inequality to find an upper bound on $E_y[H_A]$. Since
  $H_A \leq T + H_A \circ \theta_{T}$, the simple Markov property applied at
  time $T$ and \eqref{Er2} imply that for any $z \in V$,
  \begin{equation}
    \label{Er3} E_z[H_A] \leq T + E_z \bigl[ E_{X_T}[H_A] \bigr]
    \leq T + n^{3} e^{-\log^2 n} + E[H_A].
  \end{equation}
  With the first inequality in \eqref{Er1}, we deduce that
  \begin{align}
    \label{Er5}
    \frac{E_z[H_A]}{E[H_A]} - 1 &\leq  (T + n^{3} e^{-\log^2 n})
    \frac{4|A|}{n} \leq \frac {c |A| \log^2 n }{n}.
  \end{align}
  which is ample for one side of \eqref{eq:Eratio}. To prove the other half of
  \eqref{eq:Eratio}, choose $y$ as in the statement and apply the simple Markov
  property at time $T$ to infer that
  \begin{equation}
    \begin{split}
      E_y[H_A]
      &\geq E_y [\mathbf{1}_{\{H_A > T\}} E_{X_T}[H_A] ]
      = E_y [ E_{X_T}[H_A] ] - E_y [\mathbf{1}_{\{H_A \leq T\}} E_{X_T}[H_A] ] \\
      &\stackrel{\eqref{Er2}}{\geq} E[H_A] - n^{3} e^{-\log^2 n} -
      P_y[H_A \leq T]  \sup_{z \in V} E_z[H_A]\\
      &\stackrel{\eqref{Er3}}{\geq} E[H_A] - 2 n^{3} e^{-\log^2 n} -
      P_y[H_A \leq T]  ( T + E[H_A]).
    \end{split}
  \end{equation}
  Applying \eqref{eq:hmix} to the probability on the right-hand side and
  rearranging, we find that
  \begin{equation}
    \frac{E_y[H_A]}{E[H_A]} -1 \geq - c |A| (d-1)^{-s}  \log^4 n,
  \end{equation}
  which together with \eqref{Er5} completes the proof of
  Proposition~\ref{pr:Eratio}.
\end{proof}

We now analyse the distribution of the hitting time of a point $y$ conditioned
on the event that a certain set $A$ is vacant. This estimate will be helpful
for the analysis of the breadth-first search algorithm used in
Theorem~\ref{t:subcritical}.

For any non-empty connected set $A\subset V$, $r\ge 1$, and $y\in \partial_e A$
we define
\begin{equation}
  \begin{split}
    \label{e:future}
    \Q_A(y,r)=
    \{&z\in B(A,r)\setminus A: z \text{ is connected to $y$ in
        $B(A,r)\setminus A$}\}.
  \end{split}
\end{equation}
Observe that $y\in \Q_A(y,r)$. In the breadth-first search algorithm to be
introduced in Section~\ref{s:sub}, the set $\Q_A(y,r)$ can be viewed as the
`future of $y$ seen from $A$'. We say that $\Q_A(y,r)$ is \textit{proper} when
(see Figure~\ref{f:proper})
\begin{equation}
  \begin{split}
    \label{e:propercone}
    \text{(i)}\quad&\tx\big(\Q_{A}(y,r)\big)=0,\\
    \text{(ii)}\quad&\text{$y$ has a unique neighbour $\bar y$ in $A$},\\
    \text{(iii)}\quad&
    \begin{aligned}[t]
      &\text{for any vertex
        $y'\in A\setminus \bar y$, every path from $y$ to $y'$}\\[-1.4mm]
      &\text{leaves $B(A,r)\setminus A$ before reaching $y'$.}
    \end{aligned}
  \end{split}
\end{equation}
\begin{figure}
  \includegraphics[height=6cm]{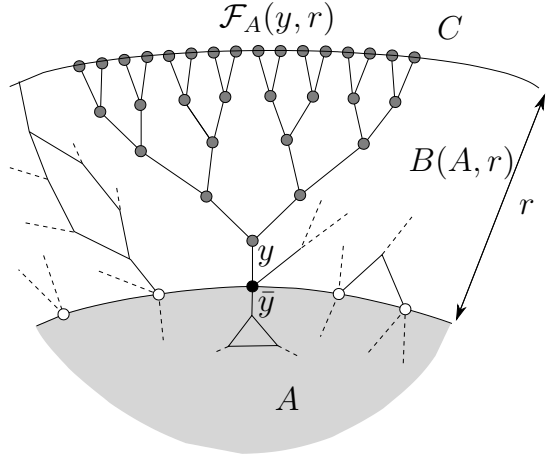}
  \caption{Proper $\Q_A(y,r)$ (gray points) on a $3$-regular graph with $r=5$.}
  \label{f:proper}
\end{figure}
\begin{proposition}
  \label{p:c}
  Let $s\in[2,(\good\wedge \frac 12) \ld n)$, $A\subset V$, $A\neq \varnothing$
  with $ |B(A,s)|\le \sqrt {n }$, and $y\in \partial_e A$, such that $\Q_A(y,s)$
  is proper. Then, for any $T > 1$,
  \begin{equation}
    \label{eq:c}
    \bigg|
    \ln P \big[H_{A \cup \{y\}} > T \bigl| H_{A} > T \big] +
    \frac{T}{n} \frac{(d-2)^2}{d(d-1)}
    \bigg|
    \le \frac{c|A|}{ n}
    \Big(\frac{T|A|\ln^4 n}{(d-1)^{s}} +1 \Big).
  \end{equation}
\end{proposition}

\begin{proof}
  We set
  \begin{equation}
    \label{e:uiz}
    F_A^y(T)=P [H_{A \cup \{y\}} > T | H_{A} > T ]
    =\frac{P[H_{A\cup \{y\}}>T]}{P[H_A>T]} ,
  \end{equation}
  and use results of \cite{AB} to estimate both numerator and denominator.
  Namely, by \cite{AB} (1) and Theorem 3, for any $A\subset V$, $t>0$,
  \begin{equation}
    \label{e:uia}
    \Big(1-\frac 1{\lambda_G E_{\alpha_A} H_A}\Big)
    \exp\Big(-\frac t{E_{\alpha_A} H_A}\Big)
    \le P[H_A > t]
    \le (1-\pi (A))
    \exp\Big(-\frac t{E_{\alpha_A} H_A}\Big) .
  \end{equation}
  Here $\alpha_A$ is the quasi-stationary distribution for the random walk
  killed on hitting $A$. We will only need its following properties, see
  \cite{AB} Lemma 2 and Corollary 4,
  \begin{equation}
    \label{e:uib}
    \frac{1-\pi (A)}{\sum_{x\in A, y\in A^c}\pi (x)p_{xy}}
    \le \frac{E H_A}{1-\pi (A)}
    \le E_{\alpha_A}H_A
    \le E H_A + \lambda_G^{-1} .
  \end{equation}
  Observe that the left-hand side is bounded from below by
  $n/(2|A|)\ge c \sqrt n$ for $A$ as in the statement.

  Writing
  $\tilde A=A\cup \{y\}$, $\alpha = \alpha_A$ and
  $\tilde \alpha = \alpha_{\tilde A}$, and applying \eqref{e:uia} for $A$ as
  well as for $\tilde A$ to bound the conditional expectation
  \eqref{e:uiz}, we obtain
  after rearranging and taking logarithm
  \begin{equation}
    \label{e:uic}
    \ln\frac{1-\frac{1}{\lambda_G E_{\tilde\alpha} H_{\tilde A}}}{1-\pi ( A)}
    \le
    \ln F_A^y(T) -
    \frac{T}{E_\alpha H_A}+ \frac{T}{E_{\tilde \alpha } H_{\tilde A}}
    \le
    \ln\frac{1-\pi (\tilde A)}{1-\frac{1}{\lambda_G E_\alpha H_A}}.
  \end{equation}
  Using \eqref{e:uib} and the observation following it, we see that
  $E_\alpha H_A\ge c \sqrt n$. Therefore, by expanding the function
  $\ln(1/(1-x))$ around zero and using \eqref{e:uib} again, we obtain that the
  right most term in \eqref{e:uic} is bounded from above by $c|A|/n$, which can
  be included in the error term of \eqref{eq:c}. A similar reasoning implies
  that the left-most term in \eqref{e:uic} is bounded from below by $-c|A|/n$,
  which can again be accommodated into the error of \eqref{eq:c}.

  The inequalities in \eqref{e:uib} further imply that
  $0\le E_\alpha H_A-E H_A \le \lambda_G^{-1}$ and therefore
  \begin{equation}
    \bigg|\frac T {E_\alpha H_A}-\frac T {E H_A}\bigg|\le
    \frac {c T}{(E H_A)^2}
    \le c T \frac {|A|^2}{n^2},
  \end{equation}
  where in the last inequality we used \eqref{e:uib} again. The right-hand term
  in the last display is again smaller than the error in \eqref{eq:c}, since
  $(d-1)^{-s}\ge n^{-1/2}$ by assumption $s\le \frac 12\ld n$.

  Finally, we use Proposition~\ref{pr:EH} to approximate $1/EH_A$ and
  $1/EH_{\tilde A}$. To this end we introduce $C=B(A,s)^c$ and we set
  $g=g^\star_{A,C}$, see \eqref{e:potential}. Then by Proposition~\ref{pr:EH}
  we have that
  \begin{equation}
    \label{e:uid}
    \bigg|\frac{T}{EH_A}- T\mathcal D(g,g)\bigg|
    \le T\mathcal D(g,g)\Big( \pi (C)^{-2}-1 +
      2 \sup_{z\in C} \Big| \frac{E_z H_A}{E H_A} - 1 \Big| \Big).
  \end{equation}
  By \eqref{eq:Dgg}, $\mathcal D(g,g)\le \pi (A)$.  From the assumption
  $B(A,s)\le \sqrt n$ it follows that
  $\pi (C)^{-2}-1\le c n^{-1/2}\le c (d-1)^{-s}$. Finally, the
  Proposition~\ref{pr:Eratio} implies that the supremum in \eqref{e:uid} is
  bounded by $c|A|(d-1)^{-s}\ln^4 n$. Hence, the right-hand side of
  \eqref{e:uid} is smaller than the error term in \eqref{eq:c}. An analogous
  computation proves that $T/E H_{\tilde A}$ is approximated by
  $T \mathcal D(\tilde g, \tilde g)$, where $\tilde g =g^\star_{\tilde A,C}$. A
  little bit of care is only needed when applying Proposition~\ref{pr:Eratio},
  since $\dist(\tilde A, C)=s-1$.

  We have thus proved that $\ln F_A^y(T)$ is well approximated by
  $T\big(\mathcal D(g,g)-\mathcal D(\tilde g,\tilde g)\big)$ up to the error on
  the right-hand side of \eqref{eq:c}.  We now estimate this expression. Let
  $\bar y$ be the unique neighbour of $y$ in $A$. By \eqref{eq:Dgg},
  \begin{equation}
    \label{e:uie}
    T\big(\mathcal{D}(g,g) - \mathcal{D}(\tilde g,\tilde g)\big)
    = \frac{T}{n} \Big[
      \sum_{z \in A} P_z [\tilde H_{A} > H_{C}]
      -\sum_{z \in \tilde A } P_z [\tilde H_{\tilde A} > H_{ C}]
      \Big].
  \end{equation}
  Now we use our assumption that $\Q_A(y,s)$ is proper. Due to
  \eqref{e:propercone}(iii), for all $z\in A\setminus \{\bar y\}$, there is no
  path from $z$ to $y$ using only vertices in $B(A,r)\setminus A$. Therefore,
  for such $z$ $P_z [\tilde H_{\tilde A} > H_{C}] = P_z [\tilde H_{A} > H_{C}]$,
  and \eqref{e:uie} equals
  \begin{equation}
    \frac{T}{n} \big[
      P_{\bar y}[\tilde H_{A} > H_{C}]
      - P_{\bar y}[\tilde H_{\tilde A} > H_{C}]
      -P_y[\tilde H_{\tilde A } > H_{C}]
      \big].
  \end{equation}
  Conditioning the first two terms on $X_{\tau_1}$, since $\Q_A(y,s)$ is
  proper, we get
  \begin{equation}
    \label{e:uif}
    \frac{T}{n} \Big\{
      \frac{1}{d} P_{y}[H_{A} > H_{C}]
      -P_y[\tilde H_{\tilde A } > H_{C}]
      \Big\}.
  \end{equation}
  Since by assumption  $\tx (\Q_A(y,s)) =0 $, these probabilities can be
  computed using the formula \eqref{eq:gamble} for the random walk with
  drift. Setting $q=1/(d-1)$ we have
  \begin{equation}
    P_y[\tilde H_{\tilde A } > H_{C}]=\frac {d-1}{d} \frac
    {1-q}{1-q^{s-1}}, \quad \text{and}\quad
    P_{y}[H_{A} > H_{C}]=\frac{1-q}{1-q^s}.
  \end{equation}
  Inserting this into \eqref{e:uif} we obtain that
  \begin{equation}
    \bigg|T\big(\mathcal{D}(g,g) - \mathcal{D}(\tilde g,\tilde g)\big)+
    \frac{T(d-2)^2}{nd(d-1)}\bigg| \le
    \frac{c T}{n} \frac{1}{(d-1)^{s}}.
  \end{equation}
  This completes the proof, since the error is smaller than the right-hand side
  of \eqref{eq:c}.
\end{proof}

We now use the same techniques to control the hitting time distribution of a point
with tree-like neighbourhood.
\begin{lemma}
  \label{l:hitpoint}
  Let $y\in V$ be such that $\tx(B(y,s))=0$ for some $s\in [1,\good \ld n)$.
  Then, for any $T>1$,
  \begin{equation}
    \label{e:hitpoint}
    \Big|\ln P[H_y>T] + \frac {T(d-2)}{n(d-1)}\Big|
    \le \frac {c}n \Big(\frac {T\ln^4 n}{(d-1)^s}+1\Big).
  \end{equation}
\end{lemma}
\begin{proof}
  The proof follows the same lines as the previous one with $\tilde A=\{y\}$
  and without the conditioning, which is equivalent to controlling the
  numerator of \eqref{e:uiz} only.  The same reasoning as before implies that,
  for $g=g^\star_{y,B(y,s)^c}$, $|\ln P[H_y\ge T]-T\mathcal D(g,g)|$ is smaller
  than the right-hand side of \eqref{e:hitpoint}. Using \eqref{eq:Dgg},
  \eqref{eq:gamble}, with $q=1/(d-1)$ again, we get
  \begin{equation}
    \mathcal D(g,g)=n^{-1}P_y[\tilde H_y>H_{B(y,s)^c}]= \frac{1}{n}\frac{1-q}{1-q^s}=
    \frac{d-2}{n(d-1)}+O((d-1)^{-s}),
  \end{equation}
  which finishes the proof.
\end{proof}

\section{Piecewise independent measure}
\label{s:partsbridges}
We now make another preparative step in order to prove our main results. In
later sections, it will be convenient to split the random walk trajectory
$X_{[0,un]}$ into smaller pieces and to treat pieces that are sufficiently
distant in time as being independent of one another. Although this kind of
independence does not hold under the random walk measure $P^{un}$, we will in
this section construct a new measure on the space of trajectories with the
desired independence properties. In Lemma~\ref{l:partsbridges}, we then
estimate the error we make when replacing $P^{un}$ by this new measure.

For the construction, we choose real parameters
\begin{equation}
  L = n^\gamma , \quad \ell = (\ln n)^2,
\end{equation}
where $\gamma \in (0,1)$  will be fixed later. We consider an abstract
probability space $(\Omega ,\Plr)$ (expectation denoted by $\Elr$)
on which we define a sequence of
i.i.d.~random variables $Y^i$, $i\ge 0$, with values in $D([0,L],V)$ and
the marginal distribution $P^L$ (defined in Section~\ref{ss:rw}). We set
$a_i$, $b_i$ to be the start- and the end-point of $Y^i$, $a_i=Y^i_0$,
$b_i=Y^i_L$.  On the same space $\Omega $ we further define a sequence of
random variables $Z^i$, $i\ge 0$, with values in $D([0,\ell],V)$.  Given
$a_i$, $b_i$, $i\ge 0$, the random variables $Z^i$ are independent,
conditionally independent of the sequence $(Y^i)$, and the random
variable $Z^i$ has the random-walk bridge distribution
$P^\ell_{b_i,a_{i+1}}$. We call the $Y^i$'s \textit{segments} and  $Z^i$'s
\textit{bridges}.

We now concatenate the $Y^i$'s and $Z^i$'s to obtain an element of
$D([0,\infty],V)$. More precisely,  we define the concatenation mapping
$\mathcal X$ from $\Omega $ to $D([0,\infty),V)$ as follows: For
$t\ge 0$, let $i_t \in \mathbb N$ and $s\in [0,L+\ell)$ be given by
$t=i_t(L+\ell)+s_t$. Then, for $t\ge 0$,
\begin{equation}
  \label{e:mathcalX}
  \mathcal X_t(Y^0,Z^0,Y^1,Z^1,\dots)=
  \begin{cases}
    Y^{i_t}_{s_t},&\text{if $0 \le s_t\le L$},\\
    Z^{i_t}_{s_t-L},&\text{if $L<s_t < L+\ell$}.
  \end{cases}
\end{equation}
The mapping $\mathcal X$ induces a new probability measure
$Q= \Plr \circ  {\mathcal X}^{-1}$ on
$D([0,\infty),V)$. We use  $Q^s$, $s\ge 0$, to denote the restriction
of $Q$ to $D([0,s],V)$. The measure $Q^\cdot$  will be used to
approximate $P^\cdot$ later. We control this approximation now.

\begin{lemma}
  \label{l:partsbridges}
  For every fixed $u>0$, the measures $P^{un}$ and $Q^{un}$ are
  absolutely continuous and there exist constants $c,c'>0$ depending only on
  $\gap$ such that
  \begin{equation}
    \label{e:partsbridges}
    \Big|\frac {\d P^{un}}{\d Q^{un}}-1\Big| \le
    c' u e^{-c \ln^2 n}.
  \end{equation}
\end{lemma}

\begin{proof}
  Let $u'\ge u$ be the smallest number such that $u'n$ is an integer multiple
  of $(L+\ell)$, and set $m=u'n/(L+\ell)\in \mathbb N$. Let further $A$ be an
  arbitrary $\mathcal F_{un}$-measurable subset of $D([0,u'n],V)$. Since
  $P^{un}$ and $Q^{un}$ are the restrictions of $P^{u'n}$ and
  $Q^{u'n}$   to $D([0,un],V)$, it is sufficient to prove the lemma
  with $u$ replaced by $u'$. To this end, we set $t_{2k}=k(L+\ell)$,
  $t_{2k+1}=k(L+\ell)+L$ for $k\in \{0,\dots,m\}$, and write
  \begin{equation}
    \label{e:PP0}
    P^{u'n}[A]=\sum_{x_0,\dots,x_{2m}\in V}
    P^{u'n}[A|X_{t_i}=x_i,0\le i\le 2m]
    P^{u'n}[X_{t_i}=x_i,0\le i\le 2m].
  \end{equation}
  By the Markov property
  \begin{equation}
    \label{e:PP}
    P^{u'n}[X_{t_i}=x_i,0\le i\le 2m] =
    \pi(x_0)\prod_{k=0}^{m-1}
    P^L_{x_{2k}} [X_L=x_{2k+1}] P^\ell_{x_{2k+1}} [X_\ell =x_{2k+2}].
  \end{equation}
  The construction of the measure $Q$ implies that
  \begin{align}
    Q^{u'n}[A|X_{t_i}=x_i,0\le i\le 2m]&=
    P^{u'n}[A|X_{t_i}=x_i,0\le i\le 2m],\\
    \label{e:QQ}
    Q^{u'n}[X_{t_i}=x_i,0\le i\le 2m] &=
    \pi (x_0)\prod_{k=0}^{m-1}
    P^L_{x_{2k}} [X_L=x_{2k+1}] \pi (x_{2k+2}).
  \end{align}
  Comparing \eqref{e:PP} and \eqref{e:QQ}, it remains to control the ratio
  $P^\ell_x[X_\ell=y]/\pi (y)$. However, by \eqref{a:gap} and \eqref{eq:I},  this
  ratio is bounded by $1+ne^{-\gap \ell}$. Hence, \eqref{e:PP0} is bounded from
  above by
  \begin{equation}
    \begin{split}
      & (1+ne^{-\gap \ell })^m\sum_{x_0,\dots,x_{2m}\in V}
      Q^{u'n}[A|X_{t_i}=x_i,0\le i\le 2m]
      Q^{u'n}[X_{t_i}=x_i,0\le i\le 2m]
      \\&\le Q^{u'n}[A] (1+u e^{-c \ln^2 n}),
    \end{split}
  \end{equation}
  where, in the last inequality, we changed the constants to accommodate the
  terms polynomial in $n$. A lower bound can be obtained analogously. We have
  thus shown
  \begin{align}
    \label{e:poiu}
    Q^{u'n}[A] (1-u e^{-c \ln^2 n}) \leq P^{u'n}[A]
    \leq Q^{u'n}[A] (1+u e^{-c \ln^2 n}).
  \end{align}
  It immediately follows that $P^{u'n}$ and $Q^{u'n}$ are absolutely
  continuous. Moreover, the fact that \eqref{e:poiu} holds for any event $A$ in
  ${\mathcal F}_{u'n}$ yields directly the estimate \eqref{e:partsbridges}.
\end{proof}

We end this section with a simple lemma which controls the number of jumps
performed by segments and bridges, which will be useful several times later
(see Subsection~\ref{ss:rw} for the definition of the jump process $(N_t)_{t \geq 0}$).

\begin{lemma}
  \label{lm:visits}
  \begin{equation}
    \label{eq:Wj}
     P[2^{-1} n^\piece < N_L < 2 n^\piece ]
    \geq 1- e^{-c_{\piece }n^\piece }.
  \end{equation}
  For any $x,y \in V$,
  \begin{equation}
    \label{eq:Wijlog3}
    P^\ell_{xy} \big[  N_\ell > \ln^3 n  \big] \leq c\, \exp\{-c' \ln^3 n\}.
  \end{equation}
\end{lemma}

\begin{proof}
  Under the measure $P$, the random variable $N_L$ has Poisson
  distribution with parameter $L=n^\piece $. Hence, \eqref{eq:Wj} follows by
  a standard large deviation argument.

  In order to prove \eqref{eq:Wijlog3}, note first that
  $N_\ell$ is not necessarily a Poisson random variable under $P^\ell_{xy}$,
  due
  to the conditioning on the position of the endpoint. However, using \eqref{eq:I} for the
  last inequality, we have
  \begin{equation}
    \label{eq:EeWij}
    \begin{split}
      \sup_{x,y\in V} E_{x,y}^{\ell}[e^{N_\ell}]
      &= \sup_{x,y\in V} E_x^{\ell} \big[ e^{N_\ell} | X_\ell = y \big]
      \leq  \frac{\sup_{x\in V} E_x^{\ell}[e^{N_\ell}]}{ \inf_{x,y \in V}
        P_x[X_{\ell} = y]} \\
      & \leq \frac{\exp\{(e-1)\ell\}}{\inf_{x,y \in V} P_x[X_{\ell} = y]}
      \le 2n^{-1} \exp\{c \ln^2 n\}
    \end{split}
  \end{equation}
  for $n$ larger than some $c'$.  The exponential Chebyshev inequality then
  implies claim (ii) for such $n$. Adjusting the constants to make the claim
  valid for all $n$ finishes the proof.
\end{proof}
\section{Sub-critical regime}
\label{s:sub}
In this section we prove Theorem~\ref{t:subcritical}, which states that if
$u>u_\star$, then the maximal connected component $\comp_\Max$ of
$\mathcal V_n^u$ is typically of size $O(\ln n)$. We will do it by analysing a
breadth-first-search (BFS) algorithm which explores the component $\comp_x$
of the vacant set containing a given vertex $x$. This algorithm is similar to
the one used in the Bernoulli percolation case, but has some important
modifications due to the dependence in our model.

We start the proof by reducing the complexity of the problem. We set, as in
Section~\ref{s:partsbridges}, $L=n^\gamma$, $\ell=\ln^2 n$, with
$\gamma \in (0,1)$.  Due to Lemma~\ref{l:partsbridges} it is sufficient to show
that Theorem~\ref{t:subcritical} holds with with $P^{un}$ replaced by
$Q^{un}=\Plr \circ \mathcal X^{-1}|_{D([0,un],V)}$.

Since we are looking for an upper bound on the vacant set, we can
disregard the bridges $Z^i$ in the concatenation $\mathcal X$
(cf.~\eqref{e:mathcalX}). More precisely, we set
$m=\lfloor un/(L+\ell)\rfloor$, and we observe that $\Plr$-a.s.~the
vacant set
\begin{equation}
  \mathcal V_n^u
  = V\setminus\{\mathcal X_t((Y^i),(Z^i)):t\in[0,un]\}
\end{equation}
is a subset of the \emph{vacant set left by segments}, $\bV^u$,
\begin{equation}
  \label{e:vsls}
  \bV^u:= V\setminus\cup_{i<m}\Ran Y^i, \text{ where } \Ran Y^i = Y^i_{[0,L]}.
\end{equation}
Let $\bC_\Max=\bC^u_\Max$ and $\bC_x=\bC^u_x$ be the largest connected
component, and the component containing $x$ of $\bV=\bV^u$, respectively.  Then
the inequality
\begin{equation}
  \Plr [|\bC_\Max| \ge K\log n]\le n \Plr \big[|\bC_x| \ge K \log
    n\big]
  \end{equation}
implies that to prove Theorem~\ref{t:subcritical} it is sufficient to show the
following proposition.

\begin{proposition}
  \label{p:onecluster}
  Let $G$ be a connected graph on $n$ vertices satisfying the assumptions
  \eqref{a:reg}, \eqref{a:tx1}, \eqref{a:gap} and let $u>u_\star$. Then for every
  $\sigma >0 $  there exist $1\le c,K<\infty$ not depending on $n$ (but
    depending on $d$, $u$, $\txconst$, $\gap$) such that
  \begin{equation}
    \label{e:onecluster}
    \Plr [|\bC^u_x|\ge K\ld n]\le c n^{-\sigma-1 }.
  \end{equation}
\end{proposition}

\begin{proof}
  We prove this proposition by analysing the following BFS algorithm. During
  the run of the algorithm, all vertices in $V$ are in one of four states:
  explored-vacant, explored-occupied, not-explored or in-queue. The set of
  vertices with the state in-queue  is organised as a queue $\mathcal Q$, that
  is it is ordered, the vertices are added to its end and removed from its
  beginning. The vertices in $\mathcal Q$ wait to get explored.

  Further, the state of any index $i\in\{0,\dots,m-1\}$ of a segment $Y_i$ can be
  either free or tied. Note here that these states do not change the
  behaviour
  of the algorithm, but will be used for its analysis. Their meaning will be
  easier to understand as we get to Lemma~\ref{lm:givenAk}.

  When the algorithm starts, all vertices different from $x$ are in
  not-explored state, $x$ is in-queue, $\mathcal Q=(x)$, and all indices
  $i \in \{0,\dots,m-1\}$ are free.

  At the step $k$ of the algorithm, the first vertex $y$ of the queue
  $\mathcal Q$ is removed from $\mathcal Q$. If $y \notin \bV$, then the state
  of $y$ is changed to explored-occupied,  and all indices $i$ of segments
  intersecting $y$ (i.e. $\{i: y \in Y^i\}$) become tied. On the other hand, if $y \in \bV$,
  then the state of $y$ changes to explored-vacant and all non-explored
  neighbours of $y$ in $G$ are placed at the end of $\mathcal Q$, in other
  words, their state changes to in-queue. To avoid ambiguity, we suppose that
  $V$ is equipped with an ordering and the neighbours of $y$ are added to
  $\mathcal Q$ according to this ordering.

  The algorithm stops if the queue $\mathcal Q$ is empty, or if the set of
  explored-vacant vertices has more than $K\ld n$ vertices. Since this set is
  subset of $\bC_x$ by construction, we know that only in the second case we
  have $|\bC_x|\ge K\ld n$. Hence, in order to establish
  Proposition~\ref{p:onecluster}, one only needs to show that
  \begin{equation}
    \begin{array}{c}
      \label{e:finish}
      \text{there are $K, c > 0$ such that, with probability at least
        $1-cn^{-\sigma -1}$,}\\
      \text{the algorithm finishes because the queue gets empty. }
    \end{array}
  \end{equation}

  To analyse the algorithm we need more notation. Throughout this section we
  fix
  \begin{equation}
    \label{e:subr}  r=(7\ld\ld n)\vee 2.
  \end{equation}
  We let $\EV_k$ ($\EO_k, \IQ_k$) stand for the (random) set of vertices in the
  explored-vacant (explored-occupied, in-queue, respectively) state before the
  beginning of the $k$-th step of the algorithm. Similarly, $\FF_k, \DD_k$
  denote the sets of free and tied indices at this moment. We set
  $\EE_k=\EO_k\cup \EV_k$ and let $y_k$ be the vertex being explored in the $k$-th
  step. In particular $y_1=x$, $\EV_1=\EO_1=\DD_1=\varnothing$. Let  $k_\Max$
  be the step when the algorithm finishes,
  \begin{equation}
    k_\Max=\min\{k:|\IQ_k|=0 \text{ or } |\EV_k|\ge K \ld n\}.
  \end{equation}
  Observe that by construction $\EO_k\subset \partial_e \EV_k$ and thus
  $\EE_k\subset (\EV_k\cup \partial_e \EV_k)$. Since $G$ is $d$-regular,
  and since exactly one vertex is explored at every step,
  this implies
  \begin{equation}
    \label{e:scsup}
    |\EE_k|=k-1\le dK \ld n,\qquad \text{for all $k\le k_\Max$.}
  \end{equation}
  We further define a filtration
  $(\mathcal A_k)_{k\ge 1}$, where the $\sigma $-algebra $\mathcal A_k$
  contains all information discovered by the algorithm before the $k$-th
  step, that is
  \begin{equation}
    \label{e:calAk}
    \mathcal A_k=\sigma (\EV_j,\EO_j, \IQ_j,\FF_j,\DD_j: j\le k\wedge
      k_\Max).
  \end{equation}
  Observe that, due to the ordering that we use while adding vertices to
  $\mathcal Q$, the random variables $y_1$, \dots, $y_{k-1}$ are
  $\sigma(\EO_k, \EV_k)$-measurable.

  To prove Proposition~\ref{p:onecluster} we analyse the process
  recording the length of the queue,  $q_k=|\IQ_k|$, $1\le k\le k_\Max$.
  We use $r_k=q_{k+1}-q_k$, $1\le k< k_\Max$, to denote the size of its jumps.  Since
  the graph $G$ is $d$-regular, in step $k$, at most $d-1$ ($d$ if $k=1$)
  vertices are added to $\mathcal Q$, and every time exactly one vertex
  is removed from it. Therefore, $r_1\in\{-1,d-1\}$ and
  $r_k\in \{-1,\dots,d-2\}$, for $k\ge 2$.

  Roughly speaking, to prove \eqref{e:finish} we will show that the
  process $(q_k)_{k \leq k_\Max}$ has a `down-drift'. For this, we need a
  lower bound on the probability that $r_k=-1$ given the past
  $\mathcal A_k$ of the algorithm. Since $r_k=-1$, whenever
  $y_k\notin\bV^u$, we have, on the event $\{ k < k_\Max\}$,
  \begin{equation}
    \label{e:sca}
    \begin{split}
      \Plr [r_k=-1|\mathcal A_k]&\ge
      \Plr [y_k\notin \bV|\mathcal A_k]
      = \Plr \big[y_k\in\cup_{i<m}\Ran Y^i\big| \mathcal A_k\big]
      \\&\ge \Plr \big[y_k\in\cup_{i\in \FF_k}\Ran Y^i\big| \mathcal A_k\big].
    \end{split}
  \end{equation}

  The reason why we have made the distinction between the free and tied indices
  is made clear in the short lemma below.
  \begin{lemma}
    \label{lm:givenAk}
    Let $k\ge 1$ and $k_* = k \wedge k_\Max$. Then, conditioned on the $\sigma$-algebra
    $\mathcal{A}_{k_*}$,
    the collection $(Y^i)_{i\in \FF_{k_*}}$ is i.i.d.~with marginal distribution $P^L[\cdot|H_{\EE_{k_*}}>L]$.
  \end{lemma}

  \begin{proof}[Proof of Lemma~\ref{lm:givenAk}.]
    Let $\mathsf H = (\EV_j, \EO_j, \IQ_j, \FF_j, \DD_j)_{j\le k_*}$
    be the whole history of the algorithm until the time $k_*$.
    We use $\bar H = (V_j,O_j,Q_j,F_j,T_j)_{j\le {\bar k}}$ to denote possible outcomes of
    $\mathsf H$, here $\bar k$ is a positive integer.
    Since the $Y^i$'s have marginal distribution $P^L$, it suffices to
    prove that for any $\bar H$ such that
    $\Plr [\mathsf H = \bar H]>0$ and any measurable subsets $A_i$ of
    $ D([0,L],V)$,
    \begin{equation}
      \label{e:gAk1}
      \Plr \bigl[\cap_{i \in F_{\bar k}} \{Y^i \in A_i\},
        {\mathsf H} = {\bar H} \bigr]
      =
      \bigg( \prod_{i \in F_{\bar k}} \Plr
        [Y^i \in A_i|\Ran Y^i \cap E_{\bar k} = \emptyset ] \bigg)
      \Plr [{\mathsf H} = {\bar H}],
    \end{equation}
    where $E_j=O_j\cup V_j$. Let us now analyse the event
    $\mathsf H = \bar H$ in detail. Let $\bar y_j=E_{j+1}\setminus E_j$
    be the vertex explored in the $j$-th step in the history $\bar H$. If
    this vertex is vacant, that is $\bar y_j\in V_{j+1}\setminus V_j$,
    then we know that $\bar y_j \notin \cap_{i<m} \Ran Y^i$. On the other
    hand, if it is occupied, that is $\bar y_j\in O_{j+1}\setminus O_j$,
    then necessarily $\bar y_j\notin \cap_{i\in F_{j+1}}\Ran Y^i$,
    $\bar y_j\in \cup_{i\in T_{j+1}} \Ran Y^i$ and
    $\bar y_j\in \cap_{i\in T_{j+1}\setminus T_j} \Ran Y^i$. Therefore,
    the event $\mathsf H=\bar H$ can be written as
    \begin{equation}
      \begin{split}
        &\bigcap_{j<{\bar k}:V_{j+1}\setminus V_j\neq \emptyset}
        (\bar y_j \notin \cap_{i<m} \Ran Y_i)
        \cap \bigcap_{j<{\bar k}:O_{j+1}\setminus O_j\neq \emptyset}\bigg\{
        (\bar y_j\notin \cap_{i\in F_{j+1}}\Ran Y^i)
        \\&\qquad\qquad \cap
        (\bar y_j\in \cup_{i\in T_{j+1}} \Ran Y^i)
        \cap
        (\bar y_j\in \cap_{i\in T_{j+1}\setminus T_j} \Ran Y^i)\bigg\}
      \end{split}
    \end{equation}
    Collecting the events containing $Y^i$ with $i\in F_{\bar k}$, using
    $F_j\supset F_{\bar k}$ for all $j\le {\bar k}$, this can
    be rearranged as
    \begin{equation}
      \bigcap_{i\in F_{\bar k}} (\Ran Y^i \cap E_{\bar k}=\emptyset)
      \cap f(\bar H,(\Ran Y^i:i \in T_{\bar k})),
    \end{equation}
    where $f$ is some event depending only on $\bar H$ and
    $\Ran Y^i$ with $i\in T_{\bar k}$.
    Inserting this expression for $\mathsf H=\bar H$ into \eqref{e:gAk1}
    and using the independence of
    $Y^i$'s under $\Plr$, the lemma follows.
  \end{proof}

  Lemma~\ref{lm:givenAk} implies that, on the event $\{k < k_\Max\}$,
  \begin{equation}
    \label{e:scb}
    \begin{split}
      \Plr[r_k=-1|\mathcal A_k]
      &\ge 1- \Plr \Big[y_k\notin\bigcap_{i\in \FF_k}\Ran Y^i| \mathcal A_k\Big]
      \\&=1-
      \big(P[H_{\EE_k\cup \{y_k\}}>L|H_{\EE_k}>L]\big)^{|\FF_k|}.
    \end{split}
  \end{equation}

  To bound \eqref{e:scb}, we will use Proposition~\ref{p:c} with $A=\EE_k$,
  $y=y_k\in \partial_e \EE_k$ and $s=r$, see below \eqref{e:finish}. We first
  check its assumptions: Inequality~\eqref{e:scsup} implies that for
  $n \geq c_{K}$, $B(\EE_k,r)\le |\EE_k|(d-1)^r \le \sqrt n$; $k\ge 2$ implies
  $\EE_k\neq \varnothing$ and $\EE_k$ is connected by construction. For $k\ge 2$,
  let $\mathcal P_k=\{\Q_{\EE_k}(y_k,r) \text{ is proper}\}$, see
  \eqref{e:propercone}.

  Take $\varepsilon_u > 0$ such that $u_\star(1+\varepsilon_u)^2 < u$. Since
  $r=7\ld\ld n$ and $L=n^\gamma $, the error term in \eqref{eq:c} is smaller
  than $c_{K}n^{\gamma -1}/\ln n$ which is much smaller than the leading term.
  Hence, for $n \ge c_{u,\gap,K}$, on the event $\mathcal P_k$, we have by
  Proposition~\ref{p:c} that on $\{k < k_\Max\}$,
  \begin{equation}
    \label{e:scc}
    P[H_{\EE_k\cup y_k}>L|H_{\EE_k}>L]\le \exp\Big\{
      \frac{-n^{\gamma -1}(d-2)^2}{d(d-1)(1+\varepsilon_u)}
      \Big\}.
  \end{equation}

  We further define
  $\mathcal G_k = \{|\FF_k|\ge u_\star n^{1-\gamma }(1+\varepsilon_u )^2\}$.
  Observe that both $\mathcal P_k$ and $\mathcal G_k$  are $\mathcal A_k$-measurable.
  Inserting \eqref{e:scc} into \eqref{e:scb} and using the definition
  \eqref{e:ustar} of $u_\star$, we get
  \begin{equation}
    \begin{split}
      \label{e:sccc}
      \Plr [r_k=-1|\mathcal A_k]&\ge
      \Big[1-\exp\Big\{
          \frac{-n^{\gamma -1}(d-2)^2}{d(d-1)
            (1+ \varepsilon_u)} \Big \}
        ^{u_\star n^{1-\gamma }(1+\varepsilon_u )^2}\Big]
      \bbone_{\mathcal G_k\cap \mathcal P_k}\\
      &
      \geq (1 + \delta_u) \frac{d-2}{d-1}
      \bbone_{\mathcal G_k\cap \mathcal P_k},
    \end{split}
  \end{equation}
  for some $\delta_u >0$ on $\{k<k_\Max\}$, provided $n \geq c_{u, \piece}$.

  To proceed we need to control the occurrence of
  $(\mathcal G_k \cap \mathcal P_k)^c$. Observe that, for $k \leq k_\Max$,
  $\mathcal G_k \supset \mathcal G_{k_\Max}$.
  \begin{lemma}
    \label{l:maxnum}
    For every $\sigma > 0$ there exists $c=c_{\sigma ,u,K,\piece}$ such that
    $\Plr [\mathcal G_{k_\Max}^c]\le c n^{-\sigma -1}$.
  \end{lemma}
  \begin{proof}
    We set $M_n=\max_{z\in V} |\{i < m:Y^i\ni z\}|$ and we define the event
    $\bar{\mathcal G} =\{M_n\le  \ld^2 n\}$. Observe that on $\bar{\mathcal G}$
    we have by \eqref{e:scsup}
    \begin{equation}
      \label{e:maxnuma}
      |\FF_k| = m -|\DD_k|\ge m -|\EO_k| \ld^2 n \ge m - d K \ld^3 n,
    \end{equation}
    which is larger than $u_\star n^{1-\gamma }(1+\varepsilon_u)^2$ for
    $n\ge c_{K,u}$. Therefore, $\mathcal G_{k_\Max}^c\subset \bar{\mathcal G}^c $
    for $n\ge c_{K,u}$.

    It remains to bound $\Plr [\bar{\mathcal G}^c]$. First, note that,
    \begin{equation}
      \begin{split}
        \Plr [z \in \Ran Y^i] &= P[H_z \leq L]
        \leq P[N_L \geq n^{2 \piece} ]
        + P \big[\cup_{j=1}^{\lfloor n^{2\piece} \rfloor} \{Y^i_{\tau_j}=z\} \big]
        \\ &
        \overset{\eqref{eq:Wj}}{\leq} c_\piece e^{-c'n^\piece}
        + 2n^{1-\piece} \leq c_\piece n^{\piece-1}.
      \end{split}
    \end{equation}
    Using the bound above and an exponential Chebyshev-type inequality, we
    obtain that
    $\Plr [|\{i<m:Y^i\ni z\}|\ge \ld^2 n] \leq c_{u,\piece}e^{-\ld^2 n}$.
    Summing over $z$ we get
    $\Plr [\bar{\mathcal G}_\zeta^c]\le c_{u,\sigma,\piece}n^{-\sigma -1}$
    and the lemma follows.
  \end{proof}

  We further control the number of steps for which $\mathcal P_k$ does not
  hold. This is the content of the following proposition whose proof is
  postponed to the end of the section.
  \begin{proposition}
    \label{p:mrproper}
    There are at most $c r  K^2$ steps of the algorithm for which
    $\mathcal P_k^c$ occurs.
  \end{proposition}

  To show \eqref{e:finish} we now couple the process $q$ with another process
  $(q'_k)_{k \ge 1}$ which is a random walk with drift such that
  $r'_k=q'_{k+1}-q'_k\in\{-1,d-2\}$ and
  $\Plr [r'_k=-1]= (1+\delta_u)\frac{d-2}{d-1}$ (see \eqref{e:sccc}). This
  implies that $\Elr [r'_k]< \delta'_u$ for a constant
  $\delta '_u<0$.

  The coupling is constructed so that $q'$ can be used as an upper bound
  for $q$. This is done as follows. Let $k\in\{2,\dots,k_\Max -1\}$. On
  $(\mathcal G_k\cap \mathcal P_k)^c$  we take $r_k'$ independent of $q$.
  On $\mathcal G_k \cap \mathcal P_k$ we
  require that $r_k'=\{d-2\}$ whenever $r_k\ge 0$. This is possible
  because $\Plr [r_k\ge 0] \le \Plr[r_k'=\{d-2\}]$
  on $\mathcal G_k \cap \mathcal P_k$, due to \eqref{e:sccc}.
  For $k=1$ and $k\ge k_\Max$, $r'_k$ is independent of $q$.

  As initial condition we take $q'_1=crK^2 (d-2)+d$. Intuitively speaking, this
  gives $q'$ a security zone, for the steps in which $\mathcal P_k$ does not
  hold. With this setting, the Lemma~\ref{l:maxnum} and
  Proposition~\ref{p:mrproper} imply that $q'_k \ge q_k$ for all $k\le k_\Max$
  with probability larger or equal to $1-cn^{-\sigma -1}$.

  We can finally show \eqref{e:finish}. The probability that the algorithm
  finishes due to $|\EV_{k_\Max}|\ge K\ld n$ is bounded from above by
  \begin{equation*}
    \Plr \big[\min_{k<K\ld n} q_k>0\big]
    \le c n^{-1-\sigma }+
    \Plr \big[\min_{k<K\ld n} q'_k>0\big]
    \le c n^{-1-\sigma } + \Plr[q'_{K\ld n}>0]\le c' n^{-1-\sigma }.
  \end{equation*}
  for $K,c'$ large enough, by an easy large deviation estimate for the random
  walk $q'$, which has a negative drift. This finishes the proof of
  Proposition~\ref{p:onecluster} and consequently of Theorem~\ref{t:subcritical}.
\end{proof}

It remains to show Proposition~\ref{p:mrproper}. The next lemma is the key step
in its proof. It controls the tree excess of small (non-necessarily ball-like)
sets.

\begin{lemma}
  \label{l:logsizesets}
  Let $G=(V,\mathcal E)$ with $|V|=n$ satisfy \eqref{a:reg} and \eqref{a:tx1}.
  Then for all $\kappa \ge 1$  and all connected sets $A\subset V$ such that
  $|A|\le \kappa \ld n$
  \begin{equation}
    \tx(A)\le c\kappa^2 =: \alpha (\kappa ).
  \end{equation}
\end{lemma}
\begin{proof}
  Let $G_A=(A,\mathcal E_A)$ be the subgraph of $G$ induced by $A$ and let
  $s=\txconst \ld n$. We call a cycle in $G_A$ \textit{short} if it has no more
  than $2s$ edges, otherwise we call it \textit{long}.

  Roughly speaking, the strategy to prove the lemma will be to erase edges
  belonging to short cycles, then to bound the amount of edges that could be
  still removed after that.

  Fix a short cycle $C$ and let $O_C=\{y\in A: C\subset B_{G_A}(y,s)\}$. Since
  $A$ is connected, either $O_C=A$ or $|O_C|\ge s$. Further, since \eqref{a:tx1}
  holds for $G$, it holds also for $G_A$. Therefore, if $C$, $C'$ are two
  distinct short cycles, then $O_C$ and $O_{C'}$ are disjoint. This implies
  that if $O_C=A$ for a short cycle $C$, then there is only one short cycle in
  $A$ and that in any case, there are at most $|A|/s=\kappa /\txconst$ short
  cycles in $G_A$. From every of these disjoint short cycles we can erase one
  edge and $G_A$ remains connected. Hence we erase at most $\kappa /\txconst$
  edges in this step.

  After this removal, we obtain a graph $G_A'=(A,\mathcal E')$ with girth
  larger than $2s$. Recall from \eqref{e:txcharact} that
  $\tx(A)= |\mathcal E_A|-|A|+1$. Hence, since $G_A$ and $G_A'$ are both
  connected and $G_A'$ was obtained by removing no more than
  $1+\kappa /\txconst$ edges of $G_A$,
  \begin{equation}
    \label{eq:txGA}
    \begin{array}{c}
      \tx(A)\le 1+\kappa / \txconst + \tx(G_A').
    \end{array}
  \end{equation}

  To estimate the last term on the right-hand side, consider the set
  $D=\{(x,y)\in A^2:\dist_{G_A'}(x,y)\ge s\}$. Let
  $\gamma = (x_0,x_1,\dots,x_m,x_0)$ be a (necessarily long) cycle in $G_A'$.
  By removing the edge $\{x_0,x_m\}$, the size of the set $D$ increases at
  least by $(\frac s2 -1)^2$. Indeed, before removing this edge any pair
  $(x_i,x_j)$, for $0\le i\le \frac s2 -1$, $0\le m-j\le \frac s2-1$ was not in
  $D$. However, after removing $\{x_0,x_m\}$, such a pair must be in $D$, since
  otherwise there would be a path in $G_A'$ connecting $x_i$ and $x_j$, not
  passing through the edge $\{x_0, x_m\}$ and having length at most $s$, thus
  there would be a short cycle in $G_A'$ which is not possible. Since the size
  of $D$ is at most $|A|^2$, it is not possible to remove more than
  $|A|^2/(\frac s2-1)^2$ edges from $G_A'$ while keeping it connected. Hence,
  \begin{equation}
    \label{eq:txGA'}
    \tx(G_A') \leq \frac{|A|^2}{(s/2-1)^2},
    \text{ which, for $n \geq c$, is smaller or equal to $16 \kappa^2 /\txconst^2$}.
  \end{equation}
  The claims \eqref{eq:txGA} and \eqref{eq:txGA'} imply that
  $\tx(A)\le 1+(\kappa /\txconst) + (16 \kappa^2/\txconst^2)\le c\kappa^2$, for
  $n \geq c'$. Lemma~\ref{l:logsizesets} now follows by possibly adjusting the
  constants.
\end{proof}

\begin{proof}[Proof of Proposition~\ref{p:mrproper}]
  The algorithm defined in the beginning of the proof of
  Proposition~\ref{p:onecluster} induces a natural random tree structure
  $T=(\EE_{k_\Max},\mathcal E_{T})$. Namely, $\{y,z\}\in \mathcal E_{T}$
  if and only if $z$ was added to the queue during the exploration of $y$
  or vice-versa. By \eqref{e:scsup} we have
  $|\EE_{k_\Max}| \le d K \ld n =: \kappa \ld n$.

  We now finish the proof of Proposition~\ref{p:mrproper} in three lemmas which
  respectively control the number of $y_k$'s for which (i),(iii), or (ii) of
  \eqref{e:propercone} do not hold. It is worth to remark that the arguments in
  these lemmas are purely deterministic and do not depend on the fact that $T$
  results from the previous BFS algorithm.

  We start dealing with the condition (i) of \eqref{e:propercone}. Recall the
  definition of $\alpha(\kappa)$ in Lemma~\ref{l:logsizesets} and that
  $r = (7\ld\ld n) \vee 2$.

  \begin{lemma}
    \label{l:notx}
    Let $B=\{y_k:k< k_\Max \text{ and } \tx(\Q_{\EE_k}(y_k,r))\neq 0\}$. Then,
    for large enough $c$, $|B|\le 2 r \alpha(2\kappa )$ for all $n\ge c$.
  \end{lemma}
  \begin{proof}
    For $i<k_\Max$, we define inductively a sequence $\gamma_i$ of paths in
    $V$ as follows. If $y_i \notin B$ or if $y_i$ belongs to
    $\bigcup_{j<i} \Ran \gamma_j$, then $\gamma_i=\varnothing$. If $y_i\in B$
    and $y_i \notin \bigcup_{j<i} \Ran \gamma_j$, then there is a cycle in
    $\Q_{\EE_i}(y_i,r)$ by definition of $B$ and this cycle is unique by
    Assumption~\eqref{a:tx1} (note that $\tx(\Q_{\EE_i}(y_i,r)) \leq 1$ since
      $r \leq \good \ld n$ for $n \geq c$). In this case, we define $\gamma_i$
    as the unique path in $\Q_{\EE_i}(y_i,r)$ from $y_i$ to this cycle,
    concatenated with the self-avoiding path exploring the whole cycle in one
    of the two directions.

    Set $A=\{1 \leq i <k_\Max:\gamma_i\neq\varnothing\}$ and observe that
    $|\Ran \gamma_i|\le 2r$, for all $i\in A$, and
    $B\subset \bigcup_{j\in A} \Ran \gamma_j$, hence $|B|\le 2r |A|$.

    It remains to show that $|A|\le \alpha (2\kappa )$. Assume the opposite.
    Let $A_0$ be any subset of $A$ with $\alpha(2\kappa)+1$ elements, set
    $R= \EE_{k_\Max}\cup\bigcup_{i\in A_0} \Ran \gamma_i$. Obviously,
    \begin{equation}
      \label{e:sizetT}
      |R|\le |\EE_{k_\Max}|+2r (\alpha (2\kappa )+1)  \le 2\kappa \ld n
      ,\qquad\text{for $n\ge c$}.
    \end{equation}
    We claim that
    \begin{equation}
      \label{e:treexs} \tx(R) \geq \alpha(2\kappa)+1,
    \end{equation}
    which together with \eqref{e:sizetT} contradicts Lemma~\ref{l:logsizesets}
    and hence proves Lemma~\ref{l:notx}. The estimate \eqref{e:treexs} will
    follow if we can show that for all $i \in A_0$, we have
    $\tx(R_{i-1})<\tx(R_i)$, where
    \begin{equation}
      R_0 = \varnothing, \text{ and }
      R_i = \{y_1, \ldots, y_i\} \cup
      \bigcup_{j \in A_0, j \leq i} \Ran \gamma_j, \, 1 \leq i \leq
      |\EE_{k_\Max}|.
    \end{equation}
    If $R_{i-1} \cap \Ran \gamma_i = \varnothing$, then this last claim is
    immediate, because $\gamma_i$ then contains an additional cycle disjoint
    from $R_{i-1}$. Suppose now that
    $R_{i-1} \cap \Ran \gamma_i \neq \varnothing$. We will now find a cycle in
    $R_i$ using an edge that is not already present in the graph induced by
    $R_{i-1}$. This will again imply that $\tx(R_{i-1})<\tx(R_i)$, because by
    removing such an edge the graph remains connected and still has the graph
    induced by $R_{i-1}$ as a subgraph.  To find the cycle, note that
    $y_i \notin R_{i-1}$, because $i$ can be in $A_0$ only if
    $\gamma_i \neq \varnothing$, which by construction can only happen if
    $\Ran \gamma_j \cap  \{y_i\} = \varnothing$ for all $j<i$. Let ${\bar y}_i$
    be the parent of $y_i$ in the tree $T$. Then by construction of $T$,
    ${\bar y}_i \in R_{i-1}$.  We now exhibit a cycle in $R_i$ as follows: we
    start at ${\bar y}_i$ and connect ${\bar y}_i$ to $y_i$. We then follow the
    path $\gamma_i$ from $y_i$ to the first vertex $x_0$ belonging to $R_{i-1}$.
    Since $R_{i-1}$ is connected, we can close our cycle by concatenating our
    path with a non-intersecting path from $x_0$ to ${\bar y}_i$ using only
    vertices in $R_{i-1}$ and therefore not intersecting the previously
    constructed path from ${\bar y}_i$ to $x_0$. We have thus found a cycle in
    $R_i$ using the edge $\{{\bar y}_i, y_i\}$. Since $y_i \notin R_{i-1}$,
    this edge is not present in $R_{i-1}$ and it again follows that
    $\tx(R_{i-1})<\tx(R_i)$. We have therefore proved \eqref{e:treexs} and
    thereby completed the proof of Lemma~\ref{l:notx}.
  \end{proof}

  We now treat condition (iii) of \eqref{e:propercone}.
  \begin{lemma}
    \label{l:noback}
    Let $ B = \{y_k: \text{there is path in $B(\EE_k,r) \setminus \EE_k $ from $y_k$
        to $\EE_k\setminus \bar y_k$}\}$, where $\bar y_k$ is the parent of $y_k$
    in $T$. Then, for $n\ge c$,    $|B|\le 2 r \alpha(2\kappa)$.
  \end{lemma}
  \begin{proof}
    The proof is analogous to the previous one. We define a sequence $\gamma_i$
    of paths in $V$ as follows: If $y_i\in B$ and
    $y_i\notin \bigcup_{j<i} \Ran \gamma_j$, then let $\gamma_i$ be a
    self-avoiding path connecting $y_i$ to $\EE_i = \{y_1, \ldots, y_{i-1}\}$,
    whose first vertex after $y_i$ is in $\Q_{\EE_i}(y_i,r)$ and whose length
    is at most $2r$. Provided $n$ is large, such a path exists for any
    $y_i\in B$, because $\tx(B(y_i,r)) \leq 1$ (cf.~\eqref{a:tx1},
      \eqref{e:subr}). Otherwise, we set $\gamma_i=\varnothing$. Defining $A$,
    $A_0$, $R$ and $R_i$ as in the proof of Lemma~\ref{l:notx}, we can again
    prove \eqref{e:sizetT} and \eqref{e:treexs}. The argument is the same as
    the one used below \eqref{e:treexs}, except that we now only have to
    consider the case $R_{i-1} \cap \Ran \gamma_i \neq \varnothing$.
  \end{proof}

  Finally, we treat condition (ii) of \eqref{e:propercone}.
  \begin{lemma}
    \label{l:twoneighbours}
    With $B = \{y_k: y_k \text{ is neighbour in $G$ of two vertices in $\EE_k
        $}\}$, $|B|\le \alpha(2\kappa)$.
  \end{lemma}
  \begin{proof}
    In this case we can remove the edges between $y_i$ and $\EE_i$ which are
    not in $\mathcal E_T$ from the subgraph of $G$ induced by
    $\EE_{k_\Max}$ while
    keeping it connected. Since $|\EE_{k_\Max}|\le \kappa \ld n$,
    Lemma~\ref{l:logsizesets} implies the result.
  \end{proof}

  Proposition~\ref{p:mrproper} follows easily from last three lemmas.
\end{proof}
\section{Super-critical regime} \label{s:super}

In this section we prove Theorem~\ref{th:usmall} stating the existence of
a giant component for $u$ smaller than $u_\star$. Since the proof is
rather lengthy we first briefly outline its strategy. The strategy is
inspired by the methods used for Bernoulli percolation. It has two major
parts: First, we consider a modification of the piecewise independent
measure and for such modification, we prove the existence of a sufficient
amount of mesoscopic clusters even under a slightly increased value
$u_n > u$. Second, by decreasing $u_n$ back to the original value $u$, we
prove that these clusters are connected by sprinkling. Both these parts
are however rather non-trivial, due to the presence of the dependence.

To construct the mesoscopic clusters, we first show in
Section~\ref{ss:intaprox} that the vacant set left by segments
(see~\eqref{e:vsls} and \eqref{eq:xi} below) on $G$ locally resembles
the vacant set of random interlacement on the $d$-regular tree
$\mathbb T_d=(\mathbb V_d, \mathbb E_d)$
(Proposition~\ref{pr:dominat}). The behaviour of
the random interlacement on $\mathbb T_d$ is well known \cite{Tei09} and
its clusters can be controlled in terms of a particular branching process.
This branching process will be super-critical for $u$'s considered in this
section.

The control by the branching process allows us to construct a sufficient
amount of mesoscopic clusters for the vacant set left by segments. Since
we are looking for a lower bound on the vacant set, we however cannot
ignore the bridges as in the previous section. In Section~\ref{ss:robust}
we show that the mesoscopic clusters of the vacant set left by segments
are robust and the addition of the bridges to the picture typically does not
destroy them, see Proposition~\ref{pr:lotgood}.

Finally, in Section~\ref{ss:proof} we use a sprinkling well adapted to
our model to prove Theorem~\ref{th:usmall}. As discussed in the
introduction, in this sprinkling we erase randomly some segments,
possibly in the middle of the trajectory. This can possibly disconnect
the trajectory. Therefore, to be able to extract a nearest-neighbour path
in the end, we must add many additional bridges to the picture; the
robustness proven in Section~\ref{ss:robust} must take them in consideration.

\subsection{Preliminaries}
\label{ss:prels}
We establish first the following technical consequence of
assumption~\eqref{a:tx1} which will be needed later in this section.
\begin{lemma}
  \label{l:txzero}
  Let $G=(V,\mathcal E)$ be a graph satisfying assumptions~\eqref{a:reg},
  \eqref{a:tx1}. Set $R=\lfloor \good \ld n\rfloor$ to be the radius of
  \eqref{a:tx1} and $r=R-\Delta$, for some $\Delta\in \{1,2,\dots\}$. Then
  \begin{equation}
    \label{e:txzero}
    |\{x:\tx(B(x,r))=0\}|\ge (1-(d-1)^{-\Delta })|V|.
  \end{equation}
\end{lemma}
\begin{proof}
  Let us consider the sets
  \begin{equation}
    A=\{x: \tx(B(x,r))=1\}\quad\text{and}\quad\tilde A=\{x: \tx(B(x,R))=1\}.
  \end{equation}
  We first study the structure of the graph $G$
  restricted to $\tilde A$. For every point $x\in \tilde A$, there is exactly
  one cycle in $B(x,R)$. This cycle should contain at most $2R+1$ vertices,
  otherwise it cannot be contained there. If $C\subset V$ is such a cycle, we
  define $\ell_C=\diam C=\lfloor |C|/2\rfloor$ and
  $N_C(s)=\{x:C\subset B(x,s)\}$ for $s \leq R$. It is easy to see that
  \eqref{a:tx1} implies $N_C(R)=B(C,R-\ell_C)$.

  We now prove the following claim: the subgraph of $G$ induced by $N_C(R)$ is
  composed by the cycle $C$ with disjoint trees rooted at its vertices with
  depth $R-\ell_C$. Indeed, the graphs attached to every $y\in C$ should be
  trees because otherwise \eqref{a:tx1} cannot hold. To see that they must be
  disjoint, suppose that they are not, that is there are two points $y,z\in C$
  such that $y$ and $z$ are connected in $N_C(R)\setminus C$. This connection
  must be shorter than $2(R-\ell_C)+1$. Joining this connection with the
  shortest connection of $y$ and $z$ in $C$, which is shorter than $\ell_C $,
  we obtain a cycle different from $C$ of length at most $2R-\ell_C +1$, which
  is contained in $B(y,R)$. This, however, contradicts \eqref{a:tx1}.

  The claim proved in the above paragraph implies that
  \begin{align}
    \label{e:cycleneigh}
    &|N_C(R)|= |C|+(d-2)|C|\big(1+\dots+(d-1)^{R-\ell_C-1}\big)
    = |C|(d-1)^{R-\ell_C},\\
    \intertext{Since $N_C(r)$ is either empty or has a similar structure as
      $N_C(R)$}
    \label{e:cycleneigh2}
    &|N_C(R)\setminus N_C(r)| \geq |C|(d-1)^{r-\ell_C}((d-1)^\Delta-1)
    \geq |N_C(r)| ((d-1)^\Delta-1).
  \end{align}
  By our assumptions,  the set $\tilde A$ can be written as a disjoint union
  $\tilde A=\cup_{i=1}^M N_{C_i}(R)$ for some $M\in \mathbb N$ and cycles
  $C_1,\dots,C_M$. Similarly $A=\cup_{i\in U} N_{C_i}(r)$ for some
  $U\subset\{1,\dots,M\}$ which contains indices of cycles shorter than $2r+1$.
  Therefore, using \eqref{e:cycleneigh} and \eqref{e:cycleneigh2},
  \begin{equation}
    \begin{split}
      |V|&\ge |\tilde A|\ge  \sum_{i\in U}|N_{C_i}(R)|
      = \sum_{i\in U}|N_{C_i}(r)|+|N_{C_i}(R)\setminus N_{C_i}(r)|
      \\&\geq\sum_{i\in U}|N_{C_i}(r)|(d-1)^\Delta = (d-1)^\Delta |A|.
    \end{split}
  \end{equation}
  Hence $|A|\le |V|/(d-1)^\Delta$ and thus
  $|A^c| = |\{x: \tx(B(x,r)) =0\}| \ge (1-(d-1)^{-\Delta}) |V|$.
\end{proof}

We now collect some notation used in the proof of
Theorem~\ref{th:usmall}.
In what follows, we write $\mathbb V=\mathbb V_d$ for the set of vertices
of the tree $\mathbb T_d$ and denote by $o$ its root.
In order to describe the clusters of random interlacement on the
tree $\mathbb T_d$ we define the
function $f: \mathbb{V} \rightarrow \mathbb{R}$ as
\begin{equation}
  \label{eq:fz}
  f(z) = \begin{cases}
    \frac{(d-2)^2}{d(d-1)} & \text{ if $z \neq o$,}\\
    \frac{d-2}{d-1} & \text{ for $z = o$}.
  \end{cases}
\end{equation}
We let $\Ber_u$ stand for the law on $\{0,1\}^\mathbb{V}$ which
associates to the vertices $z \in \mathbb{V}$ independent Bernoulli
random variables with success probability $e^{-uf(z)}$. The following
result of \cite{Tei09} provides the connection between this Bernoulli
percolation and random interlacement on $\mathbb T_d$. This result will
not be used in this paper, but is quoted here in order to provide the
natural interpretation of the model that we have just introduced.
\begin{theorem}[\cite{Tei09},Theorem~5.1 and (5.7)]
  \label{l:interbr}
  The connected component $\comp_o \subseteq {\mathbb V}$ containing the root $o$
  has the same law under $\mathbb Q_u$ characterised by \eqref{e:Qchar} as under $\Ber_u$.
\end{theorem}

Note that, under the law $\Ber_u$ the cluster $\comp_o \subseteq {\mathbb V}$ can
be regarded as a branching process, where the ancestor in generation $0$
is born with probability $\exp \left\{ - u \frac{d-2}{d-1} \right\}$, and
with binomial offspring distribution with parameters $d-1$ and $p_u$, where
\begin{equation}
  \label{e:pu}
  p_u = \exp\Big\{- {\frac{u(d-2)^2}{d(d-1)}}\Big\}.
\end{equation}
In order to deal with this branching process it is useful to define
the expected number of offsprings as well as its logarithm in
base $d-1$:
\begin{equation}
  \label{eq:muvu}
  m_u = (d-1) p_u,\quad\text{and}\quad
  v_u = \ld m_u  = 1- \frac{u(d-2)^2}{d(d-1)\ln(d-1)} = 1 - \frac{u}{u_\star} .
\end{equation}
Observe that for $u< u_\star$, we have $m_u>1 + c_u$ and $v_u \in (c_u,1)$,
for a constant $c_u >0$.

For $u \in (0, u_\star)$, it will
be convenient to fix a small $\epsilon=\epsilon(u)>0$ such that the slightly
increased intensity $u(1+\epsilon)$ satisfies
\begin{equation}
  \label{eq:epsilon}
  u(1+\epsilon) <  \frac{u+u_\star}{2} \quad \text{and}\quad
  \frac{1}{4} + \frac{11}{4} \epsilon < \frac{u_\star}{2(u+u_\star)},
\end{equation}
which by \eqref{eq:muvu} implies that
\begin{align}
  \label{e:epsv} \tfrac{3}{2}\vv_{u(1+\epsilon)}
  - \tfrac{5}{4}\vv_{u(1-\epsilon)} =
  \frac{1}{4} - \frac{u}{u_\star} \Big( \frac{1}{4}
    + \frac{11}{4} \epsilon \Big) >0.
\end{align}
Finally, we define
\begin{equation}
  \label{e:newgamma}
  \beta = \frac{\good}{100} < \frac{1}{100},
  \qquad\text{and} \qquad
  \piece = \piece(u)= \frac{\vv_{u(1+\epsilon)}\beta}{2} <  \frac{\beta }{2}.
\end{equation}
When $u\ge u_\star$ (which we allow in Section~\ref{ss:intaprox}) we only require a weaker condition
\begin{align}
  \label{e:beta}
  0 < \piece <\beta=\good /{100} \leq  1/ {100}.
\end{align}

We recall from Section~\ref{s:partsbridges} the segments $Y^i$ with length
$L=n^\piece$ constructed on the
probability space $(\Omega ,\Plr)$. We define
\begin{align}
  \label{e:M} M_u = \lceil un/ (L+\ell) \rceil, \text{ for } u>0.
\end{align}
We consider the \textit{vacant set left by segments}
$\bV^u = V \setminus \cup_{0\le i< M_u} \Ran Y^i$,
and the corresponding random configuration $\xi_u \in \{0,1\}^V$ defined by
\begin{equation}
  \label{eq:xi}
  \xi_u = 1_{\bV^u} =  1\big\{V \setminus \cup_{0\le i< M_u} \Ran Y^i \big\}.
\end{equation}

Given a configuration $\eta \in \{0,1\}^V$
or $\eta \in \{0,1\}^{\mathbb V}$, let
\begin{equation}
  \parbox{0.8\textwidth}{$\comp_x(\eta)$ be the connected component of
    $\supp \eta := \{y \in V: \eta(y) = 1\}$ containing the vertex $x$,}
\end{equation}
and $\comp_\Max (\eta)$ be the largest such component.

For any fixed vertex $y \in V$, we define
\begin{equation}
  \label{e:BB'}
  B_y=B(y, \beta \ld n), \qquad B_y' = B(y, 5 \beta \ld n) \subseteq V.
\end{equation}
We also set
\begin{equation}
  \mathbb B = B(o, \beta \ld n),\qquad
  \mathbb B'= B(o, 5 \beta \ld n) \subset \mathbb V.
\end{equation}
If $y$ has a tree-like neighbourhood of radius $5\beta \ld n$,
\begin{align}
  \label{e:phi}
  \text{there is a graph isomorphism $\phi : B'_y \to \mathbb B'$ such that
    $\phi(y)=o$.}
\end{align}
In order to make the formulas less complicated, we will mostly identify
the vertices of $G$ and of $\mathbb T_d$ linked by this isomorphism and omit
$\phi $ from the notation. The vertex $y$ is always given by the context.
In particular, for $z\in B_y$ we define $f(z)=f(\phi (z))$.

\subsection{Approximation by random interlacements}
\label{ss:intaprox}

With all notation in place, we can now approach the proof of
Theorem~\ref{th:usmall}. In this section we show that, provided
$\tx(B_y')=0$, the component of the set
$\bV^u\cap B_y =\supp \xi_u \cap B_y$ containing the centre $y$ of $B_y$
can, up to a small error, be controlled from above and from below by the
branching process introduced in \eqref{eq:fz} and below.
Note that for the next proposition it is not necessary to assume that
$u<u_\star$.

\begin{proposition}
  \label{pr:dominat}
  Assume \eqref{a:reg} and \eqref{a:gap}, and suppose that $\tx(B_y')=0$.
  Then for any $u \geq 0$, $\epsilon \in (0,1)$, we can construct random
  sets $\comp^{u(1+\epsilon)}$ and
  $\comp^{u(1-\epsilon)} \subseteq {\mathbb V}$ distributed as
  $\comp_o$ under $\Ber_{u(1+\epsilon)}$ and
  $\Ber_{u(1-\epsilon)}$ such that
  \begin{align}
    \Plr
    \left[
      \begin{array}{c}
        \comp^{u(1+\epsilon)}\cap \mathbb B \subseteq \comp_y
        (1_{B_y} \cdot \xi_{\hat u}) \subseteq
        \comp^{u(1-\epsilon)},
        \\ \text{ for all } {\hat u} \in \bigl( u (1-
            \tfrac{\epsilon}{2} ), u (1+ \tfrac{\epsilon}{2} ) \bigr)
      \end{array}
    \right]  \geq 1- c_{ \piece,  u, \epsilon} n^{-2\beta}. \nonumber
  \end{align}
\end{proposition}

\begin{remark}
  Proposition~\ref{pr:dominat} can also be interpreted as a control of the
  component of $\bV^u\cap B_y$ by random interlacement on $\mathbb T_d$.
  Indeed, due to Theorem~\ref{l:interbr}, the sets
  $\comp^{u(1\pm \epsilon)}$ have the same distribution under the
  Bernoulli measure $\Ber_{u(1\pm \varepsilon )}$ as under the random
  interlacement measure $\mathbb Q_{u(1\pm \varepsilon )}$.
\end{remark}

\begin{proof}
  Throughout this proof, we write $B$, $B'$ rather than $B_y$, $B'_y$.
  Our strategy resembles the proof of Theorem~5.1 in \cite{Tei09}. We
  first poissonise the number of trajectories entering in the
  definition of the configurations $\xi_{\hat u} $ for
  ${\hat u} \in (u(1-\epsilon/2), u(1+ \epsilon/2))$, see
  \eqref{eq:xi}. To this end we introduce two independent
  Poisson random variables  $\Pi_+$ and $\Pi_-$ defined on
  $(\Omega ,\Plr)$ with parameters
  $un^{1-\piece }(1+\frac{3\varepsilon}{4})$ and
  $un^{1-\piece }(1-\frac{3\varepsilon}{4})$, independent of all
  previously introduced random variables. We are going to compare
  $\xi_{\hat u}$ with the configurations $\xi_-$, $\xi_+$ defined by
  \begin{align}
    \xi_- = 1\big\{V \setminus \cup_{i<\Pi_-} \Ran Y^i  \big\} \textrm{ and }
    \xi_+ = 1\big\{V \setminus \cup_{i<\Pi_+} \Ran Y^i  \big\}.
  \end{align}
  Clearly, by a large deviation argument, since $M_u=un^{1-\piece }(1+o(1))$,
  \begin{equation}
    \label{eq:ldpxipm}
    \begin{split}
      &\Plr [\xi_+ \leq \xi_{\hat u} \leq \xi_- \text{, for all }
        {\hat u} \in (u(1- \epsilon/2),  u(1+ \epsilon/2))] \\
      &\quad \ge \Plr[\Pi_-\le M_{u(1-\epsilon/2)} <
        M_{u(1+\epsilon/2)} \le \Pi_+]
      \geq 1- c_{ \piece, u, \epsilon}\exp\{-cn^{1-\piece}\}.
    \end{split}
  \end{equation}
  \begin{figure} [ht]
    \begin{center}
      \includegraphics[angle=0, width=0.3\textwidth]{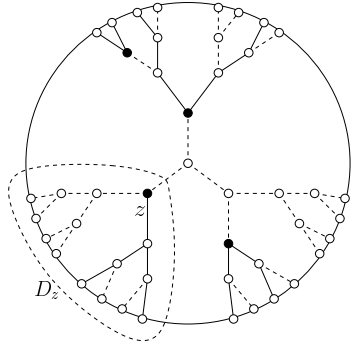}\\
      \caption{The trace left by some pieces in $B$ and the set $D_z$. The
        black and white circles represent respectively the values $0$ and $1$
        for the variables $\tilde \xi_\pm (z)$ defined in \eqref{eq:Zpmz}.}
      \label{fig:Dz}
    \end{center}
  \end{figure}

  Next, we will dominate $\xi_{\pm}$ from above and from below by a collection
  of i.i.d.\ Bernoulli random variables. For every $z \in B$, we define the set
  $D_{z}$ as the set of descendants of $z$ in $B$, that is
  \begin{align*}
    D_z = \{z' \in B:
      \text{ any path from $z'$ to $y$ meets either $z$ or some vertex in $B^c$}\},
  \end{align*}
  see Figure~\ref{fig:Dz}. Consider the following disjoint subsets of
  $D([0,L],V)$, cf. \cite{Tei09} (5.3),
  \begin{equation}
    \label{eq:Wy}
    \begin{split}
      W_{z} &= \big\{ Y \in  D([0,L],V): z \in \Ran Y  \cap B \subset D_{z}
        \big\}, \, z \in B,\\
      W &= \big\{ Y \in  D([0,L],V): \Ran Y  \cap B \neq \varnothing \big\}
      \setminus \textstyle{\bigcup\limits_{z \in B}} W_{z}.
    \end{split}
  \end{equation}
  In particular, all trajectories in $W$ must enter $B$, then exit $B'$ and enter $B$
  again, see Figure~\ref{fig:Dz}. We define the random configurations
  $\tilde \xi_+$, $\tilde \xi_-\in \{0,1\}^B$ on $(\Omega ,\Plr)$ by
  \begin{equation}
    \label{eq:Zpmz}
    \begin{split}
      \tilde \xi_-(z) &=1\{Y^i\notin W_z,\,\forall i<\Pi_-\},\quad z\in B,\\
      \tilde \xi_+(z)&=1\{Y^i\notin W_z,\,\forall i<\Pi_+\},\quad z\in B,
    \end{split}
  \end{equation}
  see Figure~\ref{fig:Dz} again. Since the sets $W_z$ are
  disjoint for distinct $z$'s, the variables $\xi_{+}(z)$ will be
  independent for distinct $z$'s due to the
  Poissonian character of $\Pi_+$ (the same
    will also hold for $\xi_-(x)$), see Lemma~\ref{l:perc}
  below. We further consider the random variable
  \begin{equation}
    \mathcal Z = 1\{ Y^i\notin W,\, \forall i< \Pi_+\}.
  \end{equation}
  Observe that
  \begin{equation}
    \label{eq:Zzero}
    \begin{array}{c}
      \text{on the event }\{\mathcal Z=0\}, \,
      \mathcal{C}_y(1_B \cdot \xi_-) = \mathcal{C}_y(\tilde \xi_-)
      \textrm{ and } \mathcal{C}_y(1_B \cdot \xi_+) = \mathcal{C}_y(\tilde \xi_+).
    \end{array}
  \end{equation}
  The following lemma shows that the laws of $\tilde \xi_-$ and
  $\tilde \xi_+$ on $\{0,1\}^B$ are comparable with the laws
  $\Ber_{u(1-3\epsilon/4)}$ and $\Ber_{u(1+3\epsilon/4)}$
  of Bernoulli percolation introduced above, restricted to
  $\{0,1\}^{\mathbb B}$ (which by assumption can be identified with
    $\{0,1\}^B$).

  \begin{lemma} \label{l:perc}
    For $\pm$ denoting either $+$ or $-$, the events
    $(\{ \tilde \xi_\pm(z) = 1 \})_{z \in B}$ are independent and satisfy
    \begin{equation}
      \label{eq:Sigmaxi}
      \begin{split}
        &\Big| \Plr[\tilde \xi_+(z) = 1]
        - e^{-u(1 + 3\epsilon/4) f (z)} \Big|
        \leq c_{ \piece, u} n^{-\piece/3}, \text{ and}\\
        &\Big| \Plr[\tilde \xi_-(z) = 1]
        - e^{-u(1 - 3\epsilon/4) f  (z)} \Big|
        \leq c_{ \piece, u} n^{-\piece/3}.
      \end{split}
    \end{equation}
  \end{lemma}
  Before we prove this lemma, we complete the proof of
  Proposition~\ref{pr:dominat}. For $z\in\mathbb V $ let
  \begin{equation}
    \begin{split}
      &p_z^+ = e^{-u(1 + \epsilon) f (z)}, \qquad
      q^+_z = \Plr[\tilde \xi_+(\phi^{-1}(z)) = 1],\\
      &p_z^- = e^{-u(1 - \epsilon) f (z)},\qquad
      q^-_z = \Plr[\tilde \xi_-(\phi^{-1}(z)) = 1].
    \end{split}
  \end{equation}
  Then Lemma~\ref{l:perc} implies that for $n \geq c'_{ \piece,u ,\epsilon}$,
  \begin{align}
    p^+_z \leq q^+_z \leq q^-_z \leq p^-_z, \text{ for all } z \in {\mathbb B}.
  \end{align}

  We now construct the sets $\comp^{u(1 \pm \epsilon)}$ as stated in
  the proposition by adding to our
  probability space $(\Omega, \Plr)$ a collection
  $\{U^+_z, U^-_z\}_{z \in {\mathbb V}}$ of independent Bernoulli-distributed
  random variables which will fine tune the values $q^\pm_z$ to match the
  $p^\pm_z$'s:
  \begin{itemize}
    \item For every $z \in {\mathbb V} \setminus {\mathbb B}$, the
    parameters of $U^+_{z}$ and $U^-_{z}$ are $p^+_z$ and $p^-_z$.
    \item For
    $z \in {\mathbb B}$, $U^+_{z}$ and $U^-_{z}$ have parameters $p_z^+/ q^+_z$
    and $(p_z^- - q^-_z)/(1-q^-_z)$.
  \end{itemize}
  We then define $\comp^{u(1 \pm \epsilon)}$ by
  \begin{align}
    \comp^{u(1+\epsilon)}
    &= \comp_o \Bigl(
      \bigl( ({\tilde \xi}_+ (z) \wedge U^+_z)
        1_{z \in {\mathbb B}}
        + U^+_z 1_{z \in {\mathbb V} \setminus {\mathbb B}}
        \bigr)_{z \in {\mathbb V}} \Bigr),
    \\ \comp^{u(1-\epsilon)}
    &= \comp_o
    \Bigl( \bigl( ({\tilde \xi}_- (z) \vee U^-_z)
        1_{z \in {\mathbb B}}
        + U^-_z 1_{z \in {\mathbb V} \setminus {\mathbb B}}
        \bigr)_{z \in {\mathbb V}} \Bigr).
  \end{align}
  Note that we then have
  \begin{equation}
    \comp^{u(1+\epsilon)} \subset
    \comp_y ({\tilde \xi}_+)
    \subseteq \comp_y ({\tilde \xi}_- )
    \subset \comp^{u(1-\epsilon)}.
  \end{equation}
  Since the variables $\tilde \xi_\pm$ and $U^\pm_z$ are all independent
  (cf.~Lemma~\ref{l:perc}), it is elementary to check
  that the laws of $\comp^{u(1\pm \epsilon)}$ agree with those of
  $\comp_o$ under $\Ber_{u(1\pm \epsilon)}$ for large $n$. Moreover, we have
  by \eqref{eq:Zzero} that on the event
  $\{{\mathcal Z}=0\} \cap \{\xi_+ \leq \xi_{\hat u} \leq \xi_-\}$,
  \begin{align}
    \comp^{u(1+\epsilon)}\cap \mathbb B \subseteq  \comp_y(1_B \cdot
      \xi_+) \subseteq \comp_y (1_B \cdot \xi_{\hat u}) \subseteq
    \comp_y(1_B \cdot \xi_-) \subseteq
    \comp^{u(1-\epsilon)}\cap \mathbb B.
  \end{align}

  Since we already know the bound \eqref{eq:ldpxipm}, it thus only remains to
  prove that
  \begin{equation}
    \Plr[\mathcal Z \neq 0]
    \leq c_{u} n^{-2\beta}.
  \end{equation}
  If $\mathcal Z \neq 0$, there is an $i< \Pi_+$ such that  $Y^i \in W$. Since
  $\tx(B') = 0$, if $Y^i\in W$, then  there exist  times $t_1 < t_2 < t_3$ such
  that $Y^i_{t_1} \in B$, $Y^i_{t_2} \not \in B'$ and again $Y^i_{t_3} \in B$,
  see \eqref{eq:Wy}.   Using the strong Markov property we thus get
  \begin{equation}
    \label{e:YjinW}
    \Plr[Y^i\in W]=P^L[W]  \leq P^L[H_{B} < L] \sup_{w \in V
      \setminus B'} P^L_w[H_{B} < L].
  \end{equation}
  Note that by stationarity of the random walk with respect to the uniform
  distribution,
  \begin{align}
    P^L[H_{B} < L] \leq E^L \Bigl[ \sum_{k=0}^{N_L} 1\{{\hat X}_k \in B\}
      \Bigr] = E^L[N_L]|B|/n = L|B|/n.
  \end{align}
  Using Lemma~\ref{lm:mix} for the second term on the right-hand side of
  \eqref{e:YjinW}, we hence obtain
  \begin{equation}
    \begin{split}
      \Plr [Y^i\in W]
      \le c_{ \piece} n^{\piece + \beta - 1}
      (c n^{\piece -4 \beta } + e^{-c'n^\piece })
      \le c_{ \piece} n^{2\piece -3\beta -1}.
    \end{split}
  \end{equation}
  Since $\mathcal Z$ has Poisson distribution with parameter
  $un^{1-\piece}(1+ 3 \varepsilon/4) P^L[W]$, using $\gamma <\beta $,
  \begin{equation}
    \label{eq:Znotzero}
    \Plr [\mathcal Z \neq 0]
    =
    1-\exp\{-u (1+ 3 \epsilon/ 4) n^{1-\piece} P^L[W]\}
    \leq c_{ u} n^{-2\beta}.
  \end{equation}
  Up to Lemma~\ref{l:perc}, this completes the proof of Proposition~\ref{pr:dominat}.
\end{proof}

\begin{proof}[Proof of Lemma~\ref{l:perc}.]
  Since the sets $W_{z}$, $z \in B$, are mutually disjoint and $\Pi_-$ is
  Poisson distributed, independent of the ${Y^i}'s$, the random variables
  $|\{Y^i\in W_z:i<\Pi_- \}|$, $z\in B$ are independent Poisson random
  variables with parameters
  $un^{1-\piece }(1 -  3\varepsilon/4)\Plr[Y^0\in W_z]$. In particular,
  since $\tilde \xi_-(z)= 1\{ |\{Y^i\in W_z:i<\Pi_- \}| = 0\}$, the events
  $(\{ \tilde \xi_-(z) = 1 \})_{z \in B}$ are indeed independent. Moreover, we
  have
  \begin{equation}
    \label{e:xixi}
    \Plr[\tilde \xi_-(z)=1]
    =\Plr[|\{Y^i\in W_z:i<\Pi_-\}|=0]
    =e^{-un^{1-\piece }({1- 3 \varepsilon/4})\Plr[Y^0\in W_z]}.
  \end{equation}
  The above arguments apply also to $\tilde \xi_+$ and yield the analogous
  claims. Since the function $e^{-x}$ is Lipschitz with constant $1$ on
  $[0,\infty)$, we see that the left hand sides of \eqref{eq:Sigmaxi} are
  bounded by $c_u|n^{1-\piece } P^L[W_z] - f(z)|$. It is therefore
  sufficient to prove that
  \begin{align}
    \label{e:perc}
    |n^{1-\piece } P^L[W_z] - f (z)| \leq c_{ \piece} n^{-\piece/3},
    \text{ for any } z \in B.
  \end{align}
  Note the relation of this approximation with \eqref{eq:approximate}.

  Conditioning on the number of jumps $N_L$ made by $X$ in the time interval
  $[0,L]$ and using independence of $N_L$ and the discrete skeleton  $\hat X$,
  we have
  \begin{align}
    \label{e:perc00}
    P^L[W_z] &= \sum_{r \geq 0} P[N_L = r]
    P[z \in \{{\hat X}_0, \ldots, {\hat X}_r\} \cap B \subset D_z].
  \end{align}
  Let us fix any $r$ such that
  $\lfloor 2^{-1} n^\piece\rfloor \leq r \leq \lfloor 2n^\piece\rfloor$ and
  throughout the rest of this proof write $A_z = B \setminus D_z$. Summing over
  all possible times $k$ when $\hat X$ first visits $z$ and applying the simple
  Markov property, we obtain
  \begin{align}
    P[z \in \{{\hat X}_0, \ldots, {\hat X}_r\} \cap B \subset D_z] = \sum_{0
      \leq k \leq r} P[{\hat H}_{A_z \cup \{z\}} =k, {\hat X}_k =z] P_z [ {\hat
        H}_{A_z} > r-k ],
  \end{align}
  where we are using the convention that ${\hat H}_\varnothing = \infty$, which
  occurs in the last probability when $z=y$, in which case
  $A_y = B \setminus D_y = \varnothing$. Using reversibility of $\hat X$ with
  respect to the uniform distribution on the first probability in the product,
  we deduce that
  \begin{align}
    \label{e:perc0}
    P[z \in \{{\hat X}_0, \ldots, {\hat X}_r\} \cap B \subset D_z] =
    \frac{1}{n} \sum_{0 \leq k \leq r}  P_z[{\hat H}^+_{A_z \cup \{z\}} > k ]
    P_z [ {\hat H}_{A_z} > r-k ].
  \end{align}

  We now claim that the following estimates hold uniformly for all
  $n^{\piece/2} \leq k \leq r-n^{\piece/2}$:
  \begin{align}
    \label{e:perc1}
    &\sup_{z' \in V \setminus B'} P_{z'} [{\hat H}_B \leq k]
    \leq c_{ \piece} n^{-3 \beta} \text{ and } \sup_{z' \in B'}
    P_{z'} [{\hat H}_{V \setminus B'} \geq k] \leq c_{\piece} (\log n) n^{-\piece/2}.
  \end{align}
  Indeed, the first estimate follows from Lemma~\ref{lm:mix} and the choice of
  $\piece < \beta$ in \eqref{e:newgamma}, while the second estimate in
  \eqref{e:perc1} follows from the Chebyshev inequality and the bound
  $E_{z'}[{\hat H}_{V \setminus B'}] \leq c_\beta \log n$, which is an
  elementary estimate on the expected amount of time it takes a one-dimensional
  biased random walk to reach the level $5 \beta \ld n$.

  For any $k$ as above,
  it follows from \eqref{e:perc1} and the strong Markov property applied at
  time ${\hat H}_{V \setminus B'}$ that
  \begin{align}
    \bigl| P_z[{\hat H}^+_{A_z \cup \{z\}} > k ] -
    P_z[ {\hat H}^+_{A_z \cup \{z\}} > {\hat H}_{V \setminus B'}] \bigr|
    \leq c_{ \piece} (n^{-3 \beta} + n^{-\piece/3}).
  \end{align}

  We now relate the second probability on the left-hand side to the escape
  probability to infinity from $\mathbb B \cup \{z\}$ for the random walk on
  the tree ${\mathbb T}_d$. By the strong Markov property applied at
  time ${\hat H}_{V \setminus B'}$, we have (identifying $z$ and $A_z$
    with corresponding objects on $\mathbb T_d$)
  \begin{align}
    P^{{\mathbb T}_d}_{z}[ {\hat H}^+_{A_z \cup \{z\}} = \infty ]
    \leq P_{z}[ {\hat H}^+_{A_z \cup \{z\}} > {\hat H}_{V \setminus B'}]
    \leq \frac{P^{{\mathbb T}_d}_{z}
      [ {\hat H}^+_{A_z \cup \{z\}} = \infty ]}
    {\inf_{z' \in {\mathbb V} \setminus \mathbb B'}
      P^{{\mathbb T}_d}_{z'} [H_{\mathbb B} = \infty]}.
  \end{align}
  By another elementary estimate on the biased random walk
  $(\dist(B,{\hat X}_n))_{n \geq 0}$, we have
  \begin{align}
    \inf_{z' \in {\mathbb V} \setminus \mathbb B'}
    P^{{\mathbb T}_d}_{z'} [H_{\mathbb B} = \infty] \geq 1 - c n^{-4 \beta}.
  \end{align}
  Collecting the above estimates  we obtain that for any
  $n^{\piece/2} \leq k \leq r-n^{\piece/2}$,
  \begin{align}
    \label{e:perc3} \Bigl| P_z[{\hat H}^+_{A_z \cup \{z\}} > k ]
    - P^{{\mathbb T}_d}_{z}[ {\hat H}^+_{A_z \cup \{z\}} = \infty ] \Bigr|
    \leq c_{ \piece} n^{-\piece/3},
  \end{align}
  and the same computations with ${\hat H}^+_{A_z \cup \{z\}}$ replaced by
  ${\hat H}_{A_z}$ show that
  \begin{align}
    \label{e:perc4} \Bigl| P_z[{\hat H}_{A_z} > r-k ]
    - P^{{\mathbb T}_d}_{z}[ {\hat H}_{A_z} = \infty ]
    \Bigr| \leq c_{ \piece} n^{-\piece/3}.
  \end{align}

  With estimates on one-dimensional random walk, we can compute the escape
  probabilities for random walk on the infinite tree explicitly. Indeed, by
  applying the simple Markov property at time $1$, and then computing the
  probability that a nearest-neighbour biased random walk on the integers does
  not return to $0$ when started at $1$, we obtain
  \begin{align}
    P^{{\mathbb T}_d}_{z}[ {\hat H}^+_{A_z \cup \{z\}} = \infty ]
    =
    \begin{cases} \frac{d-1}{d}
      \times \frac{d-2}{d-1} = \frac{d-2}{d},
      & \textrm{if } z \neq y,\\
      \frac{d-2}{d-1}, & \textrm{if } z=y,
    \end{cases}
  \end{align}
  and similarly, by the convention that ${\hat H}_{\varnothing} = \infty$,
  \begin{align}
    P^{{\mathbb T}_d}_{z}[ {\hat H}_{A_z} = \infty ]
    =
    \begin{cases}
      \frac{d-2}{d-1}, & \textrm{if } z \neq y,\\
      1, & \textrm{if } z=y.
    \end{cases}
  \end{align}
  Note that in both cases, the product of the two probabilities just
  computed equals $f (z)$, cf.~\eqref{eq:fz}. Inserting the
  estimates \eqref{e:perc3} and \eqref{e:perc4} into \eqref{e:perc0}, we
  therefore infer that for any $r$ such that
  $2^{-1} n^\piece \leq r \leq 2n^\piece$,
  \begin{align}
    \big| P[z \in \{{\hat X}_0, \ldots, {\hat X}_r\} \cap B \subset D_z]
    - {r}{n^{-1}}f (z)  \big|
    \leq c_{ \piece} n^{-1+(2\piece/3) }.
  \end{align}
  Using this estimate and the large deviation bound on $N_L$ from
  \eqref{eq:Wj} in \eqref{e:perc00}, we obtain that
  \begin{align}
    \big| P^L[W_z] - L n^{-1 }f (z)  \big| \leq c_{ \piece}
    n^{-1 + (2\piece/3) },
  \end{align}
  hence \eqref{e:perc}. This  completes the proof of Lemma~\ref{l:perc}
  and thus of Proposition~\ref{pr:dominat}.
\end{proof}

\subsection{Existence of mesoscopic components}
\label{ss:robust}
We now use the results of the last subsection to establish the existence
of many mesoscopic components in the (appropriately modified) vacant set.
In order to state the precise result, we need, as we have discussed
before, to introduce the long-range
bridges that are necessary to perform the sprinkling.

Recall from
Section~\ref{s:partsbridges} that $a_i = Y^i_0$ and $b_i = Y^i_L$
denote the start- and end-point of the segment $Y^i$, $i \geq 1$.
On the same probability space $(\Omega ,\Plr)$, we now
define a family of $D([0,\ell],V)$-valued random variables $Z^{i,j}$,
$i \in \mathbb N$, $j \in \{1, \dots, \lfloor \ln n \rfloor\}$, with law
characterized by the following:
\begin{equation}
  \begin{array}{l}
    \label{eq:ZgivenY}
    \text{conditionally on $a_i$, $b_i$, the $Z^{i,j}$'s are  independent,} \\
    \text{independent of the $Y^i$'s, and have distribution
      $P^\ell_{b_i,a_{i+j}}.$}
  \end{array}
\end{equation}
We call $Z^{i,j}$'s the \textit{long-range bridges}.
Given the $Y^i$'s and $Z^{i,j}$'s as above, we denote by
$\xi'_u \in \{0,1\}^V$ the indicator function of the vacant set left by
them, i.e.
\begin{equation}
  \label{eq:xis}
  \xi_u' = 1 \big\{V \setminus \cup_{i < M_u, j \leq \ln n }
    \{ \Ran Y^i  \cup \Ran Z^{i,j} \} \big\}.
\end{equation}
From definitions of $\xi_u$ and $\xi_u'$ it follows that
$\xi'_u \le \xi_u$.

We now show that the configuration $\xi'_u$ has many mesoscopic
components.
More precisely, the following proposition shows that
with high probability, a constant proportion of vertices is contained in
components of $\xi '_u$ with size of order $n^{\vv_{u(1+\epsilon)} \beta}$.
\begin{proposition}
  \label{pr:lotgood}
  For $0<u<u_\star$, there exist constants $c_1$, $c_2$ depending on $\good$,
  $\gap$ and $u$, such that
  \begin{equation}
    \label{eq:lotgood}
    \Plr\Big[
      \big| \{x\in V: |\comp_x(\xi'_{u})| \geq  c_1 n^{\vv_{u(1+\epsilon)} \beta} \} \big|
      \geq c_1 n \Big]
    \geq 1 - c_2 \exp\{-c \ln^3 n\}.
  \end{equation}
\end{proposition}

The proof of the proposition has two parts. First, in
Lemma~\ref{lm:proper}, we establish a similar result for the
configuration $\xi_{u}$ defined in \eqref{eq:xi} as the indicator of the
complement of the segments. We then show
that many of them survive adding the long-range bridges which will prove
the Proposition~\ref{pr:lotgood}.

\subsubsection{Robust mesoscopic components for $\xi_u$}

In order to ensure that adding the long-range bridges does not destroy
the components of $\xi_u$ of size $c_1 n^{\vv_{u(1+\epsilon)} \beta}$, we
should make them more robust. We therefore impose the following more
restrictive conditions on the components to be found.
\begin{definition}
  \label{def:proper}
  Let  $\eta$ be a configuration in  $\{0,1\}^V$ and set for
  $l\in \mathbb N$
  \begin{equation}
    \label{eq:Cxl}
    \comp_x^l(\eta )=\{y\in \partial_i B(x,l):
      y \text{  is connected to $x$ by a path in
        $\comp_x(\eta) \cap B(x,l)$}\}.
  \end{equation}
  (Note that $\comp_x^l(\eta)$ is contained in, but not necessarily equal to
  $\comp_x(\eta) \cap \partial_i B(x,l)$.)
  Given a positive parameter $h$, a
  given site $x \in V$ is said to be $h$-\emph{proper} under the
  configuration $\eta$, if $\tx(B(x,3 \beta \ld n)) = 0$ and the
  following two conditions hold (recall \eqref{eq:muvu}):
  \begin{equation}
    \label{eq:proper}
    \begin{split}
      \text{(i) } &
      |\comp_x^{\beta \ld n} (\eta)| \geq h \,
      m_{u(1+\epsilon)}^{\beta \ld n} = h  n^{\vv_{u(1+\epsilon)}\beta},\\
      \text{(ii) } &
      \text{$|\comp_y^l (\eta)| \leq  m_{u(1-\epsilon)}^{(5/4) l}$
        for all $y \in B(x,\beta \ld n)$,
        $l \in \big[l_0 , l_1 \big]\cap \mathbb N$,}\\
      &\text{where $l_0= \lceil 10\ln\ln n / \vv_{u(1-\epsilon)} \rceil$,
        $l_1= \lfloor \beta \ld n \rfloor$.}
    \end{split}
  \end{equation}
\end{definition}

The next lemma proves the existence of many proper sites.

\begin{lemma}
  \label{lm:proper}
  For $\beta$, $\varepsilon $ as in \eqref{e:newgamma},\eqref{eq:epsilon} there
  exist constants $c_3(u)$, $c_4( u)$ and
  $c(\good,u)$ such that
  \begin{equation}
    \label{eq:lmproper}
    \Plr \Big[ \big|\{x \in V: x \textnormal{ is $c_3$-proper under $\xi_{u}$}\}
      \big| \geq c_4 n \Big] \geq 1 - c \exp \{- \ln^3 n\}.
  \end{equation}
\end{lemma}

\begin{proof}
  We first show with Proposition~\ref{pr:dominat} and estimates on
  branching processes that the expected number of proper vertices is of
  order $n$ and then we show that this number is concentrated around its
  expectation. Throughout this proof, we abbreviate $m_{u(1+\epsilon)}$ and
  $m_{u(1-\epsilon)}$ by $m_+$ and $m_-$.

  Consider  $x \in V$ such that $\tx(B(x, \good \ld n)) = 0$. We estimate
  the probability that $x$ satisfies condition \eqref{eq:proper}(i) with
  $h>0$ to be chosen. Since $5 \beta < \good$, using
  Proposition~\ref{pr:dominat} with $u(1+\epsilon) < (u+u_\star)/2$
  (cf.~\eqref{eq:epsilon}),
  \begin{equation}
    \label{eq:bigpop}
    \begin{split}
      \Plr \big[
        |\comp_x^{\beta \ld n} (\xi_u)| < h m_+^{\beta \ld n}\big]
      &\leq \Ber_{u(1+\epsilon)}\big[|\comp_o^{\beta \ld n}|
        < h m_+^{\beta \ld n}\big]
      + c_{ u} n^{-2\beta}.
    \end{split}
  \end{equation}
  For $u(1+\epsilon) < (u+ u_\star)/2$, the branching
  process induced by $\Ber_{u(1+\epsilon)}$ is supercritical. Hence,
  \cite[Theorem~2,~p.~9]{AN} implies that for $h < c_{ u}$ chosen small
  enough, the first term on the right-hand side of \eqref{eq:bigpop} is
  bounded by $1-c_{ u}$ for $n \geq c'_{u}$. Therefore, letting $h$ be some
  strictly positive constant $c_3 < c_{ u}$,
  \begin{equation}
    \label{eq:biginx}
    \sup_{n \geq c_{u}}
    \Plr \big[ |\comp_x^{\beta \ld n}(\xi_u)| < c_3
      m_+^{\beta \ld n}\big] < 1 - c_{ u}.
  \end{equation}
  We now treat condition \eqref{eq:proper}(ii). Since $m_u > 1+c_u$, we can use
  \cite[Theorem 4]{KA} to find a $\theta_u \in (0,c_u)$ such
  that (here, $\EBer_u$ denotes $\Ber_u$-expectation)
  \begin{equation}
    \label{eq:ldp}
    H_u = \sup_{n}\, \sup_{l\geq 0} \EBer_u
    \Big[ \exp \Big\{ \theta_u \frac{ |\comp_y^l| } {m^l_{u}} \Big\} \Big] < \infty.
  \end{equation}
  We claim that for $l_0$ and $l_1$ as in \eqref{eq:proper}, and $\epsilon$
  as in \eqref{eq:epsilon},
  \begin{equation}
    \label{eq:tamed}
    \Ber_{u(1-\epsilon)} \Big[
      |\comp^l_o| > m_-^{(5/4)l} \text{ for some } l \geq l_0
      \Big] \le  c_{u} \exp \{-c'_{u} \ln^2n \}.
  \end{equation}
  Indeed, by the exponential Chebyshev
  inequality, the left-hand side of \eqref{eq:tamed} can be bounded from
  above by {\allowdisplaybreaks
    \begin{align}
      \label{eq:Qtamed}
      & \sum_{l = l_0}^\infty
      \Ber_{{u(1 - \epsilon)}}\Big[
        \exp\Big\{\theta_{{u(1-\epsilon)}} \frac{|\comp^l_o|}{m_{-}^l} \Big\}
        > \exp \Big\{\theta_{{u}} m_{-}^{(1/4) l }
          \Big\}\Big] \\
      & \overset{\eqref{eq:ldp}}{\leq}
      \sum_{l = l_0}^\infty
      H_{{u(1-\epsilon)}} \exp \Big\{ -\theta_{{u(1-\epsilon)}} m_{-}^{(1/4) l } \Big\} \nonumber\\
      &\leq H_{{u(1-\epsilon)}}
      \exp \Big\{ - \theta_{{u(1-\epsilon)}}
        \exp \big\{
          \lceil \tfrac{10\ln\ln n}{ \ld m_{-}} \rceil \frac{1}{4} (\ld m_{-}) (\ln d)  \big\} \Big\}
      \sum_{k=0}^\infty
      e^{-\theta_{{u(1-\epsilon)}} m_{-}^{l_0}(m_{-}^{\frac{1}{4}k}-1)} \nonumber \\
      & \stackrel{(\ln d \geq 1)}{\leq} c_{u} \exp \{ - c'_{u} \ln^2 n\}. \nonumber
  \end{align}}
  This proves \eqref{eq:tamed}. It follows that
  \begin{equation}
    \label{eq:branch} \begin{split}
      \Plr & \Big[
        |\comp_y^l(\xi_u)| >  m_{-}^{(5/4) l} \text{ for some }
        y \in B(x,\beta \ld n), l_0\le l\le l_1 \Big] \\
      & \overset{Prop.~\ref{pr:dominat}}{\leq} c
      n^\beta \bigg\{ \Ber_{u(1-\epsilon)}
        \big[ |\comp^l_o| > m_{-}^{(5/4)l}
          \text{ for some } l \geq l_0 \big] + c_{ u} n^{-2\beta} \bigg\} \overset{\eqref{eq:tamed}}{\leq} c_{ u} n^{-\beta}.
    \end{split}
  \end{equation}
  The above bound, together with \eqref{eq:biginx}, allows us to conclude
  that for all $n \geq c'_{u}$,
  $\Plr\big[x \text{ is $c_3$-proper under $\xi_{u}$}\big] \geq c_{u}>0$.
  Summing this probability over the vertices $x$ with
  $\tx(B(x,3\beta \ld n))=0 $ (which have positive proportion by
    Lemma~\ref{l:txzero}) we obtain that
  \begin{equation}
    \label{eq:EPproper}
    \Elr[|\{x \in V: \text{$x$ is $c_3$-proper under $\xi_{u}$} \}|]
    \geq c_{ u} n.
  \end{equation}

  \smallskip

  We now show that the number of $c_3$-proper points concentrates around
  its expectation. To this end we use a concentration inequality in
  \cite{M89}, Lemma~1.2. We first consider a slightly modified
  configuration $\check \xi \in \{0,1\}^V$, where we consider only the
  first $[2n^\piece]$ jumps of each $Y^i$:
  \begin{equation}
    \check \xi = \check \xi \big(Y^0, Y^1, \dots, Y^{M_{u} -1}\big)
    = 1\Big\{ V \setminus \bigcup_{i < M_{u}}Y^i_{[0,\tau_{[2n^\piece]}(Y^i) \wedge L]}
      \Big\},
  \end{equation}
  where $\tau_{k}(Y^i)$ is the time of the $k$-th jump of $Y^i$
  (we set $\tau_k(Y_i)=L$ if $Y^i$ jumps less than $k$-times). We define
  a function
  \begin{equation}
    f(Y^0,\dots,Y^{M_{u} -1}) = \big|
    \big\{x \in V: x\text{ is $c_3$-proper under }\check \xi \big\} \big|.
  \end{equation}
  We claim that, writing ${\vec{Y}}$ for $(Y^0,\dots,Y^{M_{u} -1})$,
  \begin{equation}
    \label{eq:fhasb}
    \begin{array}{l}
      \text{if ${\vec{Y}}$ and ${\vec{Y}}'$
        differ in at most one coordinate, then }
      |f({\vec{Y}})-f({\vec{Y}}')| \leq cn^{\piece + \beta}.
    \end{array}
  \end{equation}
  Indeed, changing one  segment $Y^i$,  we can change at most $2n^\piece$
  values of $\check \xi$. Moreover, the event that a given point $x \in V$
  is $c_3$-proper under $\check \xi$ only depends on the values of
  $\check \xi$ in $B(x,\beta \ld n)$, which has volume bounded by $cn^\beta$.
  This gives \eqref{eq:fhasb}. Note that
  \begin{equation}
    \begin{split}
      \Elr[f] & \geq \Elr[f 1_{\xi_{u} = \check \xi}]
      = \Elr \big[  |\{x \in V: x\text{ is $c_3$-proper under } \xi_{u} \}|
        \cdot 1_{\xi_{u} = \check \xi}\big] \\
      & \overset{\eqref{eq:EPproper}}{\geq} c_{u} n - n \cdot
      \Plr[\xi_{u} \neq \check \xi], \textrm{ for $n \geq c_{u}$.}
    \end{split}
  \end{equation}
  The bound \eqref{eq:Wj} implies that
  \begin{align}
    \Plr[\xi_{u} \neq \check \xi] &\leq M_u P [N_L \geq 2n^\piece] \leq \exp \{-c_u n^\piece\}.
  \end{align}
  Hence, we have that $\Elr[f]  \geq c_{u} n$ for $n \geq c'_{u}$.
  Setting $t = \frac{1}{2} c_{ u} n$, with the same constant as in the
  lower bound on $\Elr[f]$, we obtain
  \begin{equation}
    \begin{split}
      \Plr[|\{x & \in V: \text{ is $c_3$-proper under } \xi_{u} \}| \leq t]
      \leq \Plr[\xi_{u} \neq \check \xi] + \Plr[ \Elr[f] -f \geq  t] \\
      & \leq  \exp \{-c_u n^\piece\} + 2\exp\{ -c_{ u} n^{2 - 1 + \piece - 2(\piece + \beta)} \},
    \end{split}
  \end{equation}
  where we have used Lemma~1.2 in \cite{M89}, together with
  \eqref{eq:fhasb}, in the last inequality. Since
  $1 - \piece - 2\beta \geq 1 - 3\beta > 0$, this estimate is more than
  enough to imply \eqref{eq:lmproper} for appropriately chosen constants
  $c_4$ and $c$. This concludes the proof of Lemma~\ref{lm:proper}.
\end{proof}

\subsubsection{Robustness of proper sites}
In this sub-section we prove that the components around $h$-proper sites
(as in Definition~\ref{def:proper}) are really robust with respect to
perturbation. Observe that the following lemma is completely
deterministic.

\begin{lemma}
  \label{lm:lotstrings}
  Let $\beta $, $\varepsilon $ be as in \eqref{eq:epsilon},
  \eqref{e:newgamma} and define
  the class $\Xi$ of configurations in $\{0,1\}^V$,
  \begin{equation}
    \label{eq:Xi}
    \Xi = \big\{\eta \in \{0,1\}^V:
      \big|\{x \in V: x \textnormal{ is $c_3$-proper}\}\big| \geq c_4 n
      \big\}.
  \end{equation}
  Let $\eta \in \Xi$ and $\eta' \in \{0,1\}^V$ be such that
  $\eta'(z) \neq \eta(z)$ for at most $n^{1-\piece}\ln^5 n $ vertices
  $z\in V$. Then there exists a constant $c(\good, u)$, such that
  \begin{equation}
    \label{eq:iflots}
    \big|\{x \in V: |\comp_x(\eta')| \geq cn^{\vv_{u(1+\epsilon)} \beta} \}\big|
    \ge   cn.
  \end{equation}
\end{lemma}

\begin{proof}
  In this proof, we use the word ``proper'' to mean ``$c_3$-proper under
  $\eta $'' and use $m_+$, $m_-$, $v_+$ and $v_-$ to abbreviate
  $m_{u(1+\epsilon)}$ , $m_{u(1-\epsilon)}$, $v_{u(1+\epsilon)}$ and
  $v_{u(1-\epsilon)}$. We will use the term \emph{string} to refer to a
  self-avoiding path on $V$ with length $l_1=\lfloor \beta \ld n \rfloor$,
  as in \eqref{eq:proper}. For $\eta \in \Xi$, we are going to choose a
  particular collection $\Gamma_\eta$ of strings, which will be contained
  in $\supp \eta $, as   follows. First, we take a collection of proper
  vertices $\Pi = \{x_1, \dots, x_{\lfloor c_4 n \rfloor}\} \subseteq V$,
  according to some pre-defined order. Again using some arbitrary order,
  for each $l\le\lfloor c_4 n \rfloor$,  we insert into $\Gamma_\eta$
  $\lfloor c_3 n^{\vv_+ \beta} \rfloor$ distinct strings starting at
  $x_l \in \Pi$ and contained in $\supp \eta $. Such a collection exists
  due to \eqref{eq:proper}($i$) (see also \eqref{eq:Cxl}). Denoting by
  $|\Gamma_\eta |$ the number of strings in $\Gamma_\eta $, we have
  \begin{equation}
    \label{e:gamasize}
    |\Gamma_\eta|
    = \lfloor c_4 n \rfloor \cdot \lfloor c_3 n^{\vv_+ \beta}\rfloor.
  \end{equation}
  Since for all $l \leq \lfloor c_4 n \rfloor$, $B(x_l,2\beta \log n)$ has
  tree excess zero,
  \begin{equation}
    \label{eq:endpoints}
    \text{every string in $\Gamma_\eta$ is uniquely determined by its end-points}.
  \end{equation}
  Let $S_y$ be the number of strings in $\Gamma_\eta $ intersecting $y$.
  We claim that, for any given $y\in V$,
  \begin{equation}
    \label{eq:congestion}
    S_y \le  c_{u}(\log n)^{c'_{u}} \cdot n^{\frac{5}{4}v_- \beta}.
  \end{equation}
  To show this claim, observe that the fact that the starting point of
  every string in $\Gamma_\eta$ is proper together \eqref{eq:proper}(ii)
  imply that if there is a string intersecting $y$, then
  \begin{equation}
    \label{eq:ytamed}
    |\comp_y^l(\eta)| \leq  m_-^{(5/4) l} \text{ for every
      integer } l \in [l_0,l_1].
  \end{equation}
  We bound $S_y$ by splitting the set of strings intersecting $y$ in the
  following way:
  \begin{equation}
    S_y = \sum_{l=0}^{l_1} \#\{ \text{strings in
        $\Gamma_\eta$ intersecting $y$ and starting at distance $l$ from
        $y$}\}.
  \end{equation}
  Since the strings are contained in $\supp \eta $, using \eqref{eq:endpoints}, for
  $n \geq c_{u}$,we obtain
  \begin{equation}
    S_y \leq \sum_{l=0}^{l_1 }
    |\comp_y^l(\eta)| |\comp_y^{l_1 -l}(\eta) |
    \le 2 \sum_{l=0}^{l_0-1}|\comp_y^l(\eta)|
    |\comp_y^{l_1 -l}(\eta) |
    + 2  \sum_{l=l_0}^{l_1/2 }  |\comp_y^l(\eta)|
    |\comp_y^{l_1 -l}(\eta) |.
  \end{equation}
  Using \eqref{eq:ytamed}, the bound $|\comp_y^k|\le c (d-1)^k$
  for $k< l_0$, and   $l_0\le c_u \ln \ln n$, we get
  \begin{equation}
    S_y \le c_u \ln \ln n \cdot  (d-1)^{l_0} m_-^{(5/4) l_1}
    + 2 \sum_{l=0}^{l_1/2} m_-^{(5/4) l}
    m_-^{(5/4) (l_1 -l)}.
  \end{equation}
  From \eqref{eq:muvu}, it follows that  $m_-^{l_1} \le n^{v_- \beta }$.
  Hence,
  \begin{equation}
    S_y \le c_{u}(\ln n)^{c'_u}
    n^{\frac{5}{4}\vv_- \beta}
    + 2 \sum_{l=0}^{l_1 /2} n^{\frac{5}{4}\vv_- \beta }
    \leq c_{u}(\ln n)^{c'_{u}} \cdot n^{\frac{5}{4}v_- \beta}.
  \end{equation}
  This proves \eqref{eq:congestion}.

  Our next step is to show that there exists $c_{u}$ such that
  \begin{equation}
    \label{eq:lotstrings}
    \text{for $n \geq c_{u}$, at least half of the strings of
      $\Gamma_\eta $ are contained in $\supp \eta '$}.
  \end{equation}
  Indeed, we know that $\eta'(z) \neq \eta(z)$ for at most
  $n^{1-\piece }\ln^5 n $ vertices $z\in V$. This, together with
  \eqref{eq:congestion}, implies that   at most
  $c_{u}(\ln n)^{c'_{u}}n^{1-\piece+\frac 54 v_- \beta}$   strings in
  $\Gamma_\eta$ are not contained in $\supp \eta '$. Since
  $\piece = \vv_+ \beta/2$ (cf.~\eqref{e:newgamma}), we obtain by
  \eqref{e:epsv} that $1-\piece+\frac54 v_- \beta < 1+\vv_+ \beta$.
  Therefore, due to \eqref{e:gamasize}, for $n\ge c''_{u}$, at least half
  of the strings in $\Gamma_\eta$ are contained in $\supp \eta '$. This
  gives us \eqref{eq:lotstrings}.

  Let us recall that in the construction of the set $\Gamma_\eta$, we have
  chosen a collection $\Pi$ of $\lfloor c_4 n \rfloor$ proper vertices in
  $V$, and for each of these vertices, we have picked
  $\lfloor c_3 n^{\vv_+ \beta} \rfloor$ strings starting at $x_l$. We claim
  that
  \begin{equation}
    \label{eq:1over8}
    \begin{array}{c}
      \text{for $n \geq c_{u}$, at least
        $\lfloor \frac{c_4}{8} n \rfloor$ of the vertices in $\Pi$ have} \\
      \text{at least $\frac{c_3}{8} n^{\vv_+ \beta}$ of their strings contained in
        $\supp \eta '$}.
    \end{array}
  \end{equation}
  Indeed, otherwise the number of strings in $\Gamma_\eta$ contained in
  $\supp \eta'$ would be bounded by
  \begin{align}
    \frac{c_4}{8} n \cdot c_3 n^{v_+ \beta}
    + \frac{7 c_4}{8} n \cdot \frac{c_3}{8} n^{v_+ \beta}
    \leq \frac{c_3 c_4}{4} n^{1+v_+ \beta},
  \end{align}
  contradicting \eqref{eq:lotstrings} for $n \geq c_{u}$.

  Since distinct strings starting on a vertex in $\Pi$ have distinct end
  points, each vertex $x$ as in \eqref{eq:1over8}  satisfies
  $|\comp_{x}(\eta')|\geq \frac{c_3}8 n^{\vv_+ \beta}$. Choosing $c$
  as $\frac{c_3 \wedge c_4}{8}$, we deduce \eqref{eq:iflots}   for
  $n \geq c_{u}$. By possibly decreasing $c = c_{u}$ in such a way that
  $\lfloor c n \rfloor = 0$ for the other finitely many values of $n$, we
  obtain Lemma~\ref{lm:lotstrings}, and thus complete the proof of
  Proposition~\ref{pr:lotgood}.
\end{proof}

\subsubsection{Mesoscopic components for $\xi'_u$}
With Lemmas~\ref{lm:proper} and~\ref{lm:lotstrings}, we have all tools to
finish the proof of Proposition~\ref{pr:lotgood} stating the existence of
many mesoscopic components of the complement of the segments and the
long-range bridges.
\begin{proof}[Proof of Proposition~\ref{pr:lotgood}]
  The configurations $\xi_u$ and
  $\xi_u '$ differ only on vertices visited by the bridges. Hence, setting
  $\mathcal D=|\{z\in V:\xi_{u} (z)\neq \xi'_{u} (z)\}|$, denoting by
  $N^{i,j}$ the number of jumps of the bridge $Z^{i,j}$, for
  $n\ge e^{u_\star}$, and using $M_u \le c_u n^{1-\piece }$, we obtain
  \begin{equation}
    \label{eq:xixi'}
    \begin{split}
      \Plr \big[\mathcal D > n^{1-\piece }\ln^5 n\big]
      &\le
      \Plr \Big[\sum_{i<M_{u},j\le \ln n} N^{i,j}
        > c_u n^{1-\piece}\ln^4 n\Big]\\
      &\le \Plr \Big[ \bigcup_{i \leq M_{u}, j \leq  \ln n } \{ N^{i,j} >
          \ln^3 n \} \Big] \leq c_{ u} \exp\{-c \ln^3 n\},
    \end{split}
  \end{equation}
  where we have used \eqref{eq:Wijlog3} from Lemma~\ref{lm:visits} in the last inequality.
  By
  possibly increasing $c_{ u}$ we conclude that the equation above holds
  for every $n \geq 1$. Finally, taking $c$ as in
  Lemma~\ref{lm:lotstrings}, using Lemma~\ref{lm:lotstrings},
  \begin{equation}
    \begin{split}
      \Plr \Big[ & \big| \{x\in V: |\comp_x(\xi'_{u})| \geq c n^{\vv_{u(1+\epsilon)} \beta} \}\big| <  c n  \Big] \\
      &\leq  \Plr \Big[ \big|\{x : x \textnormal{ is $c_3$-proper under $\xi_{u}$}\} \big| < c_4 n \Big]
      + \Plr \left[ \mathcal D> n^{1-\piece }\ln^5 n
        \right].
    \end{split}
  \end{equation}
  By Lemma~\ref{lm:proper} and \eqref{eq:xixi'} the last expression is
  bounded by $ c_{u} e^{-c \ln^3 n}$. The proof of
  Proposition~\ref{pr:lotgood} is then finished by choosing the constants
  $c_1$ and $c_2$ appropriately.
\end{proof}

\subsection{Proof of Theorem~\ref{th:usmall} and sprinkling.}
\label{ss:proof}

We can now approach the second part of the proof of Theorem~\ref{th:usmall},
that is the sprinkling construction.

\begin{proof}[Proof of Theorem~\ref{th:usmall}]
  Let $u$ be as in the theorem and choose $\delta >0$ such that
  \begin{equation}
    2 \delta < 1- \frac{1}{200}
    \stackrel{\eqref{e:newgamma}}{<} 1-\piece (u).
  \end{equation}
  Set $u_n$ and $u'$ as
  \begin{equation}
    \label{e:Unew}
    u'=\frac{u+u_\star}{2},\qquad  u_n  = (u + n^{-\delta}) \wedge u',
  \end{equation}
  Throughout this proof, we write $\epsilon$ and $\gamma$ for
  $\epsilon = \epsilon(u')$ and $\gamma= \gamma(u')$ as in
  \eqref{eq:epsilon} and \eqref{e:newgamma} with $u$ replaced by $u'$.

  The strategy of this proof is the following: we apply
  Proposition~\ref{pr:lotgood} to $\xi'_{u'}$ as defined in \eqref{eq:xis}.
  This will show that with high probability, there are at least $c_1 n$
  vertices in components of volume $c_1 n^{\vv_{u'(1+\epsilon)} \beta}$ in
  $\supp \xi'_{u_n } \supseteq \supp \xi'_{u'}$. In what we call the
  sprinkling construction, we then erase some of the $Y^i$'s in the
  definition of $\xi'_{u_n }$, and thereby increase the configuration
  $\xi'_{u_n }$ to a new configuration $\xi^{sp}$. By construction, the
  sprinkled configuration $\xi^{sp}$ will be close in distribution to the
  vacant set left by the random walk trajectory $X_{[0,un]}$. Moreover, we
  will prove that with high probability some of the components of
  $\supp \xi'_{u_n }$ will merge and form a component of size $\rho n$ as
  we increase $\xi'_{u_n }$ to $\xi^{sp}$, thus proving
  Theorem~\ref{th:usmall}.

  We divide the proof into the following three steps: in the first step, we
  construct the sprinkled configuration $\xi^{sp}$ and reduce
  Theorem~\ref{th:usmall} to an estimate on $\xi^{sp}$. In the second step,
  we apply Proposition~\ref{pr:lotgood} to prove that the original
  configuration $\xi'_{u_n }$ is sufficiently well-behaved. In the third
  and final step, we deduce that with high probability, $\supp \xi^{sp}$
  has a component with volume at least $\rho n$ and conclude.

  \emph{Step 1: The sprinkling construction.} For the sprinkling
  construction, we use an auxiliary probability space
  $(\{0,1\}^{M_{u_n }}, \mathcal{Q})$, for $M_{u_n }$ defined in
  \eqref{e:M}. Under the measure $\mathcal Q$, the canonical coordinates
  $(R_k)_{0 \leq k < M_{u_n }}$ are i.i.d.~Bernoulli random variables with
  parameter
  \begin{align}
    \label{e:q} q = n^{-2\delta}.
  \end{align}
  Recall that the configuration $\xi_{u_n }'$ was defined in \eqref{eq:xis}
  as the indicator function of the set of vertices not visited by $Y^i$ and
  $Z^{i,j}$ for $i < M_{u_n }$ and $j \leq \lfloor \ln n \rfloor$,
  constructed on a suitable probability space $(\Omega, \Plr)$. On
  the probability space
  $(\Omega \times \{0,1\}^{M_{u_n }}, \Plr \otimes {\mathcal Q})$,
  we will now construct from $\xi'_{u_n }$ the sprinkled configuration
  $\xi^{sp}$, roughly according the following procedure: first, we remove
  all $Y^i$'s such that $R_i=1$. If possible, we then construct a
  trajectory by linking the remaining segments with the bridges $Z^{i,j}$,
  see Figure~\ref{fig:lrb} for a sketch.
  \begin{figure}[ht]
    \begin{center}
      \includegraphics[angle=0, width=0.8\textwidth]{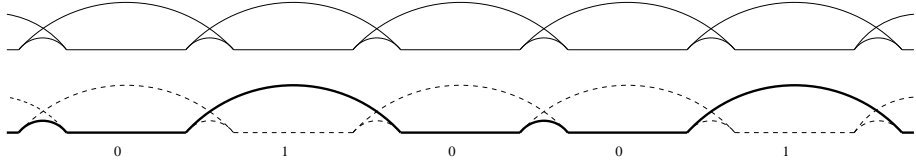}\\
      \caption{Schematic illustration the choice performed by $\psi$. The
        horizontal lines are the trajectories of the segments $Y^i$ and the arcs
        are the trajectories of the bridges $Z^{i,j}$.  The top picture
        illustrates the configuration $\xi_{u_n }'$, The
        sprinkled configuration $ \xi^{sp}$ obtained from the realization of the
        Bernoulli random variables $R_i$ associated to $Y^i$ is on the bottom
        picture.}\label{fig:lrb}
    \end{center}
  \end{figure}

  For the precise construction, we define the increasing random sequence
  $k_i$ of indices $k < M_{u_n }$ for which $R_k = 0$;  the sub-indices $i$
  here run from $0$ to the random variable
  \begin{align}
    \label{e:I} I = |\{k:0\leq k < M_{u_n},  R_k=0\}|.
  \end{align}

  We now construct a function that concatenates the $Y^i$'s with $R_i=0$,
  $i<M_{u_n }$, and some of the bridges $Z^{i,j}$ into an element of
  $D([0,un],V)$, for $u$ as in the theorem. There are two situations in
  which this construction fails. First, if $I < M_u$, then there are not
  enough segments $Y^i$ left. Second, if
  $k_{i+1} - k_i > \lfloor \ln n \rfloor$ for some $i \leq I$, there is no
  bridge  connecting $Y^{k_{i}}$ to $Y^{k_{i+1}}$. Let us hence refer to
  the intersection of the complements of these events as the good event
  $\mathcal G$,
  \begin{align}
    \label{e:calG}
    {\mathcal G} = \{I \geq M_u \} \cap
    \{ k_{i+1}-k_i \leq \lfloor \ln n \rfloor \text{ for all } 1 \leq i \leq M_u\}
    \subseteq \{0,1\}^{M_{u_n }}.
  \end{align}
  Letting $\partial$ be some arbitrary constant trajectory of length $un$,
  we now define
  $\psi:\Olr \times \{0,1\}^{M_{u_n }} \rightarrow D([0,un],V)$ as
  \begin{align}
    \psi = \begin{cases}
      \partial, \qquad  \textrm{on $\Omega \times {\mathcal G}^c$, and
        otherwise:}\\
      \mathcal X
      (Y^{k_1},Z^{k_1,k_{2} - k_1},Y^{k_2},Z^{k_2, k_3-k_2},\dots,
        Y^{k_{M_u}},Z^{k_{M_u},k_{{M_u}+1}-k_{M_u}}) \big|_{D([0,un],V)},
    \end{cases}
  \end{align}
  where $|_{D([0,un],V)}$ denotes the restriction to $D([0,un],V)$ and
  $\mathcal X$ is the concatenation mapping defined in \eqref{e:mathcalX}
  (Here we abuse the notation slightly. The mapping $\mathcal X$ takes
    infinite number of arguments, however since $un \leq M_u (L+\ell)$
    the restriction to $D([0,un],V)$ does not depend on the arguments
    which we do not specify). The sprinkled configuration is then defined
  as the indicator function of the vacant set left by the concatenated
  trajectory,
  \begin{equation}
    \label{eq:xisp}  \xi^{sp} = 1\{V \setminus X_{[0,un]}\} \circ \psi,
  \end{equation}
  where we have used the notation $X$ for the canonical coordinate process
  on the space $D([0,un], V)$.

  By construction, we then have
  $\supp \xi'_{u_n } \subseteq \supp \xi^{sp}$ on $\Omega \times \mathcal G$.
  Moreover, conditionally on $\Omega \times \mathcal G$, the concatenation
  $\psi$ is distributed according to the piecewise independent measure
  $Q^{un}$ defined in Section \ref{s:partsbridges}.

  Let us now
  see that the event $\mathcal G$ is indeed typical.  The random variable
  $I$ is binomially distributed with expectation
  $M_{u_n }(1-q) = M_u + n^{1- \piece-\delta} (1+o(1))$. With the help of a
  Chernoff bound (see  Lemma~1.1 in \cite{M89}), we find that
  \begin{align}
    \label{e:Gbd0} \mathcal{Q}[I < M_u] \leq \exp(-c_{u,\delta}
      n^{2(1-\piece-\delta)}/M_{u_n})
    \leq \exp(-c_{u,\delta} n^{1- \piece - 2\delta}).
  \end{align}
  Together with a simple union bound, this implies that
  \begin{align}
    \label{e:Gbd} {\mathcal Q}[{\mathcal G}^c] &\leq  \mathcal{Q}
    \big[\exists i \le  I: k_{i+1} - k_i > \lfloor \ln n \rfloor\big] +
    \mathcal{Q}[I < M_u] \\ &\leq c_u n^{1-\piece} q^{\lfloor \ln n
      \rfloor} + \exp(-c_{u, \delta} n^{1-\piece - 2\delta}) \nonumber\\
    &\stackrel{\eqref{e:q}}{\leq} c_u n^{1-\piece} n^{-2\delta \lfloor \ln
      n \rfloor}     + \exp(-c_{u, \delta} n^{1-\piece - 2\delta})\leq
    c_{u,\sigma,\delta} n^{-\sigma}. \nonumber
  \end{align}
  In particular, the distribution $\psi\circ(\Plr \otimes \mathcal{Q})  $
  is close to $Q$. Indeed, for any ${\mathcal F}_{un}$-measurable
  event $A$, we have
  $Q^{un}[A]= \Plr \otimes \mathcal Q[\psi^{-1}(A)|\Omega \times \mathcal G]$,
  and therefore using an easy calculation
  \begin{equation}
    \label{e:Qm-}
    | Q^{un}(A) - (\Plr \otimes \mathcal{Q}) ( \psi^{-1} (A))| \\
    \leq  \mathcal{Q} ({\mathcal G}^c) \leq c_{u,\sigma,\delta} n^{-\sigma}.
  \end{equation}

  Thanks to this estimate, we know that the sprinkled configuration
  $\xi^{sp}$ is close in distribution to the vacant set left by a
  trajectory under the piecewise independent measure $Q^{un}$.
  Together with Lemma~\ref{l:partsbridges}, we can now reduce our task to
  proving the estimate in Theorem~\ref{th:usmall} for the configuration
  left by $\xi^{sp}$. We set
  \begin{align}
    \rho = c_1/100,
  \end{align}
  where $c_1$ was defined in Proposition~\ref{pr:lotgood}. By
  Lemma~\ref{l:partsbridges},
  \begin{align}
    P \big[  |\comp^u_{\Max}| < \rho n\big] \leq Q^{un}
    \big[ |\comp_{\Max}(1 \{V \setminus X_{[0,un]}\})| < \rho n \big] +
    \exp\{ -c_{u} \ln^2 n\},
  \end{align}
  which by \eqref{e:Qm-} implies that
  \begin{align}
    \label{e:main0} P \big[  |\comp^u_{\Max}| < \rho n\big] \leq
    \Plr \otimes \mathcal{Q} \big[ |\comp_{\Max} (\xi^{sp}) | <
      \rho n \big] + c_{u,\delta, \sigma} n^{-\sigma}.
  \end{align}
  It is therefore sufficient to show that
  \begin{equation}
    \label{eq:reductusmall}
    \Plr \otimes \mathcal{Q} \Big[  |\comp_{\Max} (\xi^{sp}) | <
      \rho n  \Big] \leq c_{u, \delta, \sigma} n^{-\sigma}.
  \end{equation}

  \emph{Step 2: $\xi'_{u_n}$ is well-behaved with high probability.} In
  the second step, we apply previous estimates in order to deduce that
  $\xi'_{u_n}$ has the properties we will use to show that a component of
  size $\rho n$ appears in $\supp \xi^{sp}$.

  Let $\comp_1, \comp_2, \dots$ be the connected components of
  $\supp \xi_{u_n}'$ ordered according to their volume, $\comp_1$
  being the largest component, and define the random variable
  $\kappa \in {\mathbb N} \cup \{ \infty \}$ as the smallest integer such
  that
  \begin{equation}
    |\comp_1| + \dots + |\comp_\kappa|
    \geq  c_1 n = 100 \rho  n,
  \end{equation}
  provided such an integer exists, and $\kappa = \infty$ otherwise. Note
  that if $\supp \xi_{u_n}'$ contains many large clusters (in the sense
    of Proposition~\ref{pr:lotgood}), then $\kappa $ is small. More
  precisely, we have the following event inclusion,
  \begin{equation}
    \big\{|\{x\in V: |\comp_x (\xi_{u_n}')|
        \geq c_1 n^{\vv_{u'(1+\epsilon)} \beta} \}
      | \geq  c_1 n \big\}\subseteq
    \{\kappa \leq \lceil n^{1-\vv_{u'(1+\epsilon)} \beta} \rceil\}=:\mathcal K
    \subset \Omega.
  \end{equation}
  Hence, Proposition~\ref{pr:lotgood} (and monotonicity of $\xi'_.$) is
  more than enough to imply that
  \begin{equation}
    \label{eq:kappasmall}
    \Plr [{\mathcal K}^c] \leq c_{u,\sigma} n^{-\sigma}.
  \end{equation}
  We further define an event $\mathcal S\subset \Omega $ as the event
  that the numbers of jumps $N^i$ of all $Y^i$'s and the total length $N^{i,j}$ of all $Z^{i,j}$'s
  appearing in the construction of $\xi'_{u_n}$ do not exceed their
  expected value too significantly:
  \begin{equation}
    \label{e:calS}
    \mathcal S=\{ N^i \le 2n^{\piece},\, \forall i<M_{u_n}\}
    \cap\Big\{\sum_{i<M_{u_n}}\sum_{j\le \ln n} N^{i,j} \le
      \ln^5 n \cdot n^{1-\piece}\Big\}.
  \end{equation}
  By Lemma~\ref{lm:visits} we know that
  \begin{align}
    \label{e:Sc} \Plr [\mathcal S^c]\le c_{u, \sigma} n^{-\sigma}.
  \end{align}
  The estimates \eqref{eq:kappasmall} and \eqref{e:Sc} will allow us to
  prove the required estimate \eqref{eq:reductusmall} in the last step by
  considering only configurations $\xi'_{u_n}$ satisfying the properties
  in ${\mathcal K} \cap {\mathcal S}$.

  \emph{Step 3: $\xi^{sp}$ has a large component with high probability.}
  Finally, we prove the required estimate \eqref{eq:reductusmall} by
  showing that the random deletion of $Y^i$'s in the construction of
  $\xi^{sp}$ does make a component of size $\rho n$ appear with high
  probability.

  We define an event $\mathcal P$ as
  \begin{equation}
    \label{e:calP}
    \mathcal P=\Bigg\{
      \parbox[c]{0.7\textwidth}
      {There exists a partition of $\{1, \dots, \kappa\}$
        into sets $A$ and $B$, such that
        $\mathcal A=\cup_{a\in A} \comp_a$ is not connected to
        $\mathcal B=\cup_{b\in B} \comp_b$ in $\supp \xi^{sp}$, and
        both $|\mathcal A|$ and
        $|\mathcal B|$ are larger than  $10\rho n$.
      }
      \Bigg\}
  \end{equation}
  In analogy with the proof of Proposition~3.1 in
  \cite{ABS04}, we claim that
  \begin{equation}
    \label{eq:partition}
    \{ |\comp_{\Max} (\xi^{sp}) | < \rho n \} \cap \{ \kappa < \infty \} \subseteq \mathcal P.
  \end{equation}
  To see this, we consider the equivalence relation $\sim$ on the set
  $\{1, \dots, \kappa\}$ given by $j \sim j'$ if and only if
  $\comp_j$ is connected to $\comp_{j'}$ in $\supp \xi^{sp}$.
  Then every equivalence class corresponds to one component of
  $\supp \xi^{sp}$. In particular, if all components of $\supp \xi^{sp}$
  are smaller than $\rho n$, then the sum   of $|\comp_j|$ for all
  $j$'s in the same equivalence class must also be smaller than $\rho n$.
  So, we can partition the set $\{1, \dots, \kappa\}$ into sets $A$ and
  $B$ in such a way that equivalent indices belong to the same set and
  $\big|\sum_{a\in A}|\comp_a| -\sum_{b\in B}|\comp_b|\big| \le 2\rho n$.
  Since $\sum_j |\comp_j| \geq 100 \rho n$, we obtain that
  $\sum_{a \in A} |\comp_a|$ and
  $\sum_{b \in B} |\comp_b| \geq 10 \rho  n$, and, by construction
  of the equivalence relation $\sim$, $\cup_{a \in A} \comp_a$ is
  not connected to $\cup_{b \in B} \comp_b$ through $\supp \xi^{sp}$.
  This shows \eqref{eq:partition}.

  For subsets $F$ and $F'$ of $V$, we use the notation
  $F\overset{sp}\nleftrightarrow F'$ to denote the event
  \begin{align}
    \{ F \textrm{ and } F' \textrm{ are not connected in } \supp \xi^{sp}\}.
  \end{align}
  By \eqref{eq:partition} and ${\mathcal K} \subseteq \{ \kappa < \infty\}$,
  \begin{equation}
    \label{e:PK}
    \Plr \otimes \mathcal{Q}
    \Big[ |\comp_{\Max} (\xi^{sp}) | < \rho n \Big]
    \le \Plr [\mathcal S^c] + \Plr [\mathcal K^c]+
    \Elr [{\mathbf 1}_{\mathcal K \cap \mathcal S} \mathcal Q[\mathcal P]],
  \end{equation}
  where we have used that $\mathcal S,\mathcal K\subset \Omega$ and
  Fubini's theorem. For the sake of clarity, let us recall that
  $\mathcal Q$ is a measure on $\{0,1\}^{M_{u_n}}$ and emphasize that the
  $\mathcal Q$-probability in this last expression is computed for $Y^i$
  and $Z^{i,j}$ fixed. On $\mathcal K$, there are at most
  $2^{\lceil n^{1-\vv_{u'(1+\epsilon)} \beta} \rceil}$ ways to partition
  $\{1,\dots,\kappa \}$ into $A$, $B$. Hence, on $\mathcal K$, using the
  union bound,
  \begin{equation}
    \mathcal Q[\mathcal P]\le
    2^{\lceil n^{1-\vv_{u'(1+\epsilon)} \beta} \rceil}
    \sup \mathcal Q
    \big[\mathcal A \overset{sp}\nleftrightarrow\mathcal B \big],
  \end{equation}
  where the supremum is taken over all partitions of
  $\{1,\dots, \kappa \}$ as in \eqref{e:calP} and $\mathcal A$,
  $\mathcal B$ are defined in \eqref{e:calP}, too. By increasing the
  range of the supremum we deduce that the following estimate holds
  uniformly on the event $\mathcal K$,
  \begin{equation}
    \label{e:PKS} \mathcal Q[\mathcal P]\le
    2^{\lceil n^{1-\vv_{u'(1+\epsilon)} \beta} \rceil}
    \sup_{F,F'\subseteq V:|F|,|F'|\ge 10 \rho n}
    \mathcal Q
    [F\overset{sp}\nleftrightarrow F'],
  \end{equation}
  We will now find a bound on the event on the right-hand side, valid
  uniformly on the event $\mathcal S$. To this end, fix $Y^i$ and
  $Z^{i,j}$ such that $\mathcal S$ holds, as well as subsets $F$ and $F'$
  of $V$ containing at least $10 \rho n$ vertices. Using the expansion
  property of the graph $G$, see \eqref{eq:cheeger}, and the Max-flow
  Min-cut Theorem, we can find a collection of at least $c_5 n$ disjoint
  paths in $V$ joining the sets $F$ and $F'$ for some constant $c_5 > 0$.
  We call these paths \emph{connections}. Since, on $\mathcal S$,
  $\sum_{i<M_{U}} N^i$ is smaller
  or equal to $2 u_\star n$ (see \eqref{eq:Wj} and \eqref{e:calS}),
  we can extract from this
  collection a sub-collection $\boldsymbol C$ such that
  $|\boldsymbol{C}| = \frac{1}{2} c_5 n$, and such that all connections
  in $\boldsymbol{C}$ intersect at most $\lfloor 4 u_\star/c_5 \rfloor$
  segments $Y^i$.

  We next want to prove that with high probability, at least one of these
  connections only intersects $Y^i$'s that do not appear in $\xi^{sp}$,
  using again the concentration inequality from \cite{M89}, Lemma 1.2. To
  this end, we define the function $g$ by
  \begin{equation}
    g(R_1, \dots, R_{M_{u_n}}) = |\{\zeta \in \boldsymbol{C}: \text{for
        all $i$, either } \zeta \cap Y^i = \varnothing \text{ or } R_i=1
      \}|.
  \end{equation}
  The probability that all of the at most $\lfloor 4 u_\star/c_5 \rfloor$
  $Y^i$'s intersecting a given  $\zeta \in \boldsymbol{C}$ have an index
  $i$ with $R_i = 1$ is at least $q^{\lfloor 4 u_\star/c_5 \rfloor}$. So,
  for some $c_6 = c_6(\good,\gap,u)$, $c_7 >0$,
  \begin{equation}
    \label{e:shalom}      E^{\mathcal Q}[g] \geq  \tfrac{1}{2} c_5
    q^{\lfloor 4 u_\star/c_5 \rfloor} n =: 2c_{6} n^{1-c_7 \delta}.
  \end{equation}
  Changing one segment $Y^i$ can change the value of $g$ by at most
  $2n^\gamma $ on $\mathcal S$. Therefore, by Lemma~1.2 in \cite{M89},
  \begin{equation}
    \label{eq:Eg}
    {\mathcal Q}[ g < c_{6}n^{1-c_7 \delta} ]
    \leq 2
    \exp \Big\{ \dfrac{- c_{6}^2 n^{2(1-c_7 \delta)}}{n^{1-\piece} n^{2\piece}} \Big\}.
  \end{equation}
  If $g\ge c_6 n^{1-c_7 \delta}$, then there are at least
  $c_6 n^{1-c_7 \delta}$ disjoint connections in $\boldsymbol C$ linking
  $F$ and $F'$ and only using vertices in $\supp \xi^{sp}$. In accordance
  with \eqref{e:Unew}, we now choose $\delta$ such that
  $2c_7 \delta < \gamma$. Then, since on $\mathcal S$ the total length of
  the bridges $\sum_{i,j}N^{i,j} \leq \ln^5 n \cdot n^{1-\piece} $, the
  events $\{g \geq c_{6}n^{1-c_7 \delta}\}$ and $\mathcal S$ imply that
  at least one connection in $\boldsymbol C$ is contained in
  $\supp \xi^{sp}$, for $n\ge c_{u, \delta}$. Hence, \eqref{eq:Eg}
  implies that  uniformly on the event $\mathcal S$,
  \begin{equation}
    \label{e:spbd}  \sup_{F,F'\subset V:|F|,|F'|\ge 10 \rho n}
    \mathcal Q
    [F\overset{sp}\nleftrightarrow F'] \le
    2 \exp( - c_{6}^2n^{1-\piece- 2c_7 \delta } ).
  \end{equation}
  Inserting this estimate into \eqref{e:PKS}, noting that
  \begin{align}
    1- \gamma - 2c_7 \delta > 1 - 2 \gamma
    \stackrel{\eqref{e:newgamma}}{=} 1 - v_{u'(1+\epsilon)} \beta,
  \end{align}
  we find that, uniformly on ${\mathcal K} \cap {\mathcal S}$,
  \begin{align}
    \label{e:QPbd} \mathcal Q[\mathcal P] \le  \exp(-c_{u, \delta} n^{1-2\gamma}).
  \end{align}
  Using this estimate, together with the bound \eqref{eq:kappasmall} on
  $\Plr[{\mathcal K}^c]$ and the bound \eqref{e:Sc} on
  $\Plr[{\mathcal S}^c]$, in \eqref{e:PK}, we find
  \eqref{eq:reductusmall} for $n\ge 1$ by possibly adjusting the
  constants. This concludes the proof of Theorem~\ref{th:usmall}.
\end{proof}
\section{Uniqueness of the giant component}
\label{s:uniq}

This section contains the proof of Theorem~\ref{t:secondcomponent}, that is of
the uniqueness of the giant component. More precisely, we show that for any
choice of $\kappa >0$ and $u<u_\star$, with a high probability, the second
largest component of the vacant set $\mathcal V_n^u$ is smaller than $\kappa n$.
The sprinkling is again the major ingredient of the proof. This time, however,
we will really use the fact that $u_n-u< n^{-\delta }$.

Heuristically, our argument runs as follows. We will show that any component of
$\mathcal V_n^u$  of size at least $\kappa n$ should contain at least
$\kappa n/2$ vertices that were included in clusters of size at least
$n^{v_u\beta /2}$ of the vacant set left by segments at level $u_n$. Hence, in
order to have $|\comp_\Sec|\ge \kappa n$, there should be two groups of
such vertices which do not get connected after the sprinkling. A small
extension of the proof of the last section then shows that this happens with a
small probability.

\begin{proof}[Proof of Theorem~\ref{t:secondcomponent}]
  We choose $\gamma $, $\beta $ as in \eqref{e:newgamma} and
  recall from \eqref{e:Unew} the notation
  $u_n=(u+ n^{-\delta }) \wedge ((u+u_\star)/2)$, where $2 \delta < 1 -\gamma$.
  By decreasing $\delta$, we can also assume that $\delta<\gamma$. Recall also
  that the vacant set left by segments $\xi_{u_n}$ was defined in \eqref{eq:xi}
  (with $u_n$ replaced by $u$).

  We divide the vertices of $G$ into three sets. The vertex $x\in V$ is
  called \emph{small}, if
  \begin{equation}
    \tx(B(x,5\beta \ld n))=0, \quad |\comp_x(\xi_{u_n})|\le \ld^2 n,
    \quad\text{and}\quad
    \comp_x(\xi_{u_n})\subset B(x,\beta \ld n).
  \end{equation}
  It is called \emph{proper} (cf. Definition~\ref{def:proper}), if
  \begin{equation}
    \label{e:unz}
    \tx(B(x,5\beta \ld n))=0
    \quad\text{and}\quad
    \text{(i), (ii) of \eqref{eq:proper} hold with $h=c_3$ of
      Lemma~\ref{lm:proper}}.
  \end{equation}
  It is called \emph{bad} otherwise. We will use $\mathcal B$ to denote the set
  of bad vertices.

  The next lemma shows that the set $\mathcal B$ is small. The lemma should
  be viewed as an analogue to a non-existence of intermediate components
  in the Bernoulli percolation case.
  \begin{lemma}
    \label{l:bad}
    There exists a function $g(n)$ such that $\lim_{n\to\infty} g(n)/n=0$ and
    \begin{equation}
      \label{e:bad} \lim_{n\to\infty}
      \Plr[|\mathcal B|\ge g(n)]=0.
    \end{equation}
  \end{lemma}

  We postpone the proof of the Lemma~\ref{l:bad}  and proceed with the proof of
  Theorem~\ref{t:secondcomponent}. First, we add long-range bridges to the
  configuration $\xi_{u_n}$, that is we define $\xi'_{u_n}$ as in \eqref{eq:xis}.
  We must be careful to see that these bridges do not destroy the components of
  proper vertices in $\supp \xi_{u'}$. To this end we collect a family $\Gamma $
  of strings (that is of self-avoiding paths of length
    $\lfloor \beta \ld n\rfloor$) as in the proof of Lemma~\ref{lm:lotstrings}.
  This collection contains  $\lfloor c_3 n^{\beta v_+} \rfloor$ distinct strings
  starting at $x$ for any proper vertex $x$, as before. In particular, this
  implies that $|\Gamma |\le c_3 n^{1+\beta v_+}$. We can show, as below
  \eqref{eq:congestion}, that the number $S_y$ of strings intersecting a given
  $y\in V$ satisfies $S_y\le c_u (\log^{c_u} n) n^{\tfrac 54 v_-\beta }$. Let now
  $\mathcal S$ be the event that all segments and bridges are not too long, as in
  \eqref{e:calS}.  On $\mathcal S$, the number of strings that are intersected by
  bridges is thus at most $c_u (\log^{c'_u} n) n^{1-\gamma+ \tfrac 54 v_-\beta } $.
  We declare a proper vertex \textit{bad proper}, it at least half of the strings
  starting at this vertex is intersected by the bridges. Obviously, on
  $\mathcal S$, the set $\mathcal {BP}$ of bad proper vertices must satisfy
  $|\mathcal {BP} | \tfrac {c_3}{2} n^{\beta v_+}\le c_u (\log^{c'_u} n) n^{1-\gamma+ \tfrac 54 v_-\beta }.$
  With \eqref{e:epsv} and \eqref{e:newgamma}, it follows that
  \begin{equation}
    \label{e:BP}  |\mathcal {BP}|=o(n), \qquad\text{as $n\to\infty$, on $\mathcal S$.}
  \end{equation}
  On the other hand, the remaining proper vertices, that we call \textit{large},
  are starting vertices of at least $\tfrac12 n^{v_+\beta}$ strings which are not
  intersected by the bridges. Hence, if $x$ is large, it is contained in a
  component of $\xi '_{u_n}$ of size at least $\tfrac12 n^{v_+\beta}$. This
  implies that
  \begin{equation}
    \label{e:unxnew}
    \parbox{0.8\textwidth}{on $\mathcal S$, the number of components of $\supp \xi'_{{u_n}}$ that
      contain a large vertex is at most
      $cn^{1-\beta v_+}$. }
  \end{equation}

  We now perform the sprinkling as in Section~\ref{s:super}. Recall that on the
  probability space $(\{0,1\}^{M_{u_n}},\mathcal Q)$ we have defined
  i.i.d.~random variables $(R_k)_{0\le k< M_{u_n}}$ with success probability
  $q=n^{-2\delta}$ (cf.~\eqref{e:q}), the number of remaining segments $I$
  (cf.~\eqref{e:I}) and the good event $\mathcal G$ (cf.~\eqref{e:calG}). We
  have then constructed the sprinkled configuration $\xi^{sp}$
  (cf.~\eqref{eq:xisp}). Using  Lemma~\ref{l:partsbridges}, then the estimate
  \eqref{e:Qm-}, as $n\to\infty$,
  \begin{equation}
    \begin{split}
      \label{e:una}
      &P[|\comp_\Sec^u| \ge \kappa n]
      \le Q^{un}[|\comp_\Sec^u| \ge \kappa n]+ o(1)
      \\&\le \Plr \otimes \mathcal Q[|\comp_\Sec(\xi^{sp})|
        \ge \kappa n, {\mathcal G} , \mathcal S , |{\mathcal B}| \leq g(n)]
      + \mathcal P[\mathcal G^c] + \Plr [\mathcal S^c]
      + \Plr [|{\mathcal B}| > g(n)] + o(1) \\
      &\stackrel{\eqref{e:Gbd}, \eqref{e:Sc}, \eqref{e:bad}}{\leq} \Plr \otimes \mathcal Q[|\comp_\Sec(\xi^{sp})|\ge \kappa n,
        {\mathcal G} , \mathcal S , \{|{\mathcal B}| \leq g(n)\}] + o(1).
    \end{split}
  \end{equation}

  In order to estimate the term on the right-hand side, we claim that for any
  $\kappa>0$,
  \begin{equation}
    \label{e:large} \begin{array}{l}
      \text{on $\mathcal S\cap \mathcal G \cap \{|{\mathcal B}| \leq g(n)\}$,
        any component of $\supp \xi^{sp}$ of size $\geq \kappa n$ }\\
      \text{contains at least $\kappa n/2$ large vertices, for $n \geq c_{\kappa,u,\delta}$,}
    \end{array}
  \end{equation}
  Indeed, recall that we have divided the vertices in $\supp \xi'_{u_n}$ into
  small, bad, bad proper and large vertices. The vertices in $\supp \xi^{sp}$
  consist of these four sets, and the set
  $\{x\in V: \xi'_{u_n}(x)=0, \xi^{sp}(x)=1\}$. Let us call all the vertices in
  this last set \emph{sprinkled} vertices. Suppose now that the event
  $\mathcal S\cap \mathcal G \cap \{|{\mathcal B}| \leq g(n)\}$ occurs. Then by
  definition of $\mathcal S$ and $\mathcal G$, the number of sprinkled vertices
  is at most
  \begin{equation}
    (M_{u_n}-I) 2 n^{\gamma } + cn^{1-\gamma }\ln^5 n \le (M_{u_n} -M_u)
    2 n^{\gamma } + cn^{1-\gamma }\ln^5 n \leq  c_{u, \delta} n^{1-\delta },
  \end{equation}
  because $M_{u_n}-M_u \leq c_{u, \delta} n^{1-\piece - \delta}$ and
  $\delta < \gamma$. Consider now any component $A$ of size $\kappa n$ of
  $\supp \xi^{sp}$. Then the number of vertices in $A$ that are either bad, bad
  proper or sprinkled is at most
  $|{\mathcal B}| + |\mathcal{BP}| + c_{u, \delta} n^{1-\delta} = o(n)$, by
  definition of $g(n)$ and \eqref{e:BP}, the remaining vertices being either
  small or large. By definition of small vertex, all small vertices in
  $\supp \xi'_{u_n}$ belong to components of size at most $\ld^2 n$. Any of the
  at most $c_{u, \delta} n^{1-\delta}$ sprinkled vertices can merge at most
  $(d-1)$ such components, so the maximum number of small vertices belonging to
  the same component of $\supp \xi^{sp}$ is bounded by
  $c_{u, \delta} (\ld^2 n) n^{1-\delta},$ which is less than $o(n)$, too.
  Hence, for $n \geq c_{\kappa, u, \delta}$, at least $\kappa n/2$ of the
  vertices in $A$ are large, proving \eqref{e:large}.

  Let $\mathcal L$ be the set of large vertices.
  Defining $\mathcal P'$ to be the event
  \begin{equation}
    \label{eq:Pprime}
    \mathcal P'=\{\text{there are $A$, $B\subset \mathcal L$
        such that } |A|\ge \kappa n/2, |B|\ge \kappa n/2, \text{
        and } A\overset{sp}\nleftrightarrow B\},
  \end{equation}
  we obtain from \eqref{e:large}, for $n \geq c_{\kappa,u,\delta}$,
  \begin{equation}
    \label{e:una+}  \Plr \otimes \mathcal Q \big[|\comp_\Sec(\xi^{sp})|\ge \kappa n,
      \mathcal G,\mathcal S, \{|{\mathcal B}| \leq g(n)\}\big]\le
    \Plr \otimes \mathcal Q[\mathcal P',\mathcal S]=
    \Elr [1_{\mathcal S }\mathcal Q[\mathcal P']],
  \end{equation}
  which is essentially equivalent to the right-hand side of \eqref{e:PK}. Note
  also that since $\supp \xi'_{u_n } \subseteq \supp \xi^{sp}$, $\mathcal{P}'$
  equals
  \begin{equation*}
    \{\text{there are $A$, $B\subset \mathcal L$
        such that } |A|, |B|\ge \kappa n/2, \text{
        and } \cup_{a \in A} \comp_a(\xi_{u_n}')
      \overset{sp}\nleftrightarrow  \cup_{b \in B} \comp_b(\xi_{u_n}')\},
  \end{equation*}

  We should now bound the number of possible choices for the unions in the
  equation above. By \eqref{e:unxnew}, there are at most $n^{1-\beta v_+}$ sets
  of the form $\comp_x(\xi_{u_n}')$ with $x \in \mathcal{L}$, so there
  are at most $2^{2n^{1-\beta v_+} }$ choices for the unions in the equation
  above. Hence,
  \begin{equation}
    \label{e:unb}
    \Elr [1_{\mathcal S }\mathcal Q[\mathcal P']]
    \le 2^{2n^{1-\beta v_+} }
    \sup_{\mathcal S}
    \sup_{F,F'\subset V:|F|,|F'|\ge \kappa n/2}
    \mathcal Q[F\overset{sp}\nleftrightarrow F'].
  \end{equation}
  Repeating the argument from \eqref{e:PKS} to \eqref{e:spbd} and choosing
  $\delta $ small enough, we infer that the right-hand side of \eqref{e:unb}
  tends to zero as $n$ tends to infinity. With \eqref{e:una} and
  \eqref{e:una+}, this completes the proof of
  Theorem~\ref{t:secondcomponent}.
\end{proof}

We now prove the lemma we used in previous proof.
Due to Proposition~\ref{pr:dominat}, this proof will be
reduced to estimates on a branching process.

\begin{proof}[Proof of Lemma~\ref{l:bad}]
  Let $x\in V$ be an arbitrary vertex and let $B_x=B(x,\beta \ld n)$,
  $B_x'=B(x,5\beta \ld n)$. We write $v_+=v_{u(1+\varepsilon )}$,
  $v_-=v_{u(1-\varepsilon )}$, see \eqref{eq:muvu} for the notation. It is easy
  to see that $\mathcal B\subset \mathcal B_1\cup\dots\cup \mathcal B_4$, where
  \begin{equation}
    \begin{split}
      \mathcal B_1&=\{x\in V:\tx(B'_x)>0\},\\
      \mathcal B_2&=\{x\in V:\tx(B'_x)=0,
        |\comp_x^{\beta \ld n}(\xi_{u_n})| < c_3 n^{v_+ \beta} \text{ and }
        \comp_x(\xi_{u_n})\not\subset B_x\},\\
      \mathcal B_3&=\{x\in V:\tx(B'_x)=0,
        |\comp_x^{\beta \ld n}(\xi_{u_n})| < c_3 n^{v_+ \beta} \text{ and }
        |\comp_x(\xi_{u_n})| > \ld^2 n\},\\
      \mathcal B_4&=\{x\in V:\tx(B'_x)=0, \text{(ii) of \eqref{eq:proper}
          does not hold with $h=c_3$}\}.
    \end{split}
  \end{equation}

  By Lemma~\ref{l:txzero} we have $|\mathcal B_1|\le g(n)$ for a sequence $g$
  which decays as in the statement, deterministically. Further, by
  \eqref{eq:branch}, we have $\Plr [x\in \mathcal B_4]\le c n^{-\beta }$.
  Therefore, using the Markov inequality, there is a sequence $g(n)$ such that
  $g(n)/n$ tends to zero, such that
  $\lim_{n\to\infty} \Plr [|\mathcal B_4|\ge g(n)]=0$.

  It remains to control $\mathcal B_2$, $\mathcal B_3$. We choose any
  $\epsilon>0$ small enough such that \eqref{eq:epsilon} is satisfied, noting
  that these two constraints allow us to make $\epsilon>0$ even smaller. For
  $n \geq c_{\delta,u}$, we then have $u_n \in (u,u(1+\epsilon/4))$. Hence, the
  random sets $\comp^{u(1 \pm \epsilon/2)}$ constructed in
  Proposition~\ref{pr:dominat} (with $\epsilon$ replaced by $\epsilon/2$)
  dominate the component $\comp_y(1_B \cdot \xi_{u_n})$ from above and
  from below with probability at least $1 - c_{u,\delta, \epsilon}n^{-2 \beta},$
  for $n \geq c_{u, \delta, \epsilon}$. Recall also that
  $\comp^{u(1 \pm \epsilon/2)}$ are distributed as $\comp_o$ under
  $\Ber_{u(1 \pm \epsilon/2)}$. Let $Z^-_k$ be the branching process description
  of $\comp^{u(1+\epsilon/2)}$, that is
  $Z^-_k=|\{y\in \comp^{u(1+\epsilon/2)}, \dist(x,y)=k\}|$. Similarly,
  let $Z^+_k$ be such a description of $\comp^{u(1-\epsilon/2)}$. By our
  choice of parameters and Lemma \ref{l:interbr}, both $Z^+$ and $Z^-$
  are supercritical branching processes. We use $T_+$, $T_-$ to denote their
  extinction times, $\phi_+$, $\phi_-$ their offspring generating function, and
  $q_+$, $q_-$ their extinction probabilities. Observe that
  \begin{equation}
    \label{e:branchap}
    \lim_{\epsilon \to 0} q_--q_+=0.
  \end{equation}
  Set $r=\lfloor \beta \ld n \rfloor$, $N_n=c_3 n^{\beta v_+}$. Using
  Proposition~\ref{pr:dominat}, we get
  \begin{align}
    \label{e:unaa}
    \Plr [x\in \mathcal B_2] - c_{u, \delta, \epsilon} n^{-2\beta } &\le
    \Plr [Z^-_r\le N_n, Z^+_r\ge 1 ]\\
    &\le
    \Plr [1\le Z^-_r\le N_n] + \Plr [0=Z^-_r< Z^+_r]. \nonumber
  \end{align}
  Since $v_.$ is strictly decreasing, $N_n/ n^{\beta v_{u(1+\epsilon/2)}}\to 0$.
  Hence, the first term on the right-hand side is the probability that the
  branching process is not extinct at generation $r$, but is much smaller than
  its typical size $n^{\beta v_{u(1+\epsilon/2)}}$. This probability tends to
  $0$ as $n\to\infty$, using e.g.~Theorems 6.1, 6.2 in Chapter I, p.~9, of
  \cite{AN}. Using the fact that the generating function of $Z^\pm_r$ is the $r$-th
  iteration $\phi_\pm^{(r)}$ of $\phi $, we get that
  \begin{equation}
    \limsup_{n\to\infty}\Plr [0=Z^-_r< Z^+_r]
    =\limsup_{r\to \infty}\phi_-^{(r)}(0)-\phi_+^{(r)}(0)=q_--q_+.
  \end{equation}
  Here we have used that $q_{\pm}$ is the attractive fixed-point of
  $\phi_{\pm} $. Using \eqref{e:branchap}, this can be made arbitrarily small
  by choosing $\epsilon$ small. Inserting this back into \eqref{e:unaa}, we get
  $\lim_{n\to \infty}{\Plr[x\in \mathcal B_2]= 0}$. This implies, using
  the Markov inequality again, that $\lim_{n\to\infty}$
  $\Plr [|\mathcal B_2|\ge g(n)]= 0$ for some $g$ as in the statement.

  Similarly we have,
  \begin{equation}
    \begin{split}
      \label{e:aafdaf}
      \Plr &[x\in \mathcal B_3] - c n^{-2\beta }\le
      \Plr \Big[Z^-_r \le N_n, \sum_{k=0}^\infty Z^+_k\ge \ld^2 n\Big]
      \\&\le
      \Plr [Z^-_r\le N_n, T_-=\infty] +
      \Plr\Big[\sum_{k=0}^\infty Z^+_k\ge \ld^2 n, T_-<\infty\Big]
      \\&\le
      \Plr [1\le Z^-_r\le N_n] +
      \Plr \Big[\sum_{k=0}^\infty Z^+_k\ge \ld^2 n, T_+<\infty\Big]+
      \Plr [T_-< T_+=\infty].
    \end{split}
  \end{equation}
  The first probability on the right-hand side tends to $0$, as in the previous
  argument. By \cite{AN}, Theorem 12.3 in Chapter I, p.~52, conditioned on
  $T_+<\infty$, $Z^+$ has the law of a sub-critical branching process. Using
  this claim it is easy to show that the second probability in \eqref{e:aafdaf}
  tends to zero.  The third probability can be made arbitrarily small by
  choosing $\varepsilon$ small, by using \eqref{e:branchap} again. This then
  implies that $\lim_{n\to\infty}\Plr[|\mathcal B_3|\ge g(n)]=0$ for an
  appropriately chosen $g$, as in the previous case. This completes the proof
  of the lemma.
\end{proof}



\begin{thebibliography}{BCvdH{\etalchar{+}}05}

\bibitem[AB93]{AB}
David~J. Aldous and Mark Brown, \emph{Inequalities for rare events in
  time-reversible {M}arkov chains. {II}}, Stochastic Process. Appl. \textbf{44}
  (1993), no.~1, 15--25. \MR{MR1198660}

\bibitem[AF]{AF}
David~J. Aldous and James~A. Fill, \emph{Reversible markov chains and random
  walks on graphs}, {\tt
  http://www.stat.berkeley.edu/$\sim$aldous/RWG/book.html}.

\bibitem[ABS04]{ABS04}
Noga Alon, Itai Benjamini, and Alan Stacey, \emph{Percolation on finite graphs
  and isoperimetric inequalities}, Ann. Probab. \textbf{32} (2004), no.~3A,
  1727--1745. \MR{MR2073175}

\bibitem[Ath94]{KA}
K.~B. Athreya, \emph{Large deviation rates for branching processes. {I}.
  {S}ingle type case}, Ann. Appl. Probab. \textbf{4} (1994), no.~3, 779--790.
  \MR{MR1284985}

\bibitem[AN72]{AN}
Krishna~B. Athreya and Peter~E. Ney, \emph{Branching processes},
  Springer-Verlag, New York, 1972, Die Grundlehren der mathematischen
  Wissenschaften, Band 196. \MR{MR0373040}

\bibitem[BS08]{BS08}
Itai Benjamini and Alain-Sol Sznitman, \emph{Giant component and vacant set for
  random walk on a discrete torus}, J. Eur. Math. Soc. (JEMS) \textbf{10}
  (2008), no.~1, 133--172. \MR{MR2349899}

\bibitem[BCvdH{\etalchar{+}}05]{BCvdHSS05a}
Christian Borgs, Jennifer~T. Chayes, Remco van~der Hofstad, Gordon Slade, and
  Joel Spencer, \emph{Random subgraphs of finite graphs. {I}. {T}he scaling
  window under the triangle condition}, Random Structures Algorithms
  \textbf{27} (2005), no.~2, 137--184. \MR{MR2155704}

\bibitem[BS87]{BS87}
A.~Broder and E.~Shamir, \emph{On the second eigenvalue of random regular
  graphs}, 28th Annual Symposium on Foundations of Computer Science (Washington
  DC), IEEE Comput. Soc. Press, 1987, pp.~286--294.

\bibitem[DS06]{DS06}
Amir Dembo and Alain-Sol Sznitman, \emph{On the disconnection of a discrete
  cylinder by a random walk}, Probab. Theory Related Fields \textbf{136}
  (2006), no.~2, 321--340. \MR{MR2240791}

\bibitem[Dur96]{D05}
Richard Durrett, \emph{Probability: {T}heory and {E}xamples}, second ed.,
  Duxbury Press, Belmont, CA, 1996. \MR{MR1609153}

\bibitem[ER60]{ER60}
P.~Erd{\H{o}}s and A.~R{\'e}nyi, \emph{On the evolution of random graphs},
  Magyar Tud. Akad. Mat. Kutat\'o Int. K\"ozl. \textbf{5} (1960), 17--61.
  \MR{MR0125031}

\bibitem[Fri91]{Fri91}
Joel Friedman, \emph{On the second eigenvalue and random walks in random
  {$d$}-regular graphs}, Combinatorica \textbf{11} (1991), no.~4, 331--362.
  \MR{1137767}

\bibitem[Fri08]{Fri08}
Joel Friedman, \emph{A proof of {A}lon's second eigenvalue conjecture and
  related problems}, Mem. Amer. Math. Soc. \textbf{195} (2008), no.~910,
  viii+100. \MR{MR2437174}

\bibitem[LS08]{LS09}
Eyal Lubetzky and Allan Sly, \emph{Cutoff phenomena for random walks on random
  regular graphs}, arXiv 0812.0060, 2008.

\bibitem[LPS88]{LPS88}
A.~Lubotzky, R.~Phillips, and P.~Sarnak, \emph{Ramanujan graphs}, Combinatorica
  \textbf{8} (1988), no.~3, 261--277. \MR{MR963118}

\bibitem[McD89]{M89}
Colin McDiarmid, \emph{On the method of bounded differences}, Surveys in
  combinatorics, 1989 ({N}orwich, 1989), London Math. Soc. Lecture Note Ser.,
  vol. 141, Cambridge Univ. Press, Cambridge, 1989, pp.~148--188.
  \MR{MR1036755}

\bibitem[NP09]{NP07}
Asaf Nachmias and Yuval Peres, \emph{Critical percolation on random regular
  graphs}, arXiv 0707.2839, to appear in Random Structures Algorithms, 2009.

\bibitem[Pit08]{Pit08}
Boris Pittel, \emph{Edge percolation on a random regular graph of low degree},
  Ann. Probab. \textbf{36} (2008), no.~4, 1359--1389. \MR{MR2435852}

\bibitem[SC97]{SC97}
Laurent Saloff-Coste, \emph{Lectures on finite {M}arkov chains}, Lectures on
  probability theory and statistics ({S}aint-{F}lour, 1996), Lecture Notes in
  Math., vol. 1665, Springer, Berlin, 1997, pp.~301--413. \MR{MR1490046}

\bibitem[SS09]{SS09}
Vladas Sidoravicius and Alain-Sol Sznitman, \emph{Percolation for the vacant
  set of random interlacements}, Comm. Pure Appl. Math. \textbf{62} (2009),
  no.~6, 831--858. \MR{MR2512613}

\bibitem[Szn09a]{Szn09c}
Alain-Sol Sznitman, \emph{A lower bound on the critical parameter of
  interlacement percolation in high dimension}, preprint, 2009.

\bibitem[Szn09b]{Szn09d}
Alain-Sol Sznitman, \emph{On the domination of random walk on a discrete
  cylinder by random interlacements}, Electron. J. Probab. \textbf{14} (2009),
  no. 56, 1670--1704. \MR{MR2525107}

\bibitem[Szn09c]{Szn09b}
Alain-Sol Sznitman, \emph{Random walks on discrete cylinders and random
  interlacements}, Probab. Theory Related Fields \textbf{145} (2009), no.~1-2,
  143--174. \MR{MR2520124}

\bibitem[Szn09d]{Szn09e}
Alain-Sol Sznitman, \emph{Upper bound on the disconnection time of discrete
  cylinders and random interlacements}, Ann. Probab. \textbf{37} (2009), no.~5,
  1715--1746. \MR{MR2561432}

\bibitem[Szn09e]{Szn09}
Alain-Sol Sznitman, \emph{Vacant set of random interlacements and percolation},
  to appear in Annals of Mathematics, 2009.

\bibitem[Tei09]{Tei09}
A.~Teixeira, \emph{Interlacement percolation on transient weighted graphs},
  preprint, 2009.

\bibitem[TW10]{TW10}
Augusto Teixeira and David Windisch, \emph{On the fragmentation of a torus by
  random walk}, preprint available at http://arxiv.org/abs/1007.0902, 2010.

\bibitem[Win08]{Win08}
David Windisch, \emph{Random walk on a discrete torus and random
  interlacements}, Electron. Commun. Probab. \textbf{13} (2008), 140--150.
  \MR{MR2386070}

\end{thebibliography}

\newcommand{\etalchar}[1]{$^{#1}$}
\def\cprime{$'$}
\providecommand{\bysame}{\leavevmode\hbox to3em{\hrulefill}\thinspace}
\renewcommand\MR[1]{\relax\ifhmode\unskip\spacefactor3000
\space\fi \MRhref{#1}{#1}}
\renewcommand{\MRhref}[2]%
{\href{http://www.ams.org/mathscinet-getitem?mr=#1}{#2}}
\providecommand{\href}[2]{#2}


\end{document}